\documentclass[aos, preprint]{imsart}%

\RequirePackage{natbib}

\arxiv{math.PR/0000000}

\usepackage{amsfonts}
\usepackage{epsfig}
\usepackage{amssymb,amsmath,bm}
\usepackage{subfig}
\usepackage{amsthm}
\usepackage{mathrsfs}
\usepackage{enumerate,color}
\usepackage{amsmath}
\usepackage{amssymb}
\usepackage{graphicx}%
\providecommand{\U}[1]{\protect\rule{.1in}{.1in}}
\newtheorem{rem}{Remark}
\newtheorem{theo}{Theorem}
\newtheorem{lemma}{Lemma}
\newtheorem{example}{Example}
\newtheorem{algorithm}{Algorithm}
\newtheorem{procedure}{Procedure}
\newtheorem{corollary}{Corollary}

\newenvironment{hypA}[1]{\begin{itemize}
\item[({\bf A#1})]}{\end{itemize}}
\newcounter{hypB}

\begin{document}

\begin{frontmatter}

\title{Parameter Estimation for Hidden Markov Models with Intractable Likelihoods}
\runtitle{Estimation of Intractable HMMs}

\begin{aug}
  \author{\fnms{Thomas. A.}  \snm{Dea$\textrm{n}^{1,}$}\corref{}\thanksref{t2}\ead[label=e1]{tad36@cam.ac.uk}},
  \author{\fnms{Sumeetpal S.} \snm{Sing$\textrm{h}^{1,}$}\thanksref{t2}\ead[label=e2]{sss40@cam.ac.uk}},
  \author{\fnms{Ajay} \snm{Jasra}\ead[label=e3]{ajay.jasra@ic.ac.uk}}
  \and
  \author{\fnms{Gareth w.}  \snm{Peters} \ead[label=e4]{garethpeters@unsw.edu.au}}

  \thankstext{t2}{T.A. Dean and S.S. Singh's research is funded by the Engineering and Physical Sciences Research Council (EP/G037590/1) whose support is gratefully acknowledged. 
This}

  \runauthor{Dean et al.}

  \affiliation{University of Cambridg$\textrm{e}^1$, Imperial College London and University of New South Wales}

  \address{T.A.Dean, S.S. Singh, \\
  Department of Engineering, \\
  University of Cambridge, \\
  Cambridge, \\
  CB2 1PZ, UK\\
          \printead{e1,e2}}

  \address{A. Jasra, \\
  Department of Mathematics, \\
  Imperial College London, \\
  London, \\
  SW7 2AZ, UK \\
          \printead{e3}}

   \address{G.W. Peters, \\
  School of Mathematics and Statistics, \\
  University of New South Wales, \\
  Sydney, \\
  2052, Aus \\
          \printead{e4}}

\end{aug}

\begin{abstract}
Approximate Bayesian computation (ABC) is a popular technique for approximating likelihoods and is often used in parameter estimation when the likelihood functions are analytically intractable.  Although
the use of ABC is widespread in many fields, there has been little investigation of the theoretical properties of the resulting estimators.  In this paper we give a
theoretical analysis of the asymptotic properties of ABC based maximum likelihood parameter estimation for hidden Markov models.  In particular, we derive results analogous to those
of consistency and asymptotic normality for standard maximum likelihood estimation.
We also discuss how Sequential Monte Carlo methods provide a natural
method for implementing likelihood based ABC procedures.
\end{abstract}



\end{frontmatter}

\bigskip

%\begin{AMS}
%82C99, 65C30
%\end{AMS}

%\pagestyle{myheadings}
\thispagestyle{plain}

\section{Introduction}

\let\thefootnote\relax\footnotetext{First version: Cambridge University Engineering Department Technical Report 660, 1 October 2010} 

The hidden Markov model (HMM) is an important statistical model in many fields including
Bioinformatics (e.g.~\cite{dureddkromit1998}), Econometrics (e.g.~\cite{kimshechi1998}) and Population genetics (e.g.~\cite{felchu1996});
see also \cite{caprydmou2005} for a recent overview.   Often one has a range of HMMs parameterised by a parameter vector $\theta$ taking values in some compact subset $\Theta$ of Euclidian space.  Given a sequence of
observations $\hat{Y}_{1} , \ldots, \hat{Y}_{n}$ the objective is to find the parameter
vector $\theta^{\ast} \in\Theta$ that corresponds to the particular HMM from which the
data were generated.

A common approach to estimating $\theta^{\ast}$ is maximum likelihood
estimation (MLE). The parameter estimate, denoted $\hat{\theta}_{n}$, is obtained via
maximizing the log-likelihood of the observations:
\[
\hat{\theta}_{n} = {\arg\max}_{\theta\in\Theta} l_{n}(\theta)
\]
where
\[
\label{eq:ll}l_{n}(\theta) :=  \log p_{\theta} \left(  \hat{Y}_{1} ,
\ldots, \hat{Y}_{n} \right)  = \sum\limits_{i=1}^{n}\log p_{\theta}( \hat{Y}_{i}
\vert \hat{Y}_{1} , \ldots, \hat{Y}_{i-1}) .
\]
Unless the model is simple, e.g. linear Gaussian or when $\mathcal{X}$ is a finite set, one can seldom evaluate the likelihood analytically. There
are a variety of techniques, for example sequential Monte Carlo (SMC), for numerically estimating the likelihood. However, in a wide range of applications these
methods cannot be used, for example when the conditional
density of the observed state of the HMM given the hidden state is intractable, by which we mean that this density 
cannot be evaluated analytically and has no unbiased Monte Carlo estimator.  Despite this, one is often
still able to generate samples from the corresponding processes for different values of the parameter $\theta$ (e.g.~\cite{jassinmarmcc2010}).
This has led to the development of methods in which $\theta^{\ast}$ is
estimated by taking the value of $\theta$ which maximizes some principled
approximation of the likelihood which is itself
estimated using Monte Carlo simulation.

One such approach is the convolution particle filter of \cite{camros2009}. Another technique which can be applied to this class of problems is
indirect inference; see \cite{goumonren1993} and \cite{hegfri2004}. However in the context of HMMs, when one does not adopt a linear
Gaussian approximation of the filtering density (which can be very inaccurate,
as in extended Kalman filter approximations), this method is likely to be very
expensive. A third method which has recently
received a great deal of attention is approximate Bayesian computation (ABC).  A non-exhaustive list of references includes:
\cite{mckcoodea2009,petwutshe2010,priseiperfel1999,ratandwiuric2009,tavbalgridon1997}. See also \cite{sisfan2010} for a review
on computational methodology.

In the standard ABC approach (omitting for the moment the possible use of summary statistics) one assumes that a data set $\hat{Y}_{1} , \ldots , \hat{Y}_{n}$ is
given and approximates the likelihood function $p_{\theta} \left( \hat{Y}_{1} , \ldots , \hat{Y}_{n} \right)  $ via probabilities of the form
\begin{equation}
\label{nonHMMABCdef}\mathbb{P}_{\theta} \left(  d \left(  Y_{1} , \ldots,
Y_{n} ; \hat{Y}_{1} , \ldots , \hat{Y}_{n} \right)  \leq\epsilon\right)
\end{equation}
where $\left\{ Y_{k} \right\}_{k \geq 1}$ denotes the observed state of the HMM, $d( \cdot; \cdot)$ is some suitable metric on the n-fold product space
$\mathbb{R}^{m} \times\cdots\times\mathbb{R}^{m}$ and $\epsilon> 0$ is a
constant which reflects the accuracy of the approximation.   In practice these probabilities are themselves estimated
using Monte Carlo techniques.  

The intuitive justification for the ABC approximation is that for sufficiently small
$\epsilon$
\[
\frac{ \mathbb{P}_{\theta} \left(  d \left(  Y_{1} , \ldots, Y_{n} ; \hat{Y}_{1} , \ldots , \hat{Y}_{n} \right)  \leq\epsilon\right)  }{ V^{\epsilon}_{\hat{Y}_{1} , \ldots , \hat{Y}_{n}} }
\approx  p_{\theta} \left(  \hat{Y}_{1} , \ldots , \hat{Y}_{n} \right)
\]
where $V^{\epsilon}_{\hat{Y}_{1} , \ldots , \hat{Y}_{n}}$ denotes the volume of the $d$-ball of radius $\epsilon$ around the points $\hat{Y}_{1} , \ldots , \hat{Y}_{n}$.  Thus the probabilities \eqref{nonHMMABCdef} will provide a good
approximation to the likelihood, up to the value of some
renormalising factor which is independent of $\theta$ and hence can be ignored.  However in general  it is not at all clear in what sense an approximation to the likelihood must be `good' in order for the resulting inference procedures to be well behaved.  The purpose of this paper is to resolve this issue by directly investigating the effect of the parameter $\epsilon$, not on the quality of the approximations \eqref{nonHMMABCdef}, but on the behaviour of the resulting ABC based parameter estimators.

We note that in \eqref{nonHMMABCdef} we have implicitly assumed that one is working with the entire data set rather than a summary statistic of it as is usually done in practice, especially when the observations $\left\{ Y_{k} \right\}_{k \geq 0}$ take values in some high dimensional space.  For ease of exposition we shall persist with this assumption throughout the rest of the paper, noting where appropriate the conditions under which the results we derive will continue to hold when summary statistics are used (see in particular the remarks at the ends of Sections \ref{sec:standABC} and \ref{sec:noisyABC}).

\subsection{Contribution and Structure}

In this paper we investigate the behaviour of ABC when used to estimate the parameters of HMMs for which the conditional densities of the observations given the hidden state are intractable.  We shall use a specialization, first proposed in \cite{jassinmarmcc2010}, of the standard ABC likelihood approximation \eqref{nonHMMABCdef}  for when the observations are generated by a HMM.  Specifically we approximate the likelihood of a given sequence of observations $\hat{Y}_{1} , \ldots , \hat{Y}_{n}$ from a HMM with the probability
\begin{equation}
\mathbb{P}_{\theta} \left(  Y_{1} \in B^{\epsilon}_{\hat{Y}_{1}} , \ldots, Y_{n} \in
B^{\epsilon}_{\hat{Y}_{n}} \right)  \, \label{introABCllh}%
\end{equation}
where $B^{\epsilon}_{y}$ denotes the ball of radius $\epsilon$ centered around
the point $y$. The benefit of this approach is that it retains the Markovian
structure of the model. This facilitates both simpler Markov chain Monte Carlo (MCMC) (e.g.~\cite{mckcoodea2009}) and sequential Monte Carlo (SMC) (e.g.~\cite{jassinmarmcc2010}) implementation of the ABC
approximation. Furthermore our experience suggests that this approximation is competitive, from an
accuracy perspective, with a wide range of competing methods; see the two
afore mentioned references for a deeper discussion of this point.

One could use the approximate likelihoods \eqref{introABCllh}  to estimate the parameters of a HMM in one of two ways.  Firstly one could take a Bayesian approach and use \eqref{introABCllh} to construct an approximation to the posterior.  This is the approach most commonly taken in the literature.  Alternatively, as we shall do in this paper, one could take a frequentist approach and  estimate the parameters of the HMM with the value of the parameter vector which maximizes the corresponding approximate likelihood \eqref{introABCllh} of the observations.  We shall henceforth term this procedure approximate Bayesian computation maximum likelihood estimation (ABC MLE).

Although the use of ABC has become commonplace there has to date
been little investigation of the theoretical properties of its use in parameter estimation in either the Bayesian or frequentist context.  In particular the following questions remain to be answered.  Is ABC MLE consistent?  Do ABC based posterior distributions concentrate around the true value of the parameter vector?  Indeed do ABC based estimators converge to anything at all?  Although these questions may seem abstract it is well known that even the mighty MLE can fail to converge in practice, see \cite{fer1982}.  Thus before ABC can be placed on firm mathematical foundations the questions raised above need  to be addressed.

The purpose of this paper is to bridge this theoretical gap in the context of maximum likelihood estimation.  In particular we develop a theoretical justification of the ABC MLE procedure based on its large sample properties analogous to that provided for MLE by standard results concerning asymptotic consistency and normality.  Our approach to this problem is based on the novel observation that ABC MLE can be considered as performing MLE using the likelihoods of a collection of perturbed HMMs. This implies that the ABC MLE
should in some sense inherit its behaviour from the standard MLE.  Using this
observation we first show that unlike the MLE, which is
asymptotically consistent, the ABC MLE estimator has an innate asymptotic bias.  Secondly we show that this bias can be made arbitrarily small by choosing sufficiently small values of $\epsilon$.  Together these results show that asymptotically the ABC MLE will converge to the true parameter value with a margin of error which can be made arbitrarily small by taking a suitable choice of $\epsilon$.  Thus our results allow us to develop a rigorous formulation of the intuitive justification of ABC and in doing so to provide a firm mathematical basis for performing ABC based inference.

We complete the picture by analysing the so called noisy variant (see e.g.~\cite{feapra2010}) of ABC MLE.  We show that unlike the ABC MLE the noisy ABC MLE is always asymptotically consistent.  This raises the question: does noisy ABC provide us with a `free pass' when performing parameter estimation?  Unfortunately the answer in general is no.   We show that under
reasonable conditions the Fisher information of the noisy ABC MLE is
strictly less than that of the standard MLE. As a result we show that the
noisy ABC suffers from a relative loss of information and hence statistical efficiency.

As part of these investigations we establish a
novel asymptotic missing information principle for HMMs with observations
perturbed by additive uniform noise which may in itself be of independent
interest to the reader.  Finally we remark that although this study is theoretical it is our belief that the results presented herein will help provide guidance for future methodological developments in the field.

This paper is structured as follows. In Section \ref{sec:notassump} the
notation and assumptions are given. In Section \ref{sec:standABC} we establish
some approximate asymptotic consistency type results for the standard ABC MLE.
In Section \ref{sec:noisyABC} results concerning the asymptotic consistency
and normality of the noisy ABC estimator are presented.  An extension of the ABC method using probability kernels is discussed in Section \ref{secExtensionToKernels} and an overview of the
use of SMC methods to provide a practical way of implementing
ABC is presented in Section \ref{sec:simos}. An example is given in
Section \ref{sec:numex} which provides a qualitative demonstration of the behaviour
of the ABC estimator predicted in Sections \ref{sec:standABC} and
\ref{sec:noisyABC}. The article is summarized in Section \ref{sec:summary}.
Supporting technical lemmas and proofs of some of the theoretical results are housed
in the two appendices.

\section{Notation and Assumptions}

\label{sec:notassump}

\subsection{Notation and Main Assumptions} \label{subsecNotandassumptions}

Throughout this paper we shall use lower case letters $x,y,z$ to denote dummy variables and upper case letters $X,Y,Z$ to denote random variables.  Observations of a random variable will be denoted by $\hat{Y}$.

We shall frequently have to refer to various kinds of both finite, infinite and doubly infinite sequences.  For brevity the following shorthand notations are used.  For any pair of integers $k \leq n$, $Y_{k:n}$ denotes the sequence of random variables $Y_{k} , \ldots ,Y_{n}$; $Y_{-\infty:k}$ denotes the sequence $\ldots ,Y_{k}$; $Y_{n:\infty}$ denotes the sequence $Y_{n} , \ldots $ and $Y_{-\infty:k;n:\infty}$ denotes the sequence $\ldots, Y_{k} ; Y_{n}, \ldots$.  Given a sequence of integers $\ldots , j_{-1} , j_{0} , j_{1} , \ldots$ and indicies $r < s$ we shall let $j_{r:s}$ denote $j_{r} , j_{r+1}, \ldots , j_{s-1}, j_{s}$; $j_{-\infty:r}$ denote $\ldots , j_{r-1}, j_{r}$ and $j_{s:\infty}$ denote $j_{s}, j_{s+1} , \ldots$ respectively.  Further we shall also use $j_{-\infty:\infty}$ to denote the full sequence $\ldots , j_{-1} , j_{0} , j_{1} , \ldots$.
The two notations defined above will be combined in the following manner.  Given a doubly infinite sequence of random variables $\ldots , Y_{-1}, Y_{0} , Y_{1}, \ldots$, a doubly infinite sequence of integers $\ldots , j_{-1} , j_{0} , j_{1} , \ldots$ and indicies $r < s$ we shall let $Y_{j_{r:s}}$ denote the sequence $Y_{j_{r}} , Y_{j_{r+1}}, \ldots , Y_{j_{s-1}} , Y_{j_{s}}$.  The sequences $Y_{j_{-\infty:r}}$, $Y_{j_{s:\infty}}$ and $Y_{j_{-\infty:\infty}}$ are defined analogously.  Lastly given a measure $\mu$ on a Polish space $\mathcal{X}$ we let $\int \cdot \, \mu(d x_{1:n})$ denote integration w.r.t. the n-fold product measure $\mu^{\otimes n}$ on the n-fold product space $\mathcal{X}^{n}$.  Moreover, given a function $f(x_{1} , \ldots , x_{n}): \mathcal{X}^{n} \to \mathbb{R}$ and integers $1 \leq k \leq l \leq n$, we shall let $\int_{\mathcal{X}^{n}} f ( \cdot ) \mu ( d x_{1:k;l:n} )$ denote the partial integrals $\int_{\mathcal{X}^{n}} f ( \cdot ) \mu ( d x_{1} )$ $\cdots \mu ( d x_{k} ) \mu ( d x_{l} ) \cdots \mu ( d x_{n} )$.

The essence of our approach is to show that in some sense the ABC MLE
inherits the properties of the standard MLE. Thus we shall operate under assumptions on the HMMs that are sufficient to ensure asymptotic consistency and
normality of the MLE.

It is assumed that the Markov chain $\left\{  X_{k} \right\}_{k\geq0}$
is time-homogenous and takes values in a compact Polish space $\mathcal{X}$
with associated Borel $\sigma$-field $\mathcal{B} \left(  \mathcal{X} \right)
$.  Throughout it will be assumed that we have a collection of HMMs all defined on the same state space and parametrised by some vector $\theta$ taking values in a compact set $\Theta \in \mathbb{R}^{d}$.  Furthermore we shall reserve $\theta^{\ast}$ to denote the `true' value of the parameter vector.  For each $\theta\in\Theta$ we let $Q_{\theta} \left(  x , \cdot\right)  $
denote the transition kernel of the corresponding Markov chain and for each
$x\in\mathcal{X}$ and $\theta\in\Theta$ we assume that $Q_{\theta} \left(  x ,
\cdot\right)  $ has a density $q_{\theta} \left(  x , \cdot\right)  $
w.r.t.~some common finite dominating measure $\mu$ on $\mathcal{X}$. The
initial distribution of the hidden state will be denoted by $\pi_{0}$.

We also assume that the observations $\left\{  Y_{k} \right\}_{k\geq0}$ take values in a
state space $\mathcal{Y} \subset\mathbb{R}^{m}$ for some $m\geq1$.
Furthermore, for each $k$ we assume that the random variable $Y_{k}$ is conditionally independent of $
X_{-\infty:k-1;k+1:\infty}$ and $Y_{-\infty:k-1;k+1:\infty}$ given $X_{k}$ and that
the conditional laws have densities $g_{\theta} \left(  y \vert x \right)  $
w.r.t.~some common finite dominating measure $\nu$. We further assume that for
every $\theta$ the joint chain $\left\{  X_{k} , Y_{k} \right\}  _{k \geq0}$
is positive Harris recurrent and has a unique invariant distribution
$\pi_{\theta}$. We shall write $\overline{\mathbb{P}}_{\theta}$ to denote the laws
of the corresponding stationary processes and $\overline{\mathbb{E}}_{\theta}$ to
denote expectations with respect to the stationary laws $\bar{\mathbb{P}%
}_{\theta}$.

Given any $\epsilon> 0$ and $y
\in\mathbb{R}^{m}$ let $B^{\epsilon}_{y}$ denote the closed ball of radius $\epsilon$
centered on the point $y$
and let $\mathcal{U}_{B^{\epsilon}_{y}}$ denote the uniform distribution
on $B^{\epsilon}_{y}$.
For any $A \subset\mathbb{R}^{m}$, let $\mathbb{I}_{A}$ denote the
indicator function of $A$.  Additionally, for any square matrix $M \in \mathbb{R}^{m \times m}$, we shall let $\left\Vert M \right\Vert$ denote the Frobenius norm $\left\Vert M \right\Vert^{2} = \sum_{j,k = 1}^{m} M_{j,k}^{2}$.

For any two probability measures $\mu_{1},\mu_{2}$ on a measurable
space $(E,\mathscr{E})$ we let $\|\mu_{1}-\mu_{2}\|_{TV}$ denote the total
variation distance between them. For all $p \in [1,\infty )$ we let
$L_{p}(\mu)$ denote the set of real valued measurable functions satisfying
$\int\left\vert f(x)\right\vert ^{p}\mu(dx)<\infty$.

Finally, we note that the asymptotic results that we prove for the ABC MLE and its noisy variant hold independently of the initial condition of the hidden state process $\left\{ X_{k} \right\}_{k \geq 0}$.  Thus, in order to keep the presentation as concise as possible we shall suppress the presence of the initial condition of the hidden state except in those instances where it needs to be referred to explicitly.

\subsection{Particular Assumptions}

\label{subsecParticularAssumptions}

In addition to the assumptions above, the following assumptions are made at
various points in the article. Assumptions (A1)-(A3) below are sufficient to
guarantee asymptotic consistency of the MLE and (A4)-(A5) ensure the existence
of an asymptotic Fisher information matrix, denoted $I(\theta^{\ast})$.
Further, if the asymptotic Fisher information $I(\theta^{\ast})$ is invertible
then under assumptions (A1)-(A5) the MLE will be asymptotically normal, see \cite{doumouryd2004} for more details.

\begin{hypA}
1The parameter vector $\theta^{\ast}$\ belongs to the interior of
$\Theta$ and $\theta=\theta^{\ast}$ if and only if $\overline{\mathbb{P}%
}_{\theta} ( \ldots , Y_{-1} , Y_{0} , Y_{1} , \ldots )  =\overline{\mathbb{P}}_{\theta^{\ast}} ( \ldots , Y_{-1} , Y_{0} , Y_{1} , \ldots ) $.
\end{hypA}

\begin{hypA}
2For all $y\in\mathcal{Y}$, $x,x^{\prime}\in\mathcal{X}$, the mappings
$\theta\rightarrow q_{\theta}(x,x^{\prime})$ and $\theta\rightarrow g_{\theta
}(\left.  y\right\vert x)$ are continuous w.r.t. $\theta$.
\end{hypA}

\begin{hypA}
3There exist constants $\underline{c}_{1},\overline{c}_{1}\in(0,\infty)$
such that for every $y\in\mathcal{Y}$, $x,x^{\prime}\in\mathcal{X}$,
$\theta\in\Theta$
\[
\underline{c}_{1}\leq q_{\theta}(x,x^{\prime}),\,g_{\theta}(\left.
y\right\vert x)\leq\overline{c}_{1}.
\]

\end{hypA}

For the remaining assumptions we assume that there exists an open
ball $G \subset\Theta$ centered at $\theta^{\ast}$ such that

\begin{hypA}
4For all $y\in\mathcal{Y}$, $x,x^{\prime}\in\mathcal{X}$, the mappings
$\theta\rightarrow q_{\theta}(x,x^{\prime})$ and $\theta\rightarrow g_{\theta
}(\left.  y\right\vert x)$ are twice continuously differentiable on $G$.
\end{hypA}

\begin{hypA}
5There exists a constant $\overline{c}_{2}\in(0,\infty)$ such that for
every $y\in\mathcal{Y}$, $x,x^{\prime}\in\mathcal{X}$, $\theta\in G$
\begin{equation*}
\begin{gathered}
\left\vert \nabla_{\theta}\log q_{\theta}(x,x^{\prime})\right\vert , \left\vert
\nabla_{\theta}\log g_{\theta}\left(  y|x\right)  \right\vert ,\vert \nabla_{\theta}
^{2}\log q_{\theta}(x,x^{\prime}) \vert , \\
\vert \nabla_{\theta}^{2}\log
g_{\theta}\left(  y|x\right)  \vert \leq\overline{c}_{2}.
\end{gathered}
\end{equation*}

\end{hypA}

\begin{rem}
In general assumptions (A3) and (A5) hold when the state space $\mathcal{X}$ is compact and when the conditional laws of the observed state given the hidden state are heavy tailed, see for example Section \ref{sec:numex}.  However we expect that the behaviours predicted by Theorems \ref{thmmisspecconv}, \ref{theo:prob_conv}, \ref{theo:asympSolnBias} and \ref{asymnormthm} will provide a good qualitative guide to the behaviour of ABC MLE in practice even in cases where the underlying HMMs do not satisfy these assumptions.    
\end{rem}

Assumptions (A1)-(A5) are sufficient to show that in some sense the ABC MLE inherits the
its asymptotic properties from  the standard MLE. The Lipschitz assumptions below will be used to
establish quantitative bounds on the relative performance of the ABC MLE estimator with respect to that of the MLE.

\begin{hypA}
6There exists an $L\in(0,\infty)$ such that for all, $x\in\mathcal{X}$,
$y,y^{\prime}\in\mathcal{Y}$, $\theta\in\Theta$
\[
\left\vert g_{\theta}(y|x)-g_{\theta}(y^{\prime}|x)\right\vert \leq
L|y-y^{\prime}|.
\]

\end{hypA}

\begin{hypA}
7There exists an $L\in(0,\infty)$ such that for all, $x\in\mathcal{X}$,
$y,y^{\prime}\in\mathcal{Y}$, $\theta\in\Theta$
\[
\left\vert \nabla_{\theta} g_{\theta}(y|x)-\nabla_{\theta} g_{\theta}(y^{\prime}|x)\right\vert
\leq L|y-y^{\prime}|.
\]

\end{hypA}

\section{Approximate Bayesian Computation}

\label{sec:standABC}

\subsection{Estimation Procedure} \label{ABCProcedureApproxDiscussionSubSection}

Following \cite{jassinmarmcc2010} we consider the ABC approximation to the likelihood of a sequence of observations $\hat{Y}_{1} , \ldots , \hat{Y}_{n}$ for some fixed $\theta\in\Theta$ given by,
\begin{align}
\lefteqn{ \mathbb{P}_{\theta}\left(  Y_{1} \in B_{\hat{Y}_{1}}^{\epsilon} , \ldots,
Y_{n} \in B_{\hat{Y}_{n}}^{\epsilon} \right) } \notag \\
&&  = \int_{\mathcal{X}^{n+1}\times\mathcal{Y}^{n}} \bigg[\prod_{k=1}^{n}
q_{\theta} (x_{k-1}, x_{k})\mathbb{I}_{B^{\epsilon}_{\hat{Y}_{k}}}(y_{k})
g_{\theta} (y_{k} \vert x_{k}) \bigg] \pi_{0} (d x_{0}) \, \mu( d
x_{1:n} ) \nu( d y_{1:n} ) . \notag \\
&& \label{eq:approxnewver1}
\end{align}
The purpose of this paper is to analyse the asymptotic properties of the ABC parameter estimator for HMMs defined by
\begin{procedure}
[ABC MLE]\label{algABCMLE} Given $\epsilon>0$ and data $\hat{Y}_{1},\ldots,\hat{Y}_{n}$, estimate $\theta^{\ast}$ with
\begin{equation} \label{ABCstandestimatorargsupdiffeq999999999999111}
\hat{\theta}^{\epsilon}_{n} = \arg \max_{\theta\in\Theta} \mathbb{P}_{\theta}\left(  Y_{1} \in B_{\hat{Y}_{1}}^{\epsilon} , \ldots,
Y_{n} \in B_{\hat{Y}_{n}}^{\epsilon} \right) .
\end{equation}
\end{procedure}
The key to our analysis is the following observation which is, to our knowledge, original;
\begin{eqnarray}
\lefteqn{\int_{\mathcal{X}^{n+1}\times\mathcal{Y}^{n}} \bigg[\prod_{k=1}^{n}
q_{\theta} (x_{k-1}, x_{k})\mathbb{I}_{B^{\epsilon}_{\hat{Y}_{k}}}(y_{k})
g_{\theta} (y_{k} \vert x_{k}) \bigg] \pi_{0} (d x_{0}) \, \mu( d
x_{1:n} ) \nu( d y_{1:n} )  \notag } \\
&&  \qquad \qquad   \qquad \qquad \propto\int_{\mathcal{X}^{n+1}} \bigg[\prod_{k=1}^{n} q_{\theta} (x_{k-1},
x_{k}) g^{\epsilon}_{\theta} (\hat{Y}_{k} \vert x_{k})\bigg] \pi_{0} (d x_{0}) \mu(d
x_{1:n})  \label{eq:approx}%
\end{eqnarray}
where
\begin{equation}
\label{EqnPertCondLaw} g^{\epsilon}_{\theta} (y \vert x) = \frac{1}{ \nu \left(
B^{\epsilon}_{y} \right) } \int_{B^{\epsilon}_{y}} g_{\theta} (y^{\prime}
\vert x) \, \nu( d y^{\prime} )
\end{equation}
and where we note that by Lemma \ref{lemnullprobprob} the quantity in \eqref{EqnPertCondLaw} is well defined $\nu$ a.s..

The crucial point is that the quantity $g^{\epsilon}_{\theta} ( y \vert x )$ defined in \eqref{EqnPertCondLaw} is the density of the measure obtained by convolving the measure corresponding to $g_{\theta} ( y \vert x )$ with  $\mathcal{U}_{B^{\epsilon}_{0}}$ where the density is taken w.r.t.~the new dominating measure obtained by convolving $\nu$ with $\mathcal{U}_{B^{\epsilon}_{0}}$ .  One can then immediately see that the quantities $q_{\theta} (x, x^{\prime})$ and $g^{\epsilon}_{\theta} (y \vert x)$ appearing in \eqref{eq:approx} are the
transition kernels and conditional laws respectively for a perturbed HMM
$\left\{  X_{k} , Y^{\epsilon}_{k} \right\}  _{k \geq0}$ defined such that
it is equal in law to the process 
\begin{equation}  \label{alg1pertHMM}
\left\{  X_{k} , Y_{k} +
\epsilon Z_{k} \right\}  _{k \geq0}
\end{equation}
where $\left\{  X_{k} , Y_{k} \right\}
_{k \geq0}$ is the original HMM and the $\left\{  Z_{k} \right\}  _{k \geq0}$ are
an i.i.d. sequence of $\mathcal{U}_{B^{1}_{0}}$ distributed random variables.  Crucially the constant of proportionality in \eqref{eq:approx},
 which by definition is equal to $\nu \left( B^{\epsilon}_{\hat{Y}_{1}} \right) \times \cdots \times  \nu \left( B^{\epsilon}_{\hat{Y}_{n}} \right)$, is by Lemma \ref{lemnullprobprob} non-zero $\nu^{\otimes n}$ a.s. and is \emph{independent} of the parameter value $\theta$.  Thus it follows that \eqref{ABCstandestimatorargsupdiffeq999999999999111} is statistically identical to the estimator
\begin{equation} \label{ABCstandestimatorargsupdiffeq999999999999}
\hat{\theta}_{n}^{\epsilon}=\arg\sup_{\theta\in\Theta}p_{\theta}^{\epsilon
}\left(  \hat{Y}_{1},\ldots, \hat{Y}_{n}\right)
\end{equation}
where $p_{\theta}^{\epsilon
}\left(  \cdots \right)$ denotes the likelihood of the observations w.r.t.~the law of the perturbed process $\left\{  X_{k} , Y^{\epsilon}_{k} \right\}  _{k \geq0}$.   The value of expressing the ABC estimator \eqref{ABCstandestimatorargsupdiffeq999999999999111} in the mathematically equivalent form \eqref{ABCstandestimatorargsupdiffeq999999999999} is that \eqref{ABCstandestimatorargsupdiffeq999999999999} reveals the underlying mathematical structure of the estimator and furthermore, as we shall see in the next section, expresses it in a form which is particularly tractable to analysis.  

We note that our observations \eqref{eq:approx} and \eqref{EqnPertCondLaw} are similar in spirit to those made in \cite{wil2008}.  However in that paper the author takes the point of view that the original collection of HMMs for which we are trying to perform inference is itself misspecified.

\subsection{Theoretical Results}

It follows from the previous section that performing ABC MLE is equivalent to estimating the parameter by taking a data set
generated by one of the original HMMs $\left\{  X_{k} , Y_{k} \right\}  _{k
\geq0}$ and finding the value of $\theta$ which maximises the likelihood of
that data set under the corresponding perturbed HMM $\left\{  X_{k} ,
Y^{\epsilon}_{k} \right\}  _{k \geq0}$. Thus the ABC MLE estimator will
effectively suffer from the problem of model mis-specification. This raises
the question of whether the resulting estimator will still be asymptotically
consistent. As the following example shows one must expect that, in general,
the answer to this question will be no.

\begin{example}
For each $\theta\in\left[  0 , 1 \right]  $ let $\left\{  X_{k} \right\}  _{k
\geq0}$ be a directly observed sequence of i.i.d. random variables with common law%

\[
X_{k} = \left\{
\begin{array}
[c]{cc}%
\theta & \text{ w.p. } 0.5\\
-\theta & \text{ w.p. } 0.5
\end{array}
\right.
\]
and let $\theta^{\ast}$ denote the true value of the model parameter. Then for any $\epsilon > 0$ the ABC MLE will not be asymptotically consistent even though the MLE estimator is asymptotically
consistent for any value of $\theta^{\ast}$.  Furthermore for
$2 \theta^{\ast} > \epsilon> \theta^{\ast} > 0$ the ABC approximation to the likelihood is maximized at $\theta = 0$ for any sequence of observations.
\end{example}

Although the ABC MLE estimator is no longer asymptotically consistent we show
the following below. Almost surely the ABC MLE will converge, with increasing sample size, to a given point in parameter space (more generally the set of accumulation points will belong to a given subset of parameter space).  Further, we show that these accumulation points must lie in some neighbourhood of
the true parameter value and that the size of this neighbourhood shrinks to zero as $\epsilon$ goes
to zero (Theorem \ref{theo:prob_conv}). Finally we show that under certain
Lipschitz conditions one can obtain a rate for the decrease in the size of
these neighbourhoods (Theorem \ref{theo:asympSolnBias}).  We note that these results are very much misspecified MLE results  in the spirit of, for example, \cite{whi1982}.  However because the dominating measures of the original and perturbed HMMs are no longer necessarily mutually absolutely continuous with respect to each other they can no longer be interpreted in terms of minimising Kullback-Leibler
distances.

Before we present our results we first discuss some technical issues that arise in their proofs.  It is tempting to try and understand the behaviour of the ABC MLE by extending the parameter space $\Theta$ to include $\epsilon$ and then applying standard results from the theory of MLE.  Unfortunately the existing theory of MLE requires that the perturbed likelihoods $g^{\epsilon}_{\theta}(y \vert x)$ (see \eqref{EqnPertCondLaw}) be continuous w.r.t.~$\epsilon$ which is not true for general dominating measures $\nu$.  The essence of our method is show that despite this certain asymptotic quantities associated with the  likelihoods of the perturbed processes \eqref{alg1pertHMM} are still sufficiently continuous as functions of $\epsilon$.  In order to do this we need to establish that in some probabalistic sense the order of the operations of differentiating and taking asymptotic limits can be interchanged.  It it this that constitutes the bulk of Appendix B.

In order to state and prove these results it is convenient to make the following definitions.  For any $\theta\in$ $\Theta$ and $\epsilon>0$, let
\begin{equation}
l^{\epsilon}(\theta)=\bar{\mathbb{E}}_{\theta^{\ast}}\left[  \log p_{\theta
}^{\epsilon}(Y_{1} \vert Y_{-\infty:0})\right]  \label{eq:asympLogLik_epsilon}%
\end{equation}
where $p_{\theta}^{\epsilon}(\cdot|\cdot)$ denotes the conditional laws of the observations of the perturbed processes (\ref{alg1pertHMM}) given the infinite past and the expectations are taken with respect to the stationary measure of the unperturbed HMM with parameter $\theta^{\ast}$. Further for $\epsilon=0$ we let
\begin{equation}
l^{0}(\theta)=l(\theta)=\bar{\mathbb{E}}_{\theta^{\ast}}\left[  \log
p_{\theta}(Y_{1} \vert Y_{-\infty:0})\right]  \label{eq:asympLogLik} .
\end{equation}
Our first result shows that the ABC MLE is asymptotically biased

\begin{theo}
\label{thmmisspecconv} Assume (A2)-(A3). Then for every $\epsilon>0$,
$\sup_{\theta\in\Theta}l^{\epsilon}(\theta)$ is achieved. Further let
\[
\label{AccumSetDef}\mathcal{T}^{\epsilon}=\left\{  \theta^{\prime}\in
\Theta:l^{\epsilon}(\theta^{\prime})=\sup_{\theta\in\Theta}l^{\epsilon}%
(\theta)\right\}
\]
be the set of these maximizers, then for any initial distribution $\pi_{0}$
we have that almost surely every accumulation point of the sequence of
estimators $\hat{\theta}_{1}^{\epsilon},\ldots$ defined in Procedure
\ref{algABCMLE} belongs to $\mathcal{T}^{\epsilon}$.
\end{theo}

\begin{proof}
It follows from (A2) and (A3) that for the perturbed HMM defined in \eqref {alg1pertHMM} the conditional laws $p^{\epsilon}_{\theta
}(y_{1} \vert y_{-n:0})$ are continuous w.r.t.~$\theta$. Further
it follows from (A3) and (\ref{lemFiltStabStandardeq2}) that the conditional
laws $p^{\epsilon}_{\theta}(y_{1} \vert y_{-n:0})$ converge
uniformly to the conditional laws $p^{\epsilon}_{\theta}(y_{1} \vert
y_{-\infty:0})$ and are uniformly bounded, both above and away from zero. It then
follows that the conditional log-likelihood functions $\log p^{\epsilon
}_{\theta}(y_{1} \vert y_{-n:0})$ are continuous, uniformly
bounded and converge uniformly to $\log p^{\epsilon}_{\theta}(y_{1}
\vert y_{-\infty:0})$ and hence that the expected values $\bar{\mathbb{E}%
}_{\theta^{\ast}} \left[  \log p^{\epsilon}_{\theta} (Y_{1} \vert Y_{-\infty: 0})\right]  $ are also continuous functions of $\theta\in\Theta$. The first
part of the theorem then follows from the compactness of $\Theta$.

The second part of the result now follows from (A2) and (A3) by using the same
arguments as used by \cite{doumouryd2004} to prove the asymptotic consistency of
the MLE. We leave it to the reader to check the details.
\end{proof}

Although Theorem \ref{thmmisspecconv} shows that the ABC MLE is asymptotically
biased, the following result shows that this error can be made arbitrarily
small by choosing a sufficiently small $\epsilon$.

\begin{theo}
\label{theo:prob_conv} Assume (A1)-(A3). Then
\begin{equation} \label{thmprobconvstateeq}
\lim_{\epsilon\rightarrow0}\quad\sup_{\theta\in\mathcal{T}^{\epsilon}%
}\left\vert \theta-\theta^{\ast}\right\vert =0.
\end{equation}
\end{theo}

\begin{rem}
Theorems \ref{thmmisspecconv} and \ref{theo:prob_conv} provide a theoretical justification for the ABC MLE procedure analogous to that provided for the standard MLE procedure by the classical notion of asymptotic consistency.  In particular they show that an arbitrary degree of accuracy in the parameter estimate can be achieved given sufficient data and a sufficiently small $\epsilon$.
\end{rem}

In order to prove Theorem \ref{theo:prob_conv} we need the following Lemma whose proof is relegated to Appendix B.

\begin{lemma} \label{lemexpllhcont}
Assume (A2)-(A3).  Then the mapping $(\theta,\epsilon)\in\Theta\times\lbrack
0,\infty)\rightarrow l^{\epsilon}(\theta)$ is continuous in $\theta$ and right continuous in $\epsilon$ in the sense that for all pairs of sequences $\theta_{n} \to \theta$ and $\epsilon_{n} \searrow \epsilon$ we have that
\begin{equation*}
l^{\epsilon_{n}}(\theta_{n}) \to l^{\epsilon}(\theta) .
\end{equation*}
\end{lemma}

\begin{proof}[Proof of Theorem \ref{theo:prob_conv}]
In order to prove \eqref{thmprobconvstateeq}, given that by Lemma \ref{lemexpllhcont} the mapping $(\theta,\epsilon)\in\Theta\times\lbrack
0,\infty)\rightarrow l^{\epsilon}(\theta) - l^{\epsilon}(\theta^{\ast})$ is continuous, it is sufficient to show that for any $\delta>0$ there exists an
$\epsilon^{\prime}>0$ such that $\mathcal{T}^{\epsilon}\subset B_{\theta
^{\ast}}^{\delta}$ for all $\epsilon\leq\epsilon^{\prime}$. Suppose that this
property does not hold. Then, by the compactness of $\Theta$, there must exist $\delta > 0$ and
sequences $\epsilon_{n}\searrow0$ and $\theta_{n}\rightarrow\theta\in\left\{
\theta^{\prime}:\left\vert \theta^{\prime}-\theta^{\ast}\right\vert \geq
\delta\right\}  $ such that
\[
l^{\epsilon_{n}}(\theta_{n}) - l^{\epsilon_{n}}(\theta^{\ast}) \geq 0
\]
for all $n$.  However it would then follow from the continuity of $l^{\epsilon}(\theta) - l^{\epsilon}( \theta^{\ast} )$\ that $l( \theta) \geq l( \theta^{\ast} )$ which violates  (A1).  (In \cite{doumouryd2004} it is shown that under (A2) and (A3) that (A1) is equivalent to having that $l(\theta^{\ast}) > l(\theta)$ for all $\theta \neq \theta^{\ast}$.)
\end{proof}

The next result shows that, under some additional assumptions, we can
characterise the rate at which the asymptotic error in the ABC MLE decreases
with $\epsilon$.

\begin{theo}
\label{theo:asympSolnBias} Assume (A1)-(A7) and that the asymptotic Fisher
information matrix $I(\theta^{\ast})$ is invertible. Then there exist finite
positive constants $C,\overline{\epsilon}$ such that for all $\epsilon
\leq\overline{\epsilon}$
\[
\sup_{\theta\in\mathcal{T}^{\epsilon}}\left\vert \theta-\theta^{\ast
}\right\vert \leq C\epsilon.
\]
\end{theo}

The proof of Theorem \ref{theo:asympSolnBias} relies on the following lemma
whose proof is given in Appendix B.

\begin{lemma}
\label{lemABCGradApprox} Assume (A1)-(A7). Then $\nabla_{\theta} l^{\epsilon}$, $\nabla_{\theta} l$ and $\nabla_{\theta}^{2} l$ exist for all $\theta \in G$ where $G$ is as in (A4) and (A5).  Furthermore
\begin{equation}
\sup_{\theta\in G}\left\vert \nabla_{\theta} l^{\epsilon}(\theta)-\nabla_{\theta} l(\theta
)\right\vert \leq R\epsilon\label{eq:uniformGradlnApprox1}%
\end{equation}
for some $R>0$ and $\nabla_{\theta}^{2} l (\theta^{\ast}) = I(\theta^{\ast})$.
\end{lemma}

\begin{proof}[Proof of Theorem \ref{theo:asympSolnBias}]
Since by assumption $I(\theta^{\ast})$ is invertible and thus positive definite it follows
that there exists some $T>0$ such that
\begin{equation}
\inf_{v:\left\vert v\right\vert >0}\frac{\left\vert I(\theta^{\ast
})v\right\vert }{\left\vert v\right\vert }\geq T.\label{eq:matPerturbation1}%
\end{equation}
By Lemma \ref{lemABCGradApprox}, $l(\theta)$ is twice
continuously differentiable on $G$ and so there exists a constant $\delta>0$
such that
\begin{equation}
\sup_{\left\vert \theta-\theta^{\ast}\right\vert \leq\delta}\left\Vert
\nabla_{\theta}^{2} l (\theta)-I(\theta^{\ast})\right\Vert \leq\frac{T}{2}%
.\label{eq:matPerturbation}%
\end{equation}
By Theorem \ref{theo:prob_conv} there exists a constant $\bar{\epsilon}>0$
such that for all $\epsilon\leq\bar{\epsilon}$,
\begin{equation}
\sup_{\theta\in\mathcal{T}^{\epsilon}}\left\vert \theta-\theta^{\ast
}\right\vert \leq\delta.\label{eq:distTheta_nThetaStar}%
\end{equation}
Consider any $\hat{\theta}^{\epsilon}\in\mathcal{T}^{\epsilon}$. By Lemma \ref{lemABCGradApprox} both
$\nabla_{\theta} l^{\epsilon}(\hat{\theta}^{\epsilon})$ and $\nabla_{\theta} l(\theta^{\ast})$ exist and clearly they must both be equal to zero and hence by (\ref{eq:uniformGradlnApprox1})
\begin{equation}
\left\vert \nabla_{\theta} l(\hat{\theta}^{\epsilon})\right\vert \leq R\epsilon
.\label{thmABCMLEquant1}%
\end{equation}
Further by the fundamental theorem of calculus
\begin{equation}
\label{thmABCMLEquant2}\nabla_{\theta} l(\hat{\theta}^{\epsilon})=\nabla_{\theta} l(\theta
^{\ast})+\left(  \int_{0}^{1}\nabla_{\theta}^{2}l\left(  \theta^{\ast}+t(\hat{\theta
}^{\epsilon}-\theta^{\ast})\right)  dt\right)  (\hat{\theta}^{\epsilon}%
-\theta^{\ast}).
\end{equation}
By (\ref{eq:matPerturbation1}), (\ref{eq:matPerturbation}) and
(\ref{eq:distTheta_nThetaStar}) it now follows that
\begin{equation} \label{thmABCMLEquant3}
\left\vert \left(  \int_{0}^{1}\nabla_{\theta}^{2}l\left(  \theta^{\ast}+t(\hat{\theta
}^{\epsilon}-\theta^{\ast})\right)  dt\right)  (\hat{\theta}^{\epsilon}%
-\theta^{\ast})\right\vert \geq\frac{T}{2}\left\vert \hat{\theta}^{\epsilon
}-\theta^{\ast}\right\vert
\end{equation}
The result now follows from \eqref{thmABCMLEquant1}, \eqref{thmABCMLEquant2} and \eqref{thmABCMLEquant3}.
\end{proof}

\begin{rem} \label{remsufficentstatstandardABCMLE}
In many cases the complete data sequence $\hat{Y}_{1} , \ldots , \hat{Y}_{n}$ is too high-dimensional 
and instead one performs inference using a summary statistic $\mathcal{S} ( \hat{Y}_{1} , \ldots , \hat{Y}_{n} )$ where $\mathcal{S} ( \cdots )$ is some mapping from $\mathbb{R}^{m} \times \cdots \mathbb{R}^{m}$ to a lower dimensional Euclidean space, e.g. see \cite{tavbalgridon1997}.  In general this mapping will destroy the Markovian structure of the data and the results derived in this section will not be applicable to ABC based parameter inference conducted using the corresponding summary statistic.

However in practice it is often the case that the mapping $\mathcal{S} ( \cdots )$ is of the form $\mathcal{S} ( \hat{Y}_{1} , \ldots , \hat{Y}_{n} ) = S(\hat{Y}_{1}), \ldots , S(\hat{Y}_{n})$ for some function $S(\cdot)$ that maps from $\mathbb{R}^{m}$ to a space $\mathbb{R}^{m^{\prime}}$ of lower dimension.  When this is true it is easy to see that the Markovian structure of the data is preserved.  Moreover suppose that assumptions (A1)-(A7) hold for the underlying HMM.  If the mapping $S(\cdot)$ preserves the identifiability of the system, that is to say if assumption (A1) also holds for the HMMs with observations $S(Y_{1}), S(Y_{2}) , \ldots $, then it is trivial to see that assumptions (A2)-(A7) will also be preserved for all reasonable choices of $S(\cdot)$ and thus that Theorems \ref{thmmisspecconv}, \ref{theo:prob_conv} and \ref{theo:asympSolnBias} will also hold for ABC MLE performed using the summary statistic.
\end{rem}

\section{Noisy Approximate Bayesian Computation}

\label{sec:noisyABC}

\subsection{Estimation Procedure} \label{SubSectionNoisyABCMLEProcDef}

In the previous section we showed that performing ABC MLE is equivalent to
estimating the parameter by choosing the value of the maximizer of the likelihoods of the perturbed HMMs $\left\{  X_{k} ,
Y^{\epsilon}_{k} \right\}  _{k \geq0}$ defined in (\ref{alg1pertHMM}).
%of a sequence of data generated
%from an unperturbed HMM $\left\{  X_{k} , Y_{k} \right\}  _{k
%\geq0}$.
Since
the likelihoods over which we maximise are misspecified with respect to the law of the process that is
generating the data the resulting estimator has an inherent asymptotic bias.

Suppose now that a sequence of observations $\hat{Y}_{1} , \ldots, \hat{Y}_{n}$ from the unperturbed HMM corresponding to some $\theta^{\ast} \in\Theta$ is given.  The sequence of noisy observations $\hat{Y}_{1} + \epsilon Z_{1} , \ldots, \hat{Y}_{n} +
\epsilon Z_{n}$ where $Z_{k}\overset{\text{i.i.d.}}{\sim}\mathcal{U}_{B^{1}_{0}}$, $k\geq1$ has the same law as a sample from the corresponding perturbed HMM defined in \eqref{alg1pertHMM}.  As a result estimating $\theta^{\ast}$ by applying the  ABC MLE estimator \eqref{ABCstandestimatorargsupdiffeq999999999999111} to the {\em noisy} observations $\hat{Y}_{1} + \epsilon Z_{1} , \ldots, \hat{Y}_{n} + \epsilon Z_{n}$ in place of $\hat{Y}_{1} , \ldots, \hat{Y}_{n}$, is statistically equivalent to estimating $\theta^{\ast}$ by applying standard MLE to the perturbed HMMs \eqref{alg1pertHMM}.  Clearly one would expect that the resulting estimator would inherit the properties
of MLE, in particular that it would be asymptotically consistent.
In light of the discussion and remarks immediately following the definition of Procedure \ref{algABCMLE} these observations lead one to the following noisy ABC MLE procedure:

\begin{procedure}
[Noisy ABC MLE]\label{algNoisyABCMLE} Given $\epsilon> 0$ and data $\hat{Y}_{1} ,
\ldots, \hat{Y}_{n}$ estimate $\theta^{\ast}$ with
\begin{equation} \label{ProcNoisyABCestimator1defeq}
\tilde{\theta}^{\epsilon}_{n} = \arg\sup_{\theta\in\Theta} \mathbb{P}_{\theta}\left(  Y^{\epsilon}_{1} \in B_{\hat{Y}_{1} + \epsilon \hat{Z}_{1}}^{\epsilon} , \ldots,
Y^{\epsilon}_{n} \in B_{\hat{Y}_{n} + \epsilon \hat{Z}_{n}}^{\epsilon} \right) .
\end{equation}
\end{procedure}

\begin{rem}
Procedure \ref{algNoisyABCMLE} is a likelihood-based version of the noisy ABC method in \cite{feapra2010}.
\end{rem}

\subsection{Theoretical Results}

In this section we investigate mathematically the noisy ABC MLE procedure defined in Section \ref{SubSectionNoisyABCMLEProcDef}.   In particular we
show that under the assumptions made in Section
\ref{subsecParticularAssumptions} that the noisy ABC MLE inherits the properties of
asymptotic consistency and normality from the MLE. Further we provide an analysis of the performance of the noisy ABC MLE  relative to the standard MLE by comparing their asymptotic variances. It is first shown that the asymptotic Fisher information of the ABC MLE is strictly less than that of the
MLE and hence that the asymptotic variance of the ABC MLE estimator is
strictly greater.  Thus it follows that the noisy ABC MLE procedure comes at the
cost of a loss in accuracy relative to that of the standard ABC procedure. Finally we show that this loss in
accuracy can be made arbitrarily small by choosing $\epsilon$ small enough.

The first result establishes that under (A1)-(A3) the noisy ABC MLE inherits
the property of asymptotic consistency.

\begin{theo}
\label{thmNoisyABCAsyCon} Assume (A1)-(A3). Then Procedure
\ref{algNoisyABCMLE} is asymptotically consistent.
\end{theo}

\begin{proof}
It is sufficient to show that if (A1)-(A3) hold for the original HMM then they
also hold for the perturbed HMM.
Recall, for the perturbed HMM, the transitions are as for the original
HMM and the likelihood is as \eqref{EqnPertCondLaw}. Thus (A3) for the original model immediately implies (A3) for the perturbed model.

In order to establish that (A2) holds for the perturbed model it is sufficient
to observe that continuity of the mapping $\theta\to g^{\epsilon}_{\theta} ( y
\vert x)$ for any $x \in\mathcal{X}$, $y \in\mathcal{Y}$ follows from
continuity of the mapping $\theta\to g_{\theta} ( y \vert x)$, uniform
boundedness of $g_{\theta} ( y \vert x)$ (ie. (A3)) and the dominated convergence theorem.

It remains to show that (A1) is also inherited by the perturbed model. This
assumption is equivalent to demanding that for every $\theta^{\prime}
\neq\theta$ there exists some $r$ such that
\begin{equation}
\label{NoisyABCHMMA5}\mathcal{L}_{\theta} \left(  Y_{1}, \ldots, Y_{r}
\right)  \neq\mathcal{L}_{\theta^{\prime}} \left(  Y_{1}, \ldots, Y_{r}
\right)
\end{equation}
where $\mathcal{L}_{\theta} \left(  \cdot\right)  $ denotes the law of the
process $\left\{  Y_{k} \right\}  _{k \geq0}$. However by applying Lemma
\ref{pertdistdifferability} it immediately follows that (\ref{NoisyABCHMMA5})
holds if and only if
\[
\mathcal{L}_{\theta} \left(  Y^{\epsilon}_{1}, \ldots, Y^{\epsilon}_{r}
\right)  \neq\mathcal{L}_{\theta^{\prime}} \left(  Y^{ \epsilon}_{1}, \ldots,
Y^{\epsilon}_{r} \right)
\]
for all $\epsilon$ and so (A1) holds for the original HMMs if and only if it
also holds for the perturbed HMMs.
\end{proof}

Next we consider the question of asymptotic normality.  In \cite{doumouryd2004} it was shown that under conditions (A1)-(A5) the MLE for HMMs has asymptotic Fisher information matrix $I(\theta^{\ast})$ where

\[
I(\theta^{\ast})=\bar{\mathbb{E}}_{\theta^{\ast}}\left[  \nabla_{\theta}\log p_{\theta^{\ast
}}\left(  Y_{1} \vert Y_{-\infty:0}\right)  \nabla_{\theta}\log p_{\theta^{\ast}}\left(
Y_{1} \vert Y_{-\infty:0},\right)  ^{T}\right]  .
\]
Further it was shown that if $I(\theta^{\ast})$ is invertible then the MLE is
asymptotically normal with asymptotic variance equal to $I(\theta^{\ast}%
)^{-1}$. It follows from the proof of Theorem \ref{thmNoisyABCAsyCon} that if
(A1)-(A3) hold for the original HMM then they also hold for the perturbed HMM.
Further if (A4) and (A5) hold for the original HMM then a simple application
of the dominated convergence theorem shows that they also hold for the
perturbed HMM. Thus, under assumptions (A1)-(A5) the asymptotic Fisher
information matrix of the noisy ABC MLE exists and is equal to
$I^{\epsilon}(\theta^{\ast})$ where
\[
I^{\epsilon}(\theta^{\ast})=\bar{\mathbb{E}}_{\theta^{\ast}}\left[  \nabla_{\theta}\log
p_{\theta^{\ast}}^{\epsilon}\left(  Y^{\epsilon}_{1} \vert Y_{-\infty:0}^{\epsilon}\right)
\nabla_{\theta}\log p_{\theta^{\ast}}^{\epsilon}\left(  Y_{1}^{\epsilon} \vert Y_{-\infty:0}^{\epsilon}\right)  ^{T}\right]  .
\]
Moreover if $I^{\epsilon}(\theta^{\ast})$ is invertible then the noisy ABC MLE estimator will be asymptotically normal with asymptotic variance equal to $I^{\epsilon}(\theta^{\ast})^{-1}$.  Using these results we can analyze the asymptotic performance of the noisy ABC MLE estimator relative to that of the standard MLE estimator by comparing the two Fisher information matrices.  Unfortunately one cannot in general make any explicit quantitative comparisons
between these two quantities, however the
following result establishes some qualitative relations between the two.

\begin{theo}
\label{asymnormthm} Assume (A1)-(A5). Then:

\begin{enumerate}
[1.]

\item $I ( \theta^{\ast} ) \geq I^{\epsilon}( \theta^{\ast} )$. Further if $\nu$ is connected and $I(\theta^{\ast}) \neq 0$ (see Section \ref{subsecNotandassumptions}) then the inequality is strict.

\item $I^{\epsilon}( \theta^{\ast} ) \to0$ as $\epsilon\to\infty$.

\item $I^{\epsilon}( \theta^{\ast} ) \to I ( \theta^{\ast} )$ as $\epsilon
\to0$. Hence for epsilon sufficiently small the ABC MLE is asymptotically
normal with asymptotic variance equal to $I^{\epsilon}( \theta^{\ast} )^{-1}$.

\item If (A6) and (A7) hold then $\left\Vert I( \theta^{\ast} ) - I^{\epsilon
}( \theta^{\ast} ) \right\Vert = O ( \epsilon^{2} )$.
\end{enumerate}
\end{theo}
Theorem \ref{asymnormthm} tells us that asymptotic variance of the noisy ABC MLE estimator is strictly greater than that of the MLE estimator and hence that there is a loss in accuracy relative
to the MLE in using noisy ABC MLE.  For very large values of $\epsilon$ the
asymptotic variance of the noisy ABC MLE grows without bound and the loss in accuracy becomes almost complete.  Thus if one chooses values of $\epsilon$ which are too large the noisy ABC MLE becomes ineffective. Furthermore we have shown that by taking small enough values of $\epsilon$ the
loss in accuracy can be be made arbitrarily small and hence that we can obtain (ignoring computational issues) a
performance of the noisy ABC MLE arbitrarily close to that of
the MLE.  Finally, the theorem provides a rate of convergence for the Fisher information matricies for when the likelihoods obey certain simple Lipschitz assumptions.

The proof of Theorem \ref{asymnormthm} is based on the following lemma, see
Appendix B for the proof.

\begin{lemma}
\label{asymmissinf} Assume (A1)-(A5). Then
\[
I( \theta^{\ast} ) = I^{\epsilon}( \theta^{\ast} ) + \bar{\mathbb{E}}%
_{\theta^{\ast}} \left[  I^{ Y_{0} : Y^{\epsilon}_{0} }_{ Y_{-\infty:-1};Y^{\epsilon}_{1: \infty}} ( \theta^{\ast} ) \right]
\]
where for every doubly infinite sequence $Y_{-\infty:-1};Y^{\epsilon}_{1: \infty}$ the random variable $I^{ Y_{0} : Y^{\epsilon}_{0} }_{Y_{-\infty:-1};Y^{\epsilon}_{1: \infty}} ( \theta^{\ast} )$ is equal to the difference in the Fisher
informations of the conditional laws of $Y_{0}$ and $Y^{\epsilon}_{0} $ given
$Y_{-\infty:-1};Y^{\epsilon}_{1: \infty}$, that is
\begin{align*}
\lefteqn{I^{ Y_{0} : Y^{\epsilon}_{0} }_{Y_{-\infty:-1};Y^{\epsilon}_{1: \infty}} ( \theta^{\ast} ) := } \\
& \qquad \bar{\mathbb{E}}_{\theta^{\ast}} \bigg[ \nabla_{\theta} \log p_{\theta^{\ast}} \left( Y_{0} \vert Y_{-\infty:-1};Y^{\epsilon}_{1:\infty} \right) \cdot \\
& \qquad \qquad \qquad \qquad \nabla_{\theta} \log p_{\theta^{\ast}} \left( Y_{0} \vert Y_{-\infty:-1};Y^{\epsilon}_{1:\infty} \right)^{T} \vert Y_{-\infty:-1};Y^{\epsilon}_{1:\infty} \bigg] \\
&\qquad - \bar{\mathbb{E}}_{\theta^{\ast}} \bigg[ \nabla_{\theta} \log p_{\theta^{\ast}} \left( Y^{\epsilon}_{0} \vert Y_{-\infty:-1};Y^{\epsilon}_{1:\infty} \right) \cdot \\
& \qquad \qquad \qquad \qquad  \nabla_{\theta} \log p_{\theta^{\ast}} \left( Y^{\epsilon}_{0} \vert Y_{-\infty:-1};Y^{\epsilon}_{1:\infty} \right)^{T} \vert Y_{-\infty:-1};Y^{\epsilon}_{1:\infty} \bigg] .
\end{align*}
\end{lemma}

\begin{rem}
The quantity $I^{ Y_{0} : Y^{\epsilon}_{0} }_{Y_{-\infty:-1};Y^{\epsilon}_{1: \infty}} ( \theta^{\ast} )$ is also equal to the missing information in
the conditional law of $Y^{\epsilon}_{0}$ relative to that in the conditional law of $Y_{0}$ (where both laws are conditioned on $Y_{-\infty:-1};Y^{\epsilon}_{1: \infty}$). Here the term missing information is meant in the sense of that proposed for i.i.d. random variables in \cite{orcwoo1972}. Hence, Lemma
\ref{asymmissinf} can be considered as a conditional asymptotic missing
information principle for HMMs with observations perturbed by uniform additive
noise.
\end{rem}

Theorem \ref{asymnormthm} is then an immediate corollary of the following
lemma which establishes the behaviour of $I^{ Y_{0} : Y^{\epsilon}_{0}
}_{Y_{-\infty:-1};Y^{\epsilon}_{1: \infty}} ( \theta^{\ast} )$ for
different values of $\epsilon$.

\begin{lemma}
\label{lemABCMLEcomp} Assume (A1)-(A5). Then:

\begin{enumerate}
[1.]

\item $\bar{\mathbb{E}}_{\theta^{\ast}} \left[  I^{ Y_{0} : Y^{\epsilon}_{0}
}_{Y_{-\infty:-1};Y^{\epsilon}_{1: \infty}} ( \theta^{\ast} ) \right]  $ is
positive semi-definite. Further if $\nu$ is connected and $I(\theta^{\ast}) \neq 0$ then
$\bar{\mathbb{E}}_{\theta^{\ast}} \left[  I^{ Y_{0} : Y^{\epsilon
}_{0} }_{Y_{-\infty:-1};Y^{\epsilon}_{1: \infty}} ( \theta^{\ast} )
\right]  \neq0$ for any $\epsilon> 0$.

\item $\bar{\mathbb{E}}_{\theta^{\ast}} \left[  I^{ Y_{0} : Y^{\epsilon}_{0}
}_{Y_{-\infty:-1};Y^{\epsilon}_{1: \infty}} ( \theta^{\ast} ) \right]  \to
I ( \theta^{\ast} )$ as $\epsilon\to\infty$.

\item $\bar{\mathbb{E}}_{\theta^{\ast}} \left[  I^{ Y_{0} : Y^{\epsilon}_{0}
}_{Y_{-\infty:-1};Y^{\epsilon}_{1: \infty}} ( \theta^{\ast} ) \right]
\to0$ as $\epsilon\to0$.

\item Assume that (A6) and (A7) also hold. Then $\left\Vert \bar{\mathbb{E}}%
_{\theta^{\ast}} \left[  I^{ Y_{0} : Y^{\epsilon}_{0} }_{Y_{-\infty:-1};Y^{\epsilon}_{1: \infty}} ( \theta^{\ast} ) \right] \right\Vert  = O ( \epsilon^{2})$.
\end{enumerate}
\end{lemma}

The proof of Lemma \ref{lemABCMLEcomp} is again deferred to Appendix B.

\begin{rem} \label{remsufficentstatsnoisyABCMLE}
Comments similar to those in Remark \ref{remsufficentstatstandardABCMLE} concerning summary statistics also hold for the results on the noisy ABC MLE given in this section.  In particular we note that given a summary statistic of the form $S ( \hat{Y}_{1} ) , \ldots , S (\hat{Y}_{n} )$ one can derive a result analogous to Theorem \ref{asymnormthm} in which the Fisher information matrices $I ( \theta^{\ast} )$ and $I^{\epsilon} ( \theta^{\ast} )$ are replaced with the Fisher information matrices for the HMMs $S(Y_{1}) , \ldots$ and $S(Y_{1}) + \epsilon Z_{1}, \ldots$ where $S(Y_{1}) + \epsilon Z_{1}, \ldots$ is a perturbed version of $S(Y_{1}) , \ldots$ defined in an analogous manner to \eqref{alg1pertHMM}.
\end{rem}

\section{Smoothed ABC} \label{secExtensionToKernels}

ABC estimators based on Procedures \ref{algABCMLE} and \ref{algNoisyABCMLE}
have an inherent lack of smoothness due to the fact that the estimator effectively gives weight one to points inside the balls $B_{\hat{Y}_{1}}^{\epsilon} , \ldots , B_{\hat{Y}_{n}}^{\epsilon}$ and weight zero to those outside them.
As seen in the next section, this becomes particularly problematic if one then tries to estimate these probabilities using SMC algorithms
as the algorithm can collapse due to the use of indicator functions; see \cite{deldoujasabc2008}
for some discussion.

A common way of smoothing ABC, see for example \cite{beaumont2002}, is to approximate the likelihoods of a sequence of observations $\hat{Y}_{1} , \ldots , \hat{Y}_{n}$ not with \eqref{eq:approxnewver1} but instead with the smoothed approximations
\begin{align}
\lefteqn{ \mathbb{E}_{\theta} \left[ \phi \bigg( \frac{ \hat{Y}_{1} - Y_{1} }{\epsilon} \bigg) \cdots \phi \bigg( \frac{ \hat{Y}_{n} - Y_{n} }{\epsilon} \bigg) \right] } \notag \\
 &= \int_{\mathcal{X}^{n+1}\times\mathcal{Y}^{n}} \bigg[\prod_{k=1}^{n}
q_{\theta} (x_{k-1}, x_{k}) \phi  ( \frac{ \hat{Y}_{k} - y_{k} }{\epsilon} )
g_{\theta} (y_{k} \vert x_{k}) \bigg] \pi_{0} (d x_{0}) \, \mu( d
x_{1:n} ) \nu( d y_{1:n} ) \notag \\
\label{KernelApproxeq1}
\end{align}
where $\phi( \cdot )$ is the density w.r.t. Lebesgue measure of some smooth probability distribution $\Phi$.  One then estimates the parameters
via maximising \eqref{KernelApproxeq1}. 

By using exactly the same arguments as in Section \ref{ABCProcedureApproxDiscussionSubSection} it is clear that the smoothed ABC MLE estimator resulting from approximating the likelihoods of a sequence of observations $\hat{Y}_{1} , \ldots , \hat{Y}_{n}$ with \eqref{KernelApproxeq1} for some suitable kernel $\phi$ is statistically equivalent to estimator obtained by by approximating the true likelihoods with the likelihoods of the perturbed HMM defined to be
\begin{equation}
\left\{  X_{k},Y_{k}^{\Phi,\epsilon
}\right\}  _{k\geq0} := \{X_{k},Y_{k}+\epsilon Z_{k}\}_{k \geq 0}  \label{algSmoothed1pertHMM}%
\end{equation}
where the $\left\{  Z_{k}\right\}  _{k\geq0}$ are such that $Z_{k}\overset{\text{\emph{i.i.d.}}}{\sim} \Phi$.  Further, in an analogous manner to Section \ref{SubSectionNoisyABCMLEProcDef} one can define a smoothed noisy ABC MLE by applying the smoothed ABC MLE defined above to noisy data of the form $\hat{Y}_{1} + \epsilon \hat{Z}_{1} , \ldots, \hat{Y}_{n} + \epsilon \hat{Z}_{n}$ where again $Z_{k}\overset{\text{\emph{i.i.d.}}}{\sim} \Phi$.

It is natural to ask whether results analogous to Theorems \ref{thmmisspecconv}, \ref{theo:prob_conv}, \ref{theo:asympSolnBias}, \ref{thmNoisyABCAsyCon} and \ref{asymnormthm} hold for the smoothed ABC MLE and the smoothed noisy ABC MLE.  By a careful reading of the proofs of these theorems one can see that analogous results hold when the density of $\Phi$ satisfies the following conditions:
\begin{enumerate}[(i)]
\item $\phi(y) > 0$ for all $y \in \mathbb{R}^{m}$. \\
\item $\phi ( \cdot )$ is continuously differentiable. \\
\item for the reference measure $\nu$ and all $f \in L_{\infty}$,
\begin{equation*}
\lim_{\epsilon \to 0} \frac{\int f(y^{\prime}) \phi(\frac{y - y^{\prime}}{\epsilon}) \nu ( d y^{\prime} )}{ \int \phi(\frac{y - y^{\prime}}{\epsilon}) \nu ( d y^{\prime} )} = f(y) \quad \nu \text{ a.s.} .
\end{equation*} \\
\item $\int x^{2} \phi(x) dx < \infty$.
\end{enumerate}
We observe that these conditions hold for many commonly used smoothing distributions, in particular the Gaussian distribution.

Finally it is noted that comments analogous to those in Remarks \ref{remsufficentstatstandardABCMLE} and \ref{remsufficentstatsnoisyABCMLE} hold for the smoothed ABC MLE and smoothed noisy ABC MLE.  Moreover the quantities \eqref{KernelApproxeq1} can be straight-forwardly estimated using SMC techniques, see the following section for more details.

\section{Implementing ABC via SMC}

\label{sec:simos}

SMC algorithms are commonly used to approximate conditional laws of the form $p (
X_{k} \vert Y_{1:k} )$ (we drop the $\hat{Y}_k$ notation and omit dependence upon $\theta$ here). At each time $k$ the conditional
law of the hidden state is approximated by a collection of $N$ particles, $x_{k}^{1} , \ldots,
x_{k}^{N}$ as
\begin{equation}
\label{SMCdesceq1}\widehat{p}( \cdot \vert Y_{1:k} ) = \frac{1}{N}%
\sum_{l=1}^{N}\delta_{x_{k}^{l}}(\cdot).
\end{equation}
The crucial feature of the SMC algorithm with respect to any form of likelihood based parameter inference is that at each step,  $\frac{1}{N} \sum_{l=1}^{N} g ( Y_{k} \vert x^{l}_{k})$, is an approximation to the conditional
likelihood $p(Y_{k} \vert Y_{1:k-1})$. Thus when the conditional likelihoods $g ( \cdot \vert \cdot )$ are tractable SMC algorithms can
be used to generate approximations to the full likelihoods $p(Y_{1} , \ldots,
Y_{n})$, e.g. see \cite{anddoutad2009} for the use of SMC for MLE in this standard setting.

Consider now the ABC MLE and noisy ABC MLE procedures defined in Sections \ref{sec:standABC} and \ref{sec:noisyABC} and recall that we approximate the true likelihoods with the likelihoods of the perturbed HMMs \eqref{alg1pertHMM}.   To see how standard SMC methods can be implemented in the context of these estimators   consider the extended process $\left\{  X_{k}, Y_{k}, Y^{\epsilon}_{k}
\right\}  _{k \geq0}$ defined such that $\left\{  X_{k} , Y_{k} \right\}  _{k \geq0}$ are
the hidden state and observation process of the original HMM and for all $k
\geq0$, $Y^{\epsilon}_{k} = Y_{k} + \epsilon Z_{k}$ where $\left\{
Z_{k} \right\}  _{k \geq0}$ is an i.i.d.~ sequence of $\mathcal{U}_{B_{0}^{1}}$ random variables. Clearly the marginal
distributions of the observations of the extended process are equal to those of the
observations of the perturbed HMMs defined in
\eqref{alg1pertHMM}.   Thus in order to compute the ABC approximation to the likelihood of a sequence of observations $\hat{Y}_{1}
, \ldots, \hat{Y}_{n}$ it is sufficient to compute the likelihood of the observations under the extended HMM detailed above.  Since the conditional densities of the observed state given the hidden state of the extended HMM are trivial the corresponding likelihoods may be computed using standard SMC.  This suggests the following SMC algorithm for evaluating the ABC approximate likelihoods \eqref{eq:approxnewver1}, see \cite{jassinmarmcc2010}
\begin{algorithm}{SMC for Computation of Approximate Bayesian Likelihood $p^{\epsilon}(\hat{Y}_{1} , \ldots , \hat{Y}_{n})$.}
\newline
\newline
For $k=1, \ldots , n$ do
\begin{enumerate}[1.]
\item Generate proposal states $(\tilde{x}_{k}^{1},\tilde{y}_{k}^{1}), \ldots, (\tilde{x}_{k}^{N},\tilde{y}_{k}^{N})$
where each $\tilde{x}_{k}^{l} \sim q (x_{k-1}^{l}, \cdot)$ and each $\tilde{y}_{k}^{l} \sim g ( \cdot \vert \tilde{x}_{k}%
^{l} )$.
\item Weight each proposed state $(\tilde{x}_{k}^{l}, \tilde{y}_{k}^{l})$ with $\tilde{w}_{k}^{l} =
\mathbb{I}_{ B^{\epsilon}_{\hat{Y}_{k}} }(\tilde{y}_{k}^{l}) $.
\item Renormalise the weights; $\tilde{w}_{k}^{l} \mapsto w_{k}^{l} :=
\tilde{w}_{k}^{l} / \sum_{l=1}^{N} \tilde{w}_{k}^{l}$.
\item Generate the particles $x_{k}^{1}, \ldots, x_{k}^{N}$ by sampling
multinomially from the proposals $\tilde{x}_{k}^{1}, \ldots, \tilde{x}_{k}%
^{N}$ according to the weights $w_{k}^{1} , \ldots, w_{k}^{N}$.
\end{enumerate}
Finally approximate the likelihood $p^{\epsilon}(\hat{Y}_{1} , \ldots , \hat{Y}_{n})$ by $\prod_{k=1}^{n} \left( \frac{1}{N}\sum_{l=1}^{N} \tilde{w}_{k}^{l} \right)$.
\end{algorithm}

Similarly, given a distribution $\Phi$ with smooth density $\phi$ w.r.t. Lebesgue measure, one can define a SMC algorithm for computing the corresponding smoothed ABC approximations to the likelihoods in an analogous manner; the details follow from Algorithm 2.
% \begin{algorithm}{SMC for Computation of Smoothed Approximate Bayesian %Likelihood $p^{\Phi, \epsilon}(\hat{Y}_{1} , \ldots , \hat{Y}_{n})$.} \newline
% \newline
% For $k=1, \ldots , n$ do
% \begin{enumerate}[1.]
% \item Generate proposal states $(\tilde{x}_{k}^{1},\tilde{y}_{k}^{1}), \ldots, (\tilde{x}_{k}^{N},\tilde{y}_{k}^{N})$
% where each $\tilde{x}_{k}^{l} \sim q (x_{k-1}^{l}, \cdot)$ and each $\tilde{y}_{k}^{l} \sim g ( \cdot \vert \tilde{x}_{k}%
% ^{l} )$.
% \item Weight each proposed state $(\tilde{x}_{k}^{l}, \tilde{y}_{k}^{l})$ with $\tilde{w}_{k}^{l} =
% \phi( \frac{\hat{Y}_{k} - \tilde{y}_{k}^{l}}{\epsilon} )  $.
% \item Renormalise the weights; $\tilde{w}_{k}^{l} \mapsto w_{k}^{l} :=
% \tilde{w}_{k}^{l} / \sum_{l=1}^{N} \tilde{w}_{k}^{l}$.
% \item Generate the particles $x_{k}^{1}, \ldots, x_{k}^{N}$ by sampling
% multinomially from the proposals $\tilde{x}_{k}^{1}, \ldots, \tilde{x}_{k}%
% ^{N}$ according to the weights $w_{k}^{1} , \ldots, w_{k}^{N}$.
% \end{enumerate}
% Finally approximate the likelihood $p^{\Phi, \epsilon}(\hat{Y}_{1} , \ldots , \hat{Y}_{n})$ by $\prod_{k=1}^{n} \left( \frac{1}{N}\sum_{l=1}^{N} \tilde{w}_{k}^{l} %\right)$.
% \end{algorithm}

Note that in general one does not have to resample the particles at every step
and more efficient approaches may be possible, see for example \cite{deldoujas2008} and the references therein. A detailed analysis of the SMC method,
including description of resampling and convergence results can be found in
\cite{doudefgor2001} and \cite{del2004}.

\section{Numerical Example}

\label{sec:numex}

It is common in economics to model the log returns of a sequence of price data using a HMM.  Typically one uses the hidden state to model certain underlying economic factors which cannot be directly observed and the observed state to model the log returns of the prices themselves.  Furthermore it has become increasingly common to model the distribution of the log returns of asset prices using $\alpha$-stable distributions due to their seemingly good fit to the actual data, see for example \cite{racmit2000}.  Unfortunately the likelihoods of $\alpha$-stable distributions are intractable and so using them presents difficulties when trying to infer model parameters from real financial data.

In this section we study the performance of both the standard and noisy ABC MLE procedures when used to estimate the scale parameter of the following toy economic model with intractable likelihoods.  The hidden state $\left\{ X \right\}_{k \geq 0}$ takes values in the set $\left\{ -1, 1 \right\}$ and the corresponding Markov chain  has transition matrix
\begin{equation*}
\left( 
\begin{array}{cc}
\frac{19}{20} & \frac{1}{20} \\
\frac{1}{5} & \frac{4}{5}
\end{array}
\right) .  
\end{equation*}
Conditional on the hidden state the observed state $Y_{k} \sim S_{\alpha} ( \sigma, 0, X_{k}+\delta)$ where $S_{\alpha} ( \sigma, \beta, \delta)$ denotes the $\alpha$-stable distributions with parameters $\alpha, \sigma, \beta$ and $\delta$, see for example \cite{samtaq1994}.  Intuitively the hidden state denotes the health of the underlying 
economy, $+1$ being good ie.~growth and $-1$ being bad ie.~recession.  Given the state of the economy the log returns of the relevant asset price are then $\alpha$-stable distributed with a positive or negative drift as appropriate.

\begin{figure}
\centering
\scalebox{0.35}[0.3]{\includegraphics[viewport=620 190 550 600]{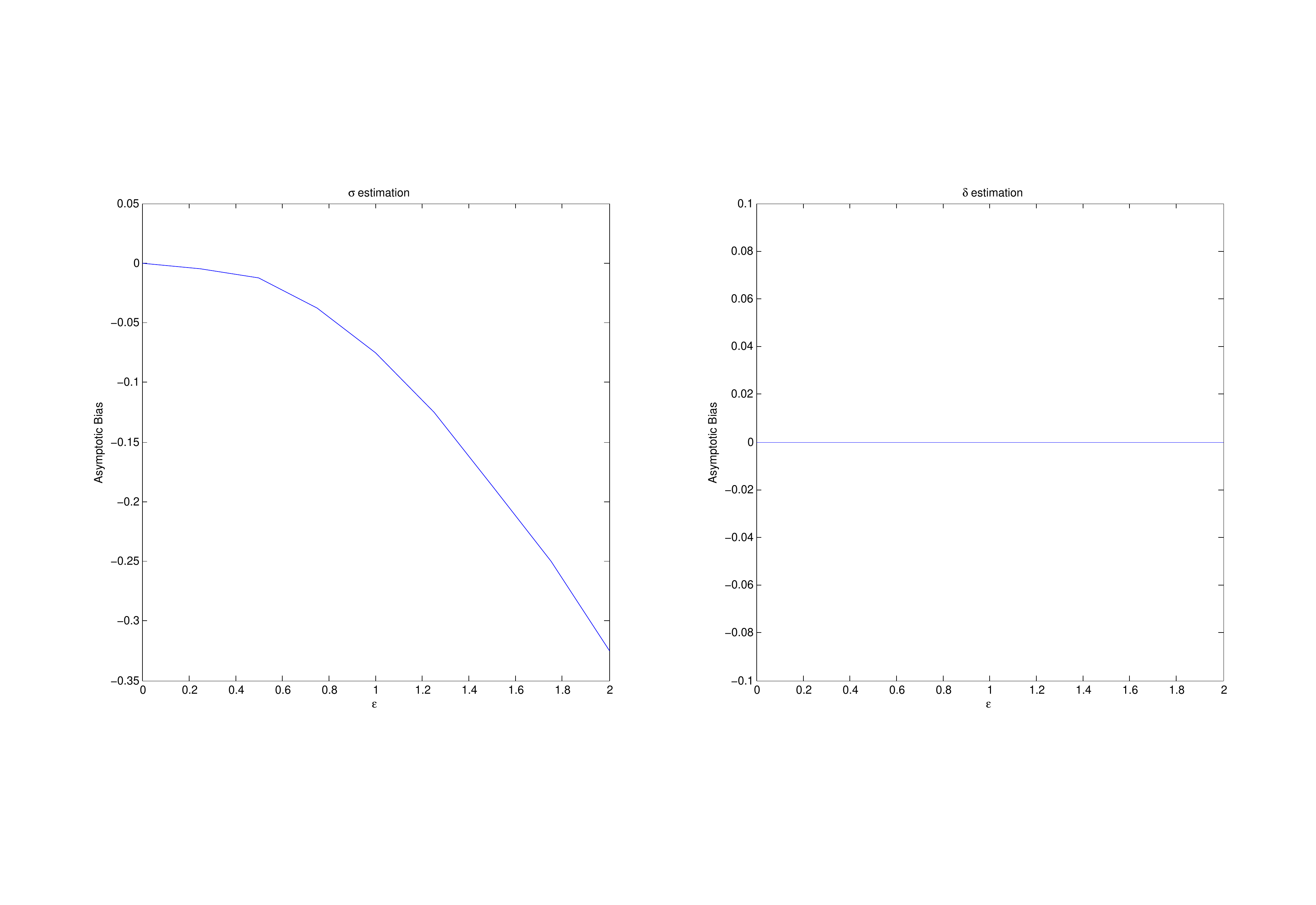}}
\caption{Asymptotic bias of ABC MLE parameter estimates.} \label{StandardABCFig}
\end{figure}

In Figure \ref{StandardABCFig} we plot the asymptotic bias of the ABC MLE when used to estimate the parameters $\sigma$ and $\delta$ given that the true model parameters are $\alpha = 1.8$, $\sigma = 1$ and $\delta = 0$.  We note that the ABC MLE seems to induce a bias in the estimates of the scale parameters but not of the location parameters.  Intuitively this can be understood as being due to the fact that the observed states of the perturbed HMMs \eqref{alg1pertHMM} have a greater variance than those of the corresponding original HMMs but the same mean position.  Lastly we note that for very small $\epsilon$ the size of the bias seems to be $O ( \epsilon^{2} )$ ie. one order of magnitude less than the upper bound obtained in Theorem \ref{theo:asympSolnBias}.

Finally we investigate the behaviour of the noisy ABC MLE.  In the first graph in Figure \ref{NoisyABCFig} we plot the Fisher information matrix as a function of $\epsilon$.  The data suggests that for small $\epsilon$ the loss of information is $O ( \epsilon^{2} )$.  In the second graph we plot the log of the inverse of the Fisher information as a function of log $\epsilon$.  In this case the resulting data suggests that the Fisher information in the noisy ABC MLE decays as the inverse of the fourth power of $\epsilon$  for sufficiently large values of $\epsilon$.

\begin{figure}[h]
\centering
\scalebox{0.35}[0.3]{\includegraphics[viewport=620 190 550 600]{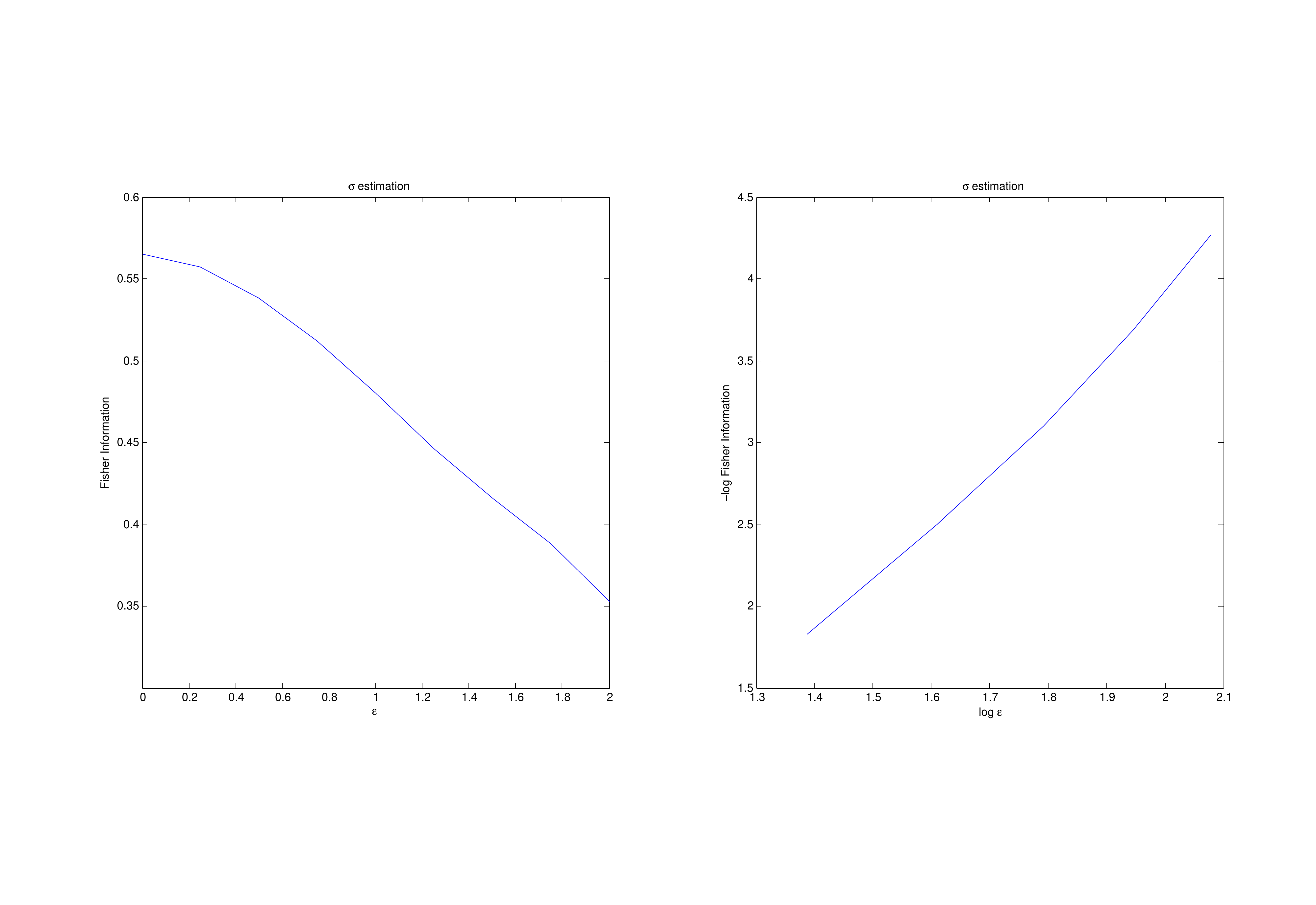}}
\caption{Fisher Information of noisy ABC MLE parameter estimates.} \label{NoisyABCFig}
\end{figure}

This second plot indicates that the Fisher information in the noisy ABC MLE decays as the fourth power of $\epsilon$, at least for large values of $\epsilon$.  This suggests that in order for ABC MLE to provide accurate parameter estimates one must use relatively small values of $\epsilon$.  However this conflicts with the need to keep $\epsilon$ reasonably large in order to achieve computational stability.  we note that even in this simple 1-D linear model we had to use large numbers of particles in our SMC algorithms to obtain accurate estimates of the ABC likelihoods for small values of $\epsilon$.  In higher dimensions this problem will be even worse as the volumes of the $\epsilon$ balls around the observations will decay even more quickly with $\epsilon$ than in the one dimensional case.  This suggests that in order for ABC to become a truly practical statistical method one needs to find algorithms that can generate samples from arbitrarily small neighborhoods of a point in an efficient manner.  One way in which this may be done is to marry ABC with techniques from the rapidly growing field of rare event simulation (see \cite{rubtuf2009} for a recent overview of this area).

\section{Summary} \label{sec:summary}

In this article we have investigated the behaviour of the ABC and noisy ABC MLEs when used for estimating the parameters of HMMs.  We have shown that mathematically these estimators should both be understood as being MLEs implemented using the likelihoods of a collection of perturbed HMMs. Using this insight we have shown that the standard ABC MLE has an innate asymptotic bias which can be made arbitrarily small by choosing a sufficiently small value of the parameter $\epsilon$.  Further we have shown that the noisy ABC MLE provides an asymptotically consistent estimator which is also, under certain conditions analogous to those for the MLE, asymptotically normal.  Moreover this noisy version of the estimator has a loss of information relative to the MLE which manifests itself via an increase in the variance of the parameter estimates.  Finally we have shown that under very mild conditions these results can be extended to smoothed versions of the standard and noisy ABC MLEs.

These theoretical results help to solidify and extend existing intuition associated to the approximations that have been considered.  Further they suggest some possible avenues for future investigation.  Firstly one would expect that the theoretical results in this paper will hold under much weaker assumptions than those presented here.  The question of finding the necessary mathematical tools to relax these assumptions remains an interesting and important open problem.  Secondly, the numerical results suggest that in order to provide an efficient and accurate method of parameter estimation ABC MLE will in practice need to be combined with computational techniques that allow one to generate samples effectively from sets with very small probabilities.  The question of finding a generally applicable method of doing this is the topic of our current research.

\section*{Appendix A: Auxiliary Results}

Here we present some supporting technical lemmas. The first result is a
standard result from real analysis which we state without proof.

\begin{lemma}
\label{lemrealanalresstand} Suppose that there exists a function
$f:\mathbb{R}^{u} \to\mathbb{R}^{v}$ and sequence of continuously
differentiable functions $f_{n}:\mathbb{R}^{u} \to\mathbb{R}^{v}$, $n \geq1$,
such that $f_{n}(z), \nabla_{\theta} f_{n}(z)$ are bounded uniformly in $n$ and $z$,
$f_{n}(z) \to f(z)$ uniformly in $z$ and the sequence $\nabla_{\theta} f_{n}(z)$ is
Cauchy uniformly in z. Then $f$ is itself uniformly bounded and continuously
differentiable and $\nabla_{\theta} f (z) = \lim_{n \to\infty} \nabla_{\theta} f_{n} (z)$
uniformly in $z$.
\end{lemma}

The second lemma is concerned with the identifiability of probability
distributions under additive noise.

\begin{lemma}
\label{pertdistdifferability} Let distributions $\mu_{1}, \mu_{2}$ and $\nu$
on $\mathbb{R}^{m}$ for some $m\geq1$ be given and suppose that the characteristic function of $\nu$ is equal to zero on a set of Lebesgue measure zero. Then%

\[
\mu_{1} = \mu_{2} \iff\mu_{1} \ast\nu= \mu_{2} \ast\nu.
\]

\end{lemma}

\begin{proof}
For any distribution $\mu$ we shall let $\varphi_{\mu} ( \lambda)$ denote the
corresponding characteristic function. It is well known that for any pair of
random variables $\mu$ and $\nu$, $\varphi_{\mu\ast\nu} ( \lambda) = \varphi_{\mu} (
\lambda) \varphi_{\nu} ( \lambda)$ and that $\mu= \nu$ if and only if $\varphi_{\mu}
( \lambda) = \varphi_{\nu} ( \lambda)$ for all $\lambda$. Thus we have that
\begin{align*}
\mu_{1} = \mu_{2}  &  \iff\varphi_{\mu_{1}} ( \lambda) = \varphi_{\mu_{2}} (
\lambda) \text{ for all } \lambda\\
&  \iff\varphi_{\mu_{1}} ( \lambda) \varphi_{\nu} ( \lambda) = \varphi_{\mu_{2}} (
\lambda) \varphi_{\nu} ( \lambda) \text{ for all } \lambda\\
&  \iff\mu_{1} \ast\nu= \mu_{2} \ast\nu.
\end{align*}

\end{proof}

The following three Lemmas are well known results concerned with the connectedness of the support of a measure.  We state them without proof.

\begin{lemma}
\label{lemnullprobprob} Let a probability distribution $\mu$ on $\mathbb{R}%
^{m}$ for some $m\geq1$ be given. Then for all $\epsilon> 0$ the set
\[
F_{\mu,\epsilon} := \left\{  y \in\mathbb{R}^{m} : P_{\mu} \left(  Y \in
B^{\epsilon}_{y} \right)  = 0 \right\}
\]
is measurable and
\[
P_{\mu} \left(  F_{\mu,\epsilon} \right)  = 0 .
\]

\end{lemma}

\begin{lemma}
If the support of $\mu$\ is connected then so is the support of the n-fold product
measure $\mu^{\otimes n}$ for any $n \geq 1$.
\label{remmeasureconnectedness}
\end{lemma}

\begin{lemma}
\label{lemmeasureconnectedness} Suppose that the support of a probability measure $\mu$ on
$\mathbb{R}^{m}$ is connected (see Section \ref{subsecNotandassumptions}),
then so is the support of the probability measure $\mu\ast\mathcal{U}_{B^{\epsilon}_{0}}$ for
any $\epsilon> 0$.
\end{lemma}

The next lemma shows that adding noise to an observation will, in general,
result in a loss of information. The lemma after shows that for very large
amounts of noise the loss in information will be almost complete.

\begin{lemma}
\label{lemnoiseinfloss} Suppose that there exists a collection of
distributions $\mathbb{P}_{\theta}$ on some $\mathcal{Y}\subset\mathbb{R}^{m}$
parameterised by $\theta\in\Theta$ and with densities $p_{\theta}\left(
\cdot\right)  $ with respect to some common finite dominating measure $\mu$,
and that the densities $p_{\theta}\left(  \cdot\right)  $ are differentiable
w.r.t. $\theta$. For all $\theta\in\Theta$ and $\epsilon>0$ let $\mathbb{P}%
_{\theta}^{\epsilon}=\mathbb{P}_{\theta}\ast\mathcal{U}_{B_{0}^{\epsilon}}$.
Then for any $\theta\in\Theta$ and $\epsilon>0$
\begin{equation}
\mathbb{E}_{\mathbb{P}_{\theta}}\left[  \nabla_{\theta}\log p_{\theta}(Y).\nabla_{\theta}\log
p_{\theta}(Y)^{T}\right]  \geq\mathbb{E}_{\mathbb{P}_{\theta}^{\epsilon}%
}\left[  \nabla_{\theta}\log p_{\theta}^{\epsilon}(Y).\nabla_{\theta}\log p_{\theta}^{\epsilon
}(Y)^{T}\right]  \label{lemnoiseinflosspfeq1}%
\end{equation}
where $p_{\theta}^{\epsilon}(\cdot)$ denotes the density of the distribution
$\mathbb{P}_{\theta}^{\epsilon}$ with respect to the finite dominating measure
$\mu\ast\mathcal{U}_{B_{0}^{\epsilon}}$. Furthermore, if the supports of the  distributions $\mathbb{P}%
_{\theta}$ are all connected then we
have equality in (\ref{lemnoiseinflosspfeq1}) if and only if both quantities
are equal to the zero matrix.
\end{lemma}

\begin{proof}
Let $\theta\in\Theta$ be given and let $Y$ be a random variable distributed
according to $p_{\theta}(\cdot)$. Observe that given $\epsilon$ the quantity
$p_{\theta}^{\epsilon}(\cdot)$ is equal to the density of the random variable
$Y^{\epsilon}=Y+\epsilon Z$ (with respect to the appropriate dominating
measure) where $Z$ is an independent random variable and $Z\sim\mathcal{U}%
_{B_{0}^{1}}$. By a straightforward application of the Fisher identity and the
fact that $p_{\theta}(Y,Y^{\epsilon})=p_{\theta}(Y)\mathbb{I}_{B_{\epsilon}%
}(Y^{\epsilon}-Y)$ one has that $\nabla_{\theta}\log p_{\theta}^{\epsilon}(Y^{\epsilon
})=\mathbb{E}\left[  \nabla_{\theta}\log p_{\theta}(Y,Y^{\epsilon})|Y^{\epsilon
}\right]  =\mathbb{E}\left[  \nabla_{\theta}\log p_{\theta}(Y)|Y^{\epsilon}\right]  $
a.s. where $p_{\theta}(\cdot,\cdot)$ denotes the joint density of the random
variables $Y,Y^{\epsilon}$ from which it follows that for any $v\in
\mathbb{R}^{m}$, $v^{T}\nabla_{\theta}\log p_{\theta}^{\epsilon}(Y^{\epsilon
})=\mathbb{E}\left[  v^{T}\nabla_{\theta}\log p_{\theta}(Y)|Y^{\epsilon}\right]  $.
Furthermore given $v\in\mathbb{R}^{m}$ we have that
\begin{equation}
\begin{gathered} v^{T} \mathbb{E}_{\mathbb{P}_{\theta}} \left[ \nabla_{\theta}\log p_{\theta} (Y). \nabla_{\theta}\log p_{\theta} (Y)^{T} \right] v = \mathbb{E}_{\mathbb{P}_{\theta}} \left[ v^{T} \nabla_{\theta}\log p_{\theta} (Y). \nabla_{\theta}\log p_{\theta} (Y)^{T} v \right] , \\ v^{T} \mathbb{E}_{\mathbb{P}_{\theta}^{\epsilon}} \left[ \nabla_{\theta}\log p^{\epsilon}_{\theta} ( Y ) . \nabla_{\theta}\log p^{\epsilon}_{\theta} ( Y )^{T} \right] v = \mathbb{E}_{\mathbb{P}_{\theta}^{\epsilon}} \left[ v^{T} \nabla_{\theta}\log p^{\epsilon}_{\theta} ( Y ) . \nabla_{\theta}\log p^{\epsilon}_{\theta} ( Y )^{T} v \right] . \end{gathered} \label{lemnoiseinflosspfeq2}%
\end{equation}
Applying Jensen's inequality to \eqref{lemnoiseinflosspfeq2} yields
\[
v^{T}\mathbb{E}_{\mathbb{P}_{\theta}}\left[  \nabla_{\theta}\log p_{\theta}%
(Y).\nabla_{\theta}\log p_{\theta}(Y)^{T}\right]  v\geq v^{T}\mathbb{E}_{\mathbb{P}%
_{\theta}^{\epsilon}}\left[  \nabla_{\theta}\log p_{\theta}^{\epsilon}(Y).\nabla_{\theta}\log
p_{\theta}^{\epsilon}(Y)^{T}\right]  v
\]
for all $v\in\mathbb{R}^{m}$ from which \eqref{lemnoiseinflosspfeq1}
immediately follows.

We now prove the second assertion. Since the mapping $z\in\mathbb{R}%
\rightarrow z^{2}$ is strictly convex it further follows from Jensen's
inequality that for any $v\in\mathbb{R}^{m}$,
\[
v^{T}\mathbb{E}_{\mathbb{P}_{\theta}}\left[  \nabla_{\theta}\log p_{\theta}%
(Y).\nabla_{\theta}\log p_{\theta}(Y)^{T}\right]  v=v^{T}\mathbb{E}_{\mathbb{P}%
_{\theta}^{\epsilon}}\left[  \nabla_{\theta}\log p_{\theta}^{\epsilon}(Y).\nabla_{\theta}\log
p_{\theta}^{\epsilon}(Y)^{T}\right]  v
\]
if and only if $v^{T}\nabla_{\theta}\log p_{\theta}(Y)$ and hence $v^{T}\nabla_{\theta}\log
p_{\theta}(Y,Y^{\epsilon})$ is $\sigma\left(  Y^{\epsilon}\right)  $
measurable. Thus equality holds in \eqref{lemnoiseinflosspfeq1} if and only if
$v^{T}\nabla_{\theta}\log p_{\theta}(Y,Y^{\epsilon})$ is $\sigma\left(  Y^{\epsilon
}\right)  $ measurable for all $v\in\mathbb{R}^{m}$ which holds if and only if
$\nabla_{\theta}\log p_{\theta}(Y,Y^{\epsilon})$ is $\sigma\left(  Y^{\epsilon}\right)
$ measurable. Hence in order to prove the final part of the result it is
sufficient to show that $\nabla_{\theta}\log p_{\theta}(Y,Y^{\epsilon})$ is
$\sigma\left(  Y^{\epsilon}\right)  $ measurable if and only if it is equal to
zero a.s. Assume that $\nabla_{\theta}\log p_{\theta}(Y,Y^{\epsilon})$ is
$\sigma\left(  Y^{\epsilon}\right)  $ measurable. Then $\nabla_{\theta}\log p_{\theta
}(Y^{\epsilon})=\nabla_{\theta}\log p_{\theta}(Y,Y^{\epsilon})$ a.s.. Using the fact
that $\nabla_{\theta}\log p_{\theta}(y,y^{\epsilon})=\nabla_{\theta}\log p_{\theta}%
(y)\mathbb{I}_{B_{\epsilon}}(y^{\epsilon}-y)$ one then has that 
\begin{equation} \label{revision1}
\nabla_{\theta}\log
p_{\theta}(y)=\nabla_{\theta}\log p_{\theta}(y^{\prime})
\end{equation}
for $\mathbb{P}_{\theta}$
a.s. all $y,y^{\prime}$ such that $\left\vert y-y^{\prime}\right\vert
\leq2\epsilon$. 

Suppose now that $\nabla_{\theta}\log p_{\theta}(Y)$ is not $\mathbb{P}%
_{\theta}$ a.s. constant. Then there must exist $v$ and $\eta$
such that $\mathbb{P}_{\theta}( \vert \nabla_{\theta}\log p_{\theta}(Y) - v \vert \leq \eta),\mathbb{P}_{\theta}%
( \vert \nabla_{\theta}\log p_{\theta}(Y) - v \vert > \eta )>0$.  It then follows from Lemma \ref{lemnullprobprob} that
there must exist points $\underline{y}$ and $\overline{y}$ such that
for all
$\delta>0$
\begin{equation} \label{revision2}
\begin{gathered}
 \mathbb{P}_{\theta} ( \vert Y - \overline{y} \vert \leq \delta , \vert \nabla_{\theta}\log p_{\theta}( Y ) - v \vert \leq \eta ) > 0 , \\
  \mathbb{P}_{\theta} ( \vert Y - \underline{y} \vert \leq \delta , \vert \nabla_{\theta}\log p_{\theta}( Y ) - v \vert \leq \eta ) > 0 .
 \end{gathered}
\end{equation}
Since the support of $\mathbb{P}_{\theta}$ is connected there exists  a continuous curve $\mathcal{C} : [0,1] \to \mathbb{R}^{m}$ contained in the support of $\mathbb{P}%
_{\theta}$ such that $\mathcal{C}(0) = \overline{y}$ and $\mathcal{C}(0) = \underline{y}$.  By the continuity of $\mathcal{C}$ one can find a finite sequence of
open balls $B_{1}^{o},\ldots,B_{n}^{o}$ of radius less than or equal to $\epsilon$ such that $\underline{y}\in B_{1}^{o}%
$, $\overline{y}\in B_{n}^{o}$, $\mathcal{C}\subset\cup_{k=1}^{n}B_{k}^{o}$
and such that for every $1\leq k<n$, $B_{k}^{o}\cap B_{k+1}^{o}\cap
\mathcal{C}\neq\emptyset$. Consider any two neighbouring balls $B_{k}^{o}$ and
$B_{k+1}^{o}$. From the above we have that $\nabla_{\theta}\log p_{\theta}(\cdot)$ is
$\mathbb{P}_{\theta}$ a.s. constant on $B_{k}^{o}$ and $B_{k+1}^{o}$ and that
there exists some ball contained in $B_{k}^{o}\cap B_{k+1}^{o}$ with non zero
$\mathbb{P}_{\theta}$ mass and thus that $\nabla_{\theta}\log p_{\theta}(\cdot)$ is
$\mathbb{P}_{\theta}$ a.s. constant on $B_{k}^{o}\cup B_{k+1}^{o}$. Hence it
follows that $\nabla_{\theta}\log p_{\theta}(\cdot)$ is $\mathbb{P}_{\theta}$ a.s.
constant $\cup_{k=1}^{n}B_{k}^{o}$ which contradicts the assumption that
$\nabla_{\theta}\log p_{\theta}(Y)$ is not $\mathbb{P}_{\theta}$ a.s. constant. Thus it
follows that if $\nabla_{\theta}\log p_{\theta}(Y,Y^{\epsilon})$ is $\sigma\left(
Y^{\epsilon}\right)  $ measurable that $\nabla_{\theta}\log p_{\theta}(\cdot)$ must be
$\mathbb{P}_{\theta}$ a.s. equal to some constant $K$. Further, since
$\mathbb{E}\left[  \nabla_{\theta}\log p_{\theta}(Y)\right]  =0$ it then follows that
$K=0$. Conversely if $\nabla_{\theta}\log p_{\theta}(\cdot)=0$ a.s. then clearly it is
$\sigma\left(  Y^{\epsilon}\right)  $ measurable.
\end{proof}

\begin{lemma}
\label{lemlargenoiseinfasym} Suppose that there exists a collection of
distributions $\mathbb{P}_{\theta}$ on some $\mathcal{Y} \subset\mathbb{R}%
^{m}$ parameterised by the parameter vector $\theta\in\Theta$. Assume that for
every $\theta$ the corresponding distribution has a density $p_{\theta}
\left(  \cdot\right)  $ with respect to some common finite dominating measure
$\mu$, that the densities $p_{\theta} \left(  \cdot\right)  $ are continuously
differentiable w.r.t. $\theta$ and that the corresponding score functions
$\nabla_{\theta}\log p_{\theta} \left(  \cdot\right)  $ are uniformly bounded above in
norm by some some $K < \infty$. For all $\theta$ and $\epsilon$ let
$\mathbb{P}_{\theta}^{\epsilon} = \mathbb{P}_{\theta} \ast\mathcal{U}%
_{B^{\epsilon}_{0}}$. Then for any $\theta$ and any sequence of positive real
numbers $\epsilon_{n}$ such that $\epsilon_{n} \nearrow\infty$
\[
\lim_{n \to\infty} \mathbb{P}^{\epsilon_{n}}_{\theta} \left( \{ y : \left\vert
\nabla_{\theta}\log p^{\epsilon_{n}}_{\theta} ( y ) \right\vert > \delta\}\right)  = 0
\]
for all $\delta> 0$ where $p^{ \epsilon_{n}}_{\theta} ( \cdot)$ denotes the
density of the distribution $\mathbb{P}_{\theta}^{ \epsilon_{n}}$ with respect
to the finite dominating measure $\mu\ast\mathcal{U}_{B_{0}^{\epsilon_{n}}}$.
\end{lemma}

\begin{proof}
Let $\theta\in\Theta$ be given and let $Y$ be a random variable distributed
according to $\mathbb{P}_{\theta}$. As in the proof of Lemma
\ref{lemnoiseinfloss} we observe that given $\epsilon$ the quantity
$p^{\epsilon}_{\theta} ( \cdot)$ is equal to the density of the random
variable $Y^{\epsilon} = Y + \epsilon Z$ (again with respect to the
appropriate dominating measure) where $Z$ is an independent random variable
with $Z\sim\mathcal{U}_{B_{0}^{1}}$. Standard computations show that for any
$y$
\[
\begin{aligned} \nabla_{\theta}\log p^{\epsilon_{n}}_{\theta} ( y ) & = \frac{ \nabla_{\theta} \int p_{\theta} ( z ) \mathbb{I}_{B^{\epsilon_{n}} } \left( y - z \right) \nu(d z ) }{\int p_{\theta} ( z ) \mathbb{I}_{B^{\epsilon_{n}} } \left( y - z \right) \nu(d z ) } \\ & = \frac{ \nabla_{\theta} \int p_{\theta} ( z ) \left( 1 - \mathbb{I}_{(B^{\epsilon_{n}})^{C} } \left( y - z \right) \right) \nu(d z ) }{\int p_{\theta} ( z ) \mathbb{I}_{B^{\epsilon_{n}} } \left( y - z \right) \nu(d z ) } \\ & = \frac{ - \int \nabla_{\theta} p_{\theta} ( z ) \mathbb{I}_{(B^{\epsilon_{n}})^{C} } \left( y - z \right) \nu(d z ) }{\int p_{\theta} ( z ) \mathbb{I}_{B^{\epsilon_{n}} } \left( y - z \right) \nu(d z ) } \end{aligned}
\]
where the last equality follows from the dominated convergence theorem by
(A2), (A3), (A4) and (A5). Since $\left\vert \nabla_{\theta}\log p_{\theta} \left(  y
\right)  \right\vert \leq K$ it follows that $\left\vert \int\nabla_{\theta} p_{\theta}
( z ) \mathbb{I}_{(B^{\epsilon_{n}})^{C} } \left(  y - z \right)  \nu(d z )
\right\vert \leq K \mathbb{P}_{\theta} \left(  Y^{n} \in(B^{\epsilon_{n}}%
_{y})^{C} \right)  $ for all $y$. Hence the proof will follow once we
establish that for any $\delta^{\prime}$
\begin{equation}
\label{lemlargenoiseinfasymeq1}\limsup_{n \to\infty} \mathbb{P}_{\theta
}^{\epsilon_{n}} \left(  \{y : \mathbb{P}_{\theta} \left(  Y \in B^{\epsilon
_{n}}_{ y } \right)  \leq1 - \delta^{\prime} \}\right)  \leq\delta^{\prime} .
\end{equation}
Note that given any $\delta^{\prime}$ there exist $R < \infty$ and $r < 1$
such that $\mathbb{P}_{\theta} \left(  Y \in B^{R}_{0} \right)  $, $\mathbb{P}
\left(  Z \in B^{r}_{0} \right)  > 1 - \delta^{\prime} / 2$ and thus that for
any $\epsilon_{n} > 2 R / \left(  1 - r \right)  $ we have that $\mathbb{P}%
_{\theta} \left(  Y \in B^{R}_{0}, Y + \epsilon_{n} Z \in B^{\epsilon_{n} -
R}_{0} \right)  > 1 - \delta^{\prime}$. Clearly if $y \in B^{\epsilon_{n}%
-R}_{0}$ then $\mathbb{P}_{\theta} \left(  Y \in B^{\epsilon_{n}}_{ y }
\right)  \geq\mathbb{P}_{\theta} \left(  Y \in B^{R}_{ 0} \right)  > 1 -
\delta^{\prime}/2$ and so the result follows.
\end{proof}

The following result establishes a stability-like property of the filter as
the amount of noise in certain components of the observations becomes
infinite. Before we state the result we recall the extended HMM defined in
Section \ref{sec:simos}. Given a HMM $\left\{  X_{k} , Y_{k} \right\}  _{k
\geq0}$ and a perturbed version $\left\{  X_{k} , Y^{\epsilon}_{k} \right\}
_{k \geq0}$ (see \eqref{alg1pertHMM}) we define the extended HMM to be the
joint process $\left\{  X_{k} , Y_{k}, Y^{\epsilon}_{k} \right\}  _{k \geq0}$.
In other words given a HMM $\left\{  X_{k} , Y_{k} \right\}  _{k \geq0}$ and
some $\epsilon> 0$ the extended HMM is the process
\begin{equation}
\left\{  X_{k} , Y_{k} , Y^{\epsilon}_{k} \right\}
_{k \geq0} : =\left\{  X_{k} , Y_{k}, Y_{k} + \epsilon Z_{k} \right\}  _{k
\geq0}%
\label{eqextendedHMMdef}
\end{equation}
where $\left\{  Z_{k} \right\}  _{k \geq0}$ is such that for each $k\geq0$,
$Z_{k}\overset{\text{i.i.d.}}{\sim}\mathcal{U}_{B_{0}^{1}}$.

\begin{lemma}
\label{lemcondprobsatbasepsilongotozero} Let $\left\{  X_{k} , Y_{k} \right\}
_{k \geq0}$ be a HMM which satisfies (A3) and let $\left\{  X_{k} , Y_{k},
Y^{\epsilon}_{k} \right\}  _{k \geq0}$ be the corresponding extended HMM
defined in \eqref{eqextendedHMMdef}. Then for any $l < m$, sequences $j_{1} <
\cdots< j_{r}$, $\tilde{j}_{1} < \cdots< \tilde{j}_{s}$, any $j \leq
\min\left\{  l , j_{1}, \tilde{j}_{1} \right\}  $, $x \in\mathcal{X}$ and
$\delta> 0$
\begin{align}
\lefteqn{\lim_{\epsilon\to\infty} \mathbb{P}\left(  \left\Vert p
\left(  X_{l:m} \vert Y_{j_{1}:j_{r}} ; Y^{\epsilon}_{\tilde{j}_{1}:\tilde
{j}_{s}} ; X_{j} = x \right) \right. \right. }  \nonumber \\
& \qquad \qquad \qquad  \left. \left. - p \Big(  X_{l:m} \vert Y_{j_{1}:j_{r}} ;
X_{j} = x \right)  \right\Vert _{TV} > \delta \Big)  = 0 . \label{lemFiltStabeq9}
\end{align}

\end{lemma}

\begin{proof}
Clearly we can assume that $\left\{  j_{1} , \ldots, j_{r} \right\}
\cap\left\{  \tilde{j}_{1} , \ldots, \tilde{j}_{s} \right\}  = \emptyset$. Let
$k = \max\left\{  m, j_{r}, \tilde{j}_{s} \right\}  $, then using assumption
(A3) and the well known identity
\begin{align}
\label{lemFiltCondProbUpLowBoundseq3}
p \left(  X_{l:m} \vert
Y_{j_{1}:j_{r}} ; Y^{\epsilon}_{\tilde{j}_{1}:\tilde{j}_{s}} ; X_{j} = x
\right) \qquad \qquad \qquad \qquad \qquad \qquad \qquad \qquad \qquad \qquad  \nonumber\\
= \frac{ \int\prod_{u = j+1}^{k} q (x_{u-1}, x_{u}) \prod_{v=1}^{r} g
(Y_{j_{v}} \vert x_{j_{v}}) \prod_{w=1}^{s} g^{\epsilon} (Y^{\epsilon}_{j_{w}}
\vert x_{j_{w}}) d \mu( x_{j+1:l-1;m+1:k} ) }{ \int\prod_{u = j+1}^{k} q
(x_{u-1}, x_{u}) \prod_{v=1}^{r} g (Y_{j_{v}} \vert x_{j_{v}}) \prod_{w=1}^{s}
g^{\epsilon} (Y^{\epsilon}_{j_{w}} \vert x_{j_{w}}) d \mu( x_{j+1:k} )
}\nonumber\\
\end{align}
where $g^{\epsilon} ( \cdot\vert\cdot)$ is as in (\ref{EqnPertCondLaw}) it
follows that in order to show (\ref{lemFiltStabeq9}) it is sufficient to show
that for any $l$ and $\delta> 0$
\begin{equation}
\label{lemFiltStabeq10}\lim_{\epsilon\to\infty} \mathbb{P}\left(  \sup_{x,
x^{\prime} \in\mathcal{X}} \left\vert \frac{g^{\epsilon} (Y^{\epsilon}_{l}
\vert x) }{ g^{\epsilon} (Y^{\epsilon}_{l} \vert x^{\prime}) } - 1 \right\vert
> \delta\right)  = 0 .
\end{equation}
In order to prove \eqref{lemFiltStabeq10} it is sufficient, by assumption
(A3), to show that for any $\delta> 0$
\begin{equation}
\label{lemFiltStabeq99999}\lim_{\epsilon\to\infty} \nu\ast\mathcal{U}%
_{B^{\epsilon}_{0}} \left(  y : \sup_{x, x^{\prime} \in\mathcal{X}} \left\vert
\frac{ \int_{B^{\epsilon}_{y}} g ( y^{\prime} \vert x ) \nu( d y^{\prime} ) }{
\int_{B^{\epsilon}_{y}} g ( y^{\prime} \vert x^{\prime} ) \nu( d y^{\prime} )
} - 1 \right\vert > \delta\right)  = 0.
\end{equation}
By assumption (A3) we have that for any $\delta^{\prime} > 0$ there exists
some $R_{\delta^{\prime}} < \infty$ such that for all $x \in\mathcal{X}$
\[
\int_{(B^{R_{\delta^{\prime}}}_{0})^{C}} g ( y \vert x ) \nu( d y ) <
\delta^{\prime} .
\]
It then follows that given the above $\delta$ there exists some $R_{\delta} <
\infty$ such that $\sup_{x, x^{\prime} \in\mathcal{X}} \left\vert \frac{
\int_{B^{\epsilon}_{y}} g ( y^{\prime} \vert x ) \nu( d y^{\prime} ) }{
\int_{B^{\epsilon}_{y}} g ( y^{\prime} \vert x^{\prime} ) \nu( d y^{\prime} )
} - 1 \right\vert \leq\delta$ for all $y$ such that $B^{R_{\delta}}_{0}
\subset B^{\epsilon}_{y}$. Thus in order to prove \eqref{lemFiltStabeq99999}
it is sufficient to show that for any $R > 0$, $\lim_{\epsilon\to\infty}
\nu\ast\mathcal{U}_{B^{\epsilon}_{0}} \left(  (B^{\epsilon- R}_{0})^{C}
\right)  = 0$. However for any $r \in(0,1)$ we have that
\begin{align*}
\limsup_{\epsilon\to\infty} \nu\ast\mathcal{U}_{B^{\epsilon}_{0}} \left(
(B^{\epsilon- R}_{0})^{C} \right)   &  \leq\limsup_{\epsilon\to\infty} \nu
\ast\mathcal{U}_{B^{\epsilon}_{0}} \left(  (B^{(1-r) \epsilon}_{0})^{C}
\right) \\
&  \leq\limsup_{\epsilon\to\infty} \left(  \nu\left(  (B^{r \epsilon}_{0})^{C}
\right)  + \mathcal{U}_{B^{\epsilon}_{0}} \left(  (B^{(1-2r)\epsilon}_{0})^{C}
\right)  \right)
\end{align*}
from which the result follows.
\end{proof}

The next five results are restatements of certain well-known stability
properties of the filter.

\begin{lemma}

\label{lemFiltStabStandard} Let $\left\{  X_{k},Y_{k}\right\}  $ be a HMM
which satisfies (A3) and let the process $\left\{  X_{k},Y_{k},Y_{k}%
^{\epsilon}\right\}  $ be the corresponding extended HMM defined as in
\eqref{eqextendedHMMdef}.  Then for all $k\leq l<m\leq n$, $j_{1}<\cdots<j_{r}$
and $\tilde{j}_{1}<\cdots<\tilde{j}_{s}$ such that $j_{1}\wedge\tilde{j}%
_{1}\geq k$,\ $j_{r}\vee\tilde{j}_{s}\leq n$, all $x_{k},x_{k}^{\prime}%
,x_{n},x_{n}^{\prime}\in\mathcal{X}$ and all sequences $Y_{j_{1}}%
,\ldots,Y_{j_{r}};Y_{\tilde{j}_{1}}^{\epsilon},\ldots,Y_{\tilde{j}_{s}%
}^{\epsilon}$
\begin{align}
&  \left\Vert \mathbb{P} \left(  X_{l:m}|Y_{j_{1:r}};Y_{\tilde{j}_{1:s}}^{\epsilon
};X_{k}=x_{k}\right)  - \mathbb{P} \left(  X_{l:m}|Y_{j_{1:r}};Y_{\tilde
{j}_{1:s}}^{\epsilon};X_{k}=x_{k}^{\prime}\right)
\right\Vert _{TV} \leq \rho^{(l-k)}  \nonumber\\
&\label{lemFiltStabStandardeq2}%
\end{align}
and
\begin{align}
&  \left\Vert \mathbb{P} \left(  X_{l:m}|Y_{j_{1:r}};Y_{\tilde{j}_{1:s}}^{\epsilon
};X_{k}=x_{k};X_{n}=x_{n}\right) \right. \nonumber \\
& \qquad \qquad \qquad  \left. - \mathbb{P} \left(  X_{l:m}|Y_{j_{1:r}};Y_{\tilde
{j}_{1:s}}^{\epsilon};X_{k}=x_{k}^{\prime},X_{n}=x_{n}^{\prime}\right)
\right\Vert _{TV}  \leq2\rho^{(l-k)\wedge(n-m)}  \nonumber\\
&\label{extralabelforv510revisiontom3}
\end{align}
where $\rho=\left(  1-\underline{c}_{1}^{2}\big/\overline{c}_{1}^{2}\right)  $.
\end{lemma}

\begin{proof}
Equations (\ref{lemFiltStabStandardeq2}) and (\ref{lemFiltStabeq5}) follow
immediately from standard results in the literature, see for example
\cite{del2004} and \cite{caprydmou2005}.
\end{proof}

\begin{corollary}
Let $\left\{  X_{k},Y_{k}\right\}  $ be a HMM
which satisfies (A3) and let the process $\left\{  X_{k},Y_{k},Y_{k}%
^{\epsilon}\right\}  $ be the corresponding extended HMM defined as in
\eqref{eqextendedHMMdef}.  Then for all $l\leq m$ and infinite sequences $\ldots,j_{-1},j_{0}$ and $\tilde{j}_{0},\tilde{j}_{1},\ldots$ the conditional probability laws $p\left(
X_{l:m}|Y_{j_{-\infty:0}} \right)$ and $p\left(
X_{l:m}|Y_{j_{-\infty:0}};Y_{\tilde{j}_{0:\infty}}^{\epsilon
}\right)$ exist and are well defined.  Further for any $x \in \mathcal{X}$

\begin{equation}
 \left\Vert \mathbb{P} \left(  X_{l:m}|Y_{j_{-k:0}};Y_{\tilde
{j}_{0:n}}^{\epsilon};X_{-k}=x\right) - \mathbb{P} \left(
X_{l:m}|Y_{j_{-\infty:0}};Y_{\tilde{j}_{0:\infty}}^{\epsilon
}\right) \right\Vert_{TV} \to 0 , \label{lemFiltStabeq5}%
\end{equation}
\begin{equation}
 \left\Vert \mathbb{P} \left(  X_{l:m}|Y_{j_{-k:0}};X_{-k}=x\right) - \mathbb{P} \left(
X_{l:m}|Y_{j_{-\infty:0}} \right) \right\Vert_{TV} \to 0 \label{extralabelforv510revisiontom4}
\end{equation}
as $k, n \to \infty$.
\end{corollary}

\begin{proof}
Equations \eqref{lemFiltStabeq5} and \eqref{extralabelforv510revisiontom4} are simple consequences of \eqref{lemFiltStabStandardeq2}.
\end{proof}

\begin{corollary}
\label{lemFiltCondProbUpLowBounds} Let $\left\{  X_{k},Y_{k}\right\}  $ be a
HMM which satisfies (A3) and let $\left\{  X_{k},Y_{k},Y_{k}^{\epsilon
}\right\}  $ be the corresponding extended HMM defined as in
\eqref{eqextendedHMMdef}. Then for all $k<l$, $j_{1}<\cdots<j_{r}$ and
$\tilde{j}_{1}<\cdots<\tilde{j}_{s}$ such that $j_{1}\wedge\tilde{j}_{1}\geq
k$, all $x\in\mathcal{X}$ and all sequences $Y_{j_{1}},\ldots,Y_{j_{r}%
};Y_{\tilde{j}_{1}}^{\epsilon},\ldots,Y_{\tilde{j}_{s}}^{\epsilon}$

\begin{equation}
\frac{\underline{c}_{1}^{3}}{\overline{c}_{1}^{2}}\leq p(x_{l}|Y_{j_{1:r}%
};Y_{\tilde{j}_{1:s}}^{\epsilon};X_{k}=x)\leq\frac{\overline{c}_{1}^{3}%
}{\underline{c}_{1}^{2}} \label{lemFiltCondProbUpLowBoundseq1}%
\end{equation}
where the constants $\underline{c}_{1},\overline{c}_{1}$ are as in (A3) and the
central quantity in \eqref{lemFiltCondProbUpLowBoundseq1} denotes the density
of the corresponding conditional probability with respect to the dominating
measure $\mu$.
\end{corollary}

\begin{proof}
To simplify the exposition we shall only give a proof of \eqref{lemFiltCondProbUpLowBoundseq1} for conditional probabilities of the form $p ( x_{l} \vert Y_{j_{1:r}})$, the proof in the general case following in an identical manner.

It is clear by (A3) that when $j_{r}<l$\
\begin{equation}
\underline{c}_{1}\leq p(x_{l}|Y_{j_{1:r}};X_{k}=x)\leq\overline{c}_{1}
\label{lemFiltCondProbUpLowBoundsEq2}%
\end{equation}
Consider the case when $j_{r}\geq l$. Let $r^{\prime}$ be such that
$j_{r^{\prime}-1}<l\leq j_{r^{\prime}}$. By (A3) we have
\[
p(Y_{j_{r^{\prime}:r}}|X_{l}=x_{l})\leq\frac{\overline{c}_{1}^{2}}%
{\underline{c}_{1}^{2}}p(Y_{j_{r^{\prime}:r}}|X_{l}^{\prime}=x_{l}^{\prime})
\]
for any $x_{l}^{\prime}$ . Note that if $l<j_{r^{\prime}}$ one obtains the
tighter bound $p(Y_{j_{r^{\prime}:r}}|X_{l}=x_{l})\leq(\overline{c}%
_{1}/\underline{c}_{1})p(Y_{j_{r^{\prime}:r}}|X_{l}^{\prime}=x_{l}^{\prime})$.
Thus%
\begin{align*}
p(x_{l}|Y_{j_{1:r}};X_{k}=x)  &  =\frac{p(x_{l}|Y_{j_{1:r^{\prime}-1}}%
;X_{k}=x)p(Y_{j_{r^{\prime}:r}}|X_{l}=x_{l})}{\int p(x_{l}^{\prime
}|Y_{j_{1:r^{\prime}-1}};X_{k}=x)p(Y_{j_{r^{\prime}:r}}|X_{l}=x_{l}^{\prime
})\mu(dx_{l}^{\prime})}\\
&  \leq p(x_{l}|Y_{j_{1:r^{\prime}-1}};X_{k}=x)\frac{\overline{c}_{1}^{2}%
}{\underline{c}_{1}^{2}}%
\end{align*}
and the upper bound in (\ref{lemFiltCondProbUpLowBoundseq1}) is obtained using
(\ref{lemFiltCondProbUpLowBoundsEq2}). The lower bound in
(\ref{lemFiltCondProbUpLowBoundseq1}) is proved similarly.
\end{proof}

\begin{corollary}
\label{lemProdFiltStab} Let $\left\{  X_{k},Y_{k}\right\}  $ be a HMM which
satisfies (A3) and and let $\left\{  X_{k},Y_{k},Y_{k}^{\epsilon}\right\}  $
be the corresponding extended HMM defined as in \eqref{eqextendedHMMdef}. Then
for all $k\leq l\leq l^{\prime}<m\leq m^{\prime}$, $j_{1}<\cdots<j_{r}$ and
$\tilde{j}_{1}<\cdots<\tilde{j}_{s}$ such that $j_{1}\wedge\tilde{j}_{1}\geq
k$, and all $f,h\in L_{\infty}$, $x\in\mathcal{X}$
\begin{align}
\lefteqn{\left\vert \mathbb{E}\left[  f(X_{l:l^{\prime}})|Y_{j_{1:r}%
};Y_{\tilde{j}_{1:s}}^{\epsilon};X_{k}=x\right]  .\mathbb{E}\left[
h(X_{m:m^{\prime}})|Y_{j_{1:r}};Y_{\tilde{j}_{1:s}}^{\epsilon};X_{k}=x\right]
\right.  }\nonumber\\
&  \qquad\qquad\left.  -\mathbb{E}\left[  f(X_{l:l^{\prime}})h(X_{m:m^{\prime
}})|Y_{j_{1:r}};Y_{\tilde{j}_{1:s}}^{\epsilon};X_{k}=x\right]  \right\vert
\leq\left\Vert f\right\Vert _{\infty}\left\Vert h\right\Vert _{\infty}%
\rho^{m-l^{\prime}} \label{lemFiltStabStandardeq3}%
\end{align}
where $\rho$ is as in Lemma \ref{lemFiltStabStandard}.
\end{corollary}

\begin{proof}
Let
\[
\Delta H=h(X_{m:m^{\prime}})-\mathbb{E}\left[  h(X_{m:m^{\prime}})|Y_{j_{1:r}%
};Y_{\tilde{j}_{1:s}}^{\epsilon};X_{k}=x\right]  .
\]
It follows from (\ref{lemFiltStabStandardeq2}) that
\[
\left\vert \mathbb{E}\left[  \Delta H|Y_{j_{1:r}};Y_{\tilde{j}_{1:s}%
}^{\epsilon};X_{k}=x;X_{l}\right]  \right\vert \leq\left\Vert h\right\Vert
_{\infty}\rho^{m-l^{\prime}}.
\]
The proof is completed by noting that the difference of the two expectations
in (\ref{lemFiltStabStandardeq3}) can be expressed as
\[
\mathbb{E}\left[  f(X_{l:l^{\prime}})\Delta H|Y_{j_{1:r}};Y_{\tilde{j}_{1:s}%
}^{\epsilon};X_{k}=x\right]  .
\]

\end{proof}

\begin{rem}
\label{corlemProdFiltStabresextend} The proof of Corollary
\ref{lemProdFiltStab} actually yields the stronger result that the left hand
side of \eqref{lemFiltStabStandardeq3} is bounded above by
\[
\left\Vert h\right\Vert _{\infty}\rho^{m-l^{\prime}}\mathbb{E}\left[
\left\vert f(X_{l:l^{\prime}})\right\vert |Y_{j_{1:r}};Y_{\tilde{j}_{1:s}%
}^{\epsilon};X_{k}=x\right]  .
\]

\end{rem}

\begin{corollary}
\label{corFiltStabtwoended} Let $\left\{  X_{k},Y_{k}\right\}  $ be a HMM
which satisfies (A3) and and let $\left\{  X_{k},Y_{k},Y_{k}^{\epsilon
}\right\}  $ be the corresponding extended HMM defined as in
\eqref{eqextendedHMMdef}. Then for all $k^{\prime}\leq k\leq l<m$,
$j_{1}<\cdots<j_{r}$ and $\tilde{j}_{1}<\cdots<\tilde{j}_{s}$ such that
$j_{1}\wedge\tilde{j}_{1}\geq k^{\prime}$, $f\in L_{\infty}$, $x,x^{\prime}%
\in\mathcal{X}$ and $1\leq r_{b}\leq r_{e}\leq r$, $1\leq s_{b}\leq s_{e}\leq
s$ such that $j_{r_{b}}\wedge\tilde{j}_{s_{b}}\geq k$, $j_{r_{e}}\wedge
\tilde{j}_{s_{e}}\geq m$ and $l\geq j_{r_{b}}\vee\tilde{j}_{s_{b}}$ we have that
\begin{align}
\lefteqn{\left\vert \mathbb{E}\left[  f(X_{l:m})|Y_{j_{1:r}};Y_{\tilde
{j}_{1:s}}^{\epsilon};X_{k^{\prime}}=x^{\prime}\right]  -\mathbb{E}\left[
f(X_{l:m})|Y_{j_{r_{b}:r_{e}}};Y_{\tilde{j}_{s_{b}:s_{e}}}^{\epsilon}%
;X_{k}=x\right]  \right\vert }\nonumber\\
&  \qquad\qquad\qquad\qquad\qquad\qquad\qquad \quad \leq2\left\Vert f\right\Vert
_{\infty}\rho^{(j_{r_{e}}\wedge\tilde{j}_{s_{e}}-m)\wedge(l-j_{r_{b}}%
\vee\tilde{j}_{s_{b}})} \label{corFiltStabtwoendedeq1}%
\end{align}
where $\rho$ is as in Lemma \ref{lemFiltStabStandard}.
\end{corollary}

\begin{proof}
It is clear that $Y_{j_{r_{b}:r_{e}}}\subseteq Y_{j_{1:r}}$ and $Y_{\tilde
{j}_{s_{b}:s_{e}}}^{\epsilon}\subseteq$\ $Y_{\tilde{j}_{1:s}}^{\epsilon}$. By
conditioning on $X_{j_{r_{b}}\vee\tilde{j}_{s_{b}}}$ and $X_{j_{r_{e}}%
\wedge\tilde{j}_{s_{e}}}$, the difference of the two expectations in the left
hand-side of (\ref{corFiltStabtwoendedeq1}) can be expressed as
\begin{align*}
&  \int \left\vert  \mathbb{E}\left[  f(X_{l:m})|Y_{j_{r_{b}:r_{e}}};Y_{\tilde
{j}_{s_{b}:s_{e}}}^{\epsilon};x_{j_{r_{b}}\vee\tilde{j}_{s_{b}}}^{\prime
};x_{j_{r_{e}}\wedge\tilde{j}_{s_{e}}}^{\prime}\right]  \right. \\
& \qquad \qquad \qquad \qquad \qquad  \left.  -\mathbb{E}\left[  f(X_{l:m})|Y_{j_{r_{b}:r_{e}}};Y_{\tilde
{j}_{s_{b}:s_{e}}}^{\epsilon};x_{j_{r_{b}}\vee\tilde{j}_{s_{b}}};x_{j_{r_{e}%
}\wedge\tilde{j}_{s_{e}}}\right]  \right\vert \\
&  \times p\left(  x_{j_{r_{b}}\vee\tilde{j}_{s_{b}}}^{\prime},x_{j_{r_{e}%
}\wedge\tilde{j}_{s_{e}}}^{\prime}|Y_{j_{1:r}};Y_{\tilde{j}_{1:s}}^{\epsilon
};X_{k^{\prime}}=x^{\prime}\right) \\
& \qquad \qquad \qquad \qquad \qquad \times p\left(  x_{j_{r_{b}}\vee\tilde{j}_{s_{b}}},x_{j_{r_{e}}\wedge
\tilde{j}_{s_{e}}}|Y_{j_{r_{b}:r_{e}}};Y_{\tilde{j}_{s_{b}:s_{e}}}^{\epsilon
};X_{k}=x\right) \\
& \qquad  \times\mu(dx_{j_{r_{b}}\vee\tilde{j}_{s_{b}}}^{\prime})\mu(dx_{j_{r_{e}%
}\wedge\tilde{j}_{s_{e}}}^{\prime})\mu(dx_{j_{r_{b}}\vee\tilde{j}_{s_{b}}})\mu(dx_{j_{r_{e}}\wedge\tilde{j}_{s_{e}}}).
\end{align*}
The result now follows by bounding the difference of the two conditional expectations in the integrand using \eqref{extralabelforv510revisiontom3}.
\end{proof}

\begin{rem}
\label{rewtwocorsfnalt} Using exactly the same proofs as above one can show
that the conclusions of Corollaries \ref{lemProdFiltStab} and
\ref{corFiltStabtwoended} and Remark \ref{corlemProdFiltStabresextend} are
still valid if the functions $f(X_{l:l^{\prime}})$, $h(X_{m:m^{\prime}})$ and
$f(X_{l:m})$ in the statements of those results are replaced with the
functions $f(X_{l:l^{\prime}},Y_{l:l^{\prime}})$, $h(X_{m:m^{\prime}%
},Y_{m:m^{\prime}})$, $f(X_{l:m},Y_{l:m})$.
\end{rem}

The next result establishes certain properties of the gradient of the filter
conditioned on the infinite past, see \cite{legmev2000} or \cite{taddou2005}
for further information concerning the gradient of the filter.

\begin{lemma}
\label{corGradDiffRate} Let $\left\{  X_{k} , Y_{k} \right\}  $ be a
parameterised collection of HMMs which satisfy (A3)-(A5) and let $\left\{
X_{k} , Y_{k}, Y^{\epsilon}_{k} \right\}  $ be the corresponding extended HMMs
defined as in \eqref{eqextendedHMMdef}. Then for all $\theta\in G$ where $G$
is as in assumptions (A4) and (A5) and every sequence of observations $\ldots,
Y_{-1} ; Y^{\epsilon}_{1} , \ldots$ there exists an $\mathbb{R}^{d}$ valued
function $\bar{\nabla_{\theta}} p_{\theta; Y_{-\infty:-1} ; Y^{\epsilon}_{1:\infty} }
\left(  x_{0} \right)  $ in $L_{1} ( \mu)$ such that such that for all $k, n >
0$, $x \in\mathcal{X}$
\begin{align}
\lefteqn{ \sup_{f : \left\Vert f \right\Vert _{\infty} \leq1} \left\vert \int
f(x_{0}) \bar{ \nabla_{\theta}} p_{\theta; Y_{-\infty:-1} ; Y^{\epsilon}_{1:\infty} }
\left(  x_{0} \right)  \mu( d x_{0} ) \right.  }\nonumber\\
&  \qquad\left.  - \int f(x_{0}) \nabla_{\theta} p_{\theta} \left(  x_{0} \vert
Y_{-n:-1} ; Y^{\epsilon}_{1:k} ; X_{-n} = x \right)  \mu( d x_{0} )
\right\vert \leq C \rho^{ \frac{n}{2} \wedge\frac{k}{2} }
\label{lemGradFiltStabeq201}%
\end{align}
where $\rho$ is as in Lemma \ref{lemFiltStabStandard}, $C < \infty$ is a
global constant independent of $\theta$ and $\ldots, Y_{-1} ; Y^{\epsilon}_{1}
, \ldots$ and $\nabla_{\theta} p_{\theta} \left(  x_{0} \vert Y_{-n:-1} ; Y^{\epsilon
}_{1:k} ; X_{-n} = x \right)  $ denotes the gradient of the density of the
conditional law $\mathbb{P}_{\theta} \left(  x_{0} \vert Y_{-n:-1} ;
Y^{\epsilon}_{1:k} ; X_{-n} = x \right)  $ w.r.t. $\mu$.

Furthermore there exists $K<\infty$ such that for all $k,n>0$, $x$ and
$\theta\in G$
\begin{equation}
\nabla_{\theta} p_{\theta}\left(  x_{0} \vert Y_{-n:-1};Y_{1:k}^{\epsilon}%
;X_{-n}=x\right)  ,\bar{\nabla_{\theta}}p_{\theta; Y_{-\infty:-1};Y_{1:\infty
}^{\epsilon} }\left(  x_{0}\right)  \leq K \label{lemGradFiltStabeq4}%
\end{equation}
almost surely. Finally we have that for any $f\in L_{\infty}$
\begin{align}
\lefteqn{ \nabla_{\theta}\int f(x_{0})p_{\theta}\left(  x_{0} \vert Y_{-\infty
:-1};Y_{1:\infty}^{\epsilon} \right)  \mu(dx_{0}) }\nonumber\\
&  \qquad\qquad=\int f(x_{0})\bar{\nabla_{\theta}}p_{\theta; Y_{-\infty:-1}%
;Y_{1:\infty}^{\epsilon},\ldots}\left(  x_{0}\right)  \mu(dx_{0}),
\label{lemGradFiltStabeq305}%
\end{align}
where (\ref{lemGradFiltStabeq305}) defines a continuous function of $\theta$
on $G$.
\end{lemma}

\begin{proof}
We begin by proving (\ref{lemGradFiltStabeq201}) and (\ref{lemGradFiltStabeq4}%
). First note that since it is sufficient to prove the results component wise
with respect to the vectors $\nabla_{\theta} p_{\theta} ( \cdot\vert\cdots)$ and $\bar{
\nabla_{\theta}} p_{\theta} ( \cdot\vert\cdots)$ we can assume that $d = 1$. For any
suitable $x$, $f$, $n$ and $k$
\begin{align}
\lefteqn{ \int f(x_{0}) \nabla_{\theta} p_{\theta} \left(  d
x_{0} \vert Y_{-n:-1} ; Y^{\epsilon}_{1:k} ; X_{-n} = x \right)  }\nonumber\\
&  = \sum_{j=-n}^{-1} \mathbb{E} \left[  f(X_{0}) \nabla_{\theta}\log\left(  q_{\theta
}(X_{j}, X_{j+1}) g_{\theta}(Y_{j} \vert X_{j}) \right)  \vert Y_{-n:-1} ;
Y^{\epsilon}_{1:k} ; X_{-n} = x \right] \nonumber \\ \label{lemGradFiltStabeq5}\\
&  - \sum_{j=-n}^{-1} \mathbb{E}_{\theta} \left[  f(X_{0}) \vert Y_{-n:-1} ;
Y^{\epsilon}_{1:k} ; X_{-n} = x \right]  \times\nonumber\\
&  \qquad\qquad\qquad\mathbb{E}_{\theta} \left[  \nabla_{\theta}\log\left(  q_{\theta
}(X_{j}, X_{j+1}) g_{\theta}(Y_{j} \vert X_{j}) \right)  \vert Y_{-n:-1} ;
Y^{\epsilon}_{1:k} ; X_{-n} = x \right] \label{lemGradFiltStabeq6}\\
\end{align}

\begin{align}
&  + \sum_{l=1}^{k} \mathbb{E}_{\theta} \left[  f(X_{0}) \nabla_{\theta}\log\left(
q_{\theta}(X_{l-1}, X_{l}) g_{\theta}(Y_{l} \vert X_{l}) \right)  \vert
Y_{-n:-1} ; Y^{\epsilon}_{1:k} ; X_{-n} = x \right] \nonumber \\  \label{lemGradFiltStabeq7}\\
&  - \sum_{l=1}^{k} \mathbb{E}_{\theta} \left[  f(X_{0}) \vert Y_{-n:-1} ;
Y^{\epsilon}_{1:k} ; X_{-n} = x \right]  \times\nonumber\\
&  \qquad\qquad\qquad\mathbb{E}_{\theta} \left[  \nabla_{\theta}\log\left(  q_{\theta
}(X_{l-1}, X_{l}) g_{\theta}(Y_{l} \vert X_{l}) \right)  \vert Y_{-n:-1} ;
Y^{\epsilon}_{1:k} ; X_{-n} = x \right]. \label{lemGradFiltStabeq8}%
\end{align}
By (A3), (A5), \eqref{corFiltStabtwoendedeq1} and Remark \ref{rewtwocorsfnalt}
we have that for all $f : \left\Vert f \right\Vert _{\infty} \leq1$, $x,
x^{\prime} \in\mathcal{X}$, $\theta\in G$, $k, k^{\prime}, n, n^{\prime} > 0$
and $j$ such that $-n^{\prime} \leq-n < j < k \leq k^{\prime}$ that
\begin{align}
\lefteqn{ \left\vert \mathbb{E} \left[  f(X_{0}) \nabla_{\theta}\log\left(  q_{\theta
}(X_{j}, X_{j+1}) g_{\theta}(Y_{j} \vert X_{j}) \right)  \vert Y_{-n:-1} ;
Y^{\epsilon}_{1:k} ; X_{-n} = x \right]  \right.  }  &  & \nonumber\\
&  & \left.  - \mathbb{E} \left[  f(X_{0}) \nabla_{\theta}\log\left(  q_{\theta}(X_{j},
X_{j+1}) g_{\theta}(Y_{j} \vert X_{j}) \right)  \vert Y_{-n^{\prime}:-1} ;
Y^{\epsilon}_{1:k^{\prime}} ; X_{-n^{\prime}} = x^{\prime} \right]
\right\vert \nonumber\\
&  & \leq\frac{2 \overline{c}_{1} \overline{c}_{2}}{\underline{c}_{1}} C
\rho^{(j+n) \wedge(k-j-1)} \label{lemGradFiltStabeq101}%
\end{align}
and
\begin{align}
\lefteqn{\left\vert \mathbb{E} \left[  f(X_{0}) \vert Y_{-n:-1} ; Y^{\epsilon
}_{1:k} ; X_{-n} = x \right]  \times\right.  }\nonumber\\
&  \qquad\mathbb{E} \left[  \nabla_{\theta}\log\left(  q_{\theta}(X_{j}, X_{j+1})
g_{\theta}(Y_{j} \vert X_{j}) \right)  \vert Y_{-n:-1} ; Y^{\epsilon}_{1:k} ;
X_{-n} = x \right] \nonumber\\
&  - \mathbb{E} \left[  f(X_{0}) \vert Y_{-n^{\prime}:-1} ; Y^{\epsilon
}_{1:k^{\prime}} ; X_{-n^{\prime}} = x^{\prime} \right]  \times\nonumber\\
&  \qquad\left.  \mathbb{E} \left[  \nabla_{\theta}\log\left(  q_{\theta}(X_{j},
X_{j+1}) g_{\theta}(Y_{j} \vert X_{j}) \right)  \vert Y_{-n^{\prime}:-1} ;
Y^{\epsilon}_{1:k^{\prime}} ; X_{-n^{\prime}} = x^{\prime} \right]
\right\vert .\nonumber\\
&  \qquad\qquad\qquad\qquad\qquad\qquad\qquad\leq4C \frac{\overline{c}_{1}
\overline{c}_{2}}{\underline{c}_{1}} \left(  1 + C \frac{\overline{c}_{1}%
}{\underline{c}_{1}} \right)  \rho^{(j+n) \wedge(k-j-1)}
\label{lemGradFiltStabeq202}%
\end{align}
where $\rho$ is as in Lemma \ref{lemFiltStabStandard}, $C$ is as in Corollary
\ref{corFiltStabtwoended} and $\underline{c}_{1}, \overline{c}_{1},
\overline{c}_{2}$ are as in assumption (A3) and (A5). Further by (A3), (A5)
and \eqref{lemFiltStabStandardeq3} it follows that for all $x \in\mathcal{X}$,
$\theta\in G$, $k, n > 0$ and $j \neq0$ that
\begin{align}
\lefteqn{ \left\vert \mathbb{E} \left[  f(X_{0}) \nabla_{\theta}\log\left(  q_{\theta
}(X_{j}, X_{j+1}) g_{\theta}(Y_{j} \vert X_{j}) \right)  \vert Y_{-n:-1} ;
Y^{\epsilon}_{1:k} ; X_{-n} = x \right]  \right.  }\nonumber\\
&  - \mathbb{E} \left[  f(X_{0}) \vert Y_{-n:-1} ; Y^{\epsilon}_{1:k} ; X_{-n}
= x \right]  \times\nonumber\\
&  \qquad\qquad\left.  \mathbb{E} \left[  \nabla_{\theta}\log\left(  q_{\theta}(X_{j},
X_{j+1}) g_{\theta}(Y_{j} \vert X_{j}) \right)  \vert Y_{-n:-1} ; Y^{\epsilon
}_{1:k} ; X_{-n} = x \right]  \right\vert \nonumber\\
&  \qquad\qquad\qquad\qquad\qquad\qquad\qquad\qquad\leq\frac{2\overline{c}_{1}
\overline{c}_{2}}{\underline{c}_{1}} \rho^{\left\vert j \right\vert
\wedge\left\vert j+1 \right\vert } . \label{lemGradFiltStabeq102}%
\end{align}
It thus follows from \eqref{lemGradFiltStabeq5}-\eqref{lemGradFiltStabeq102}
that for all $\theta\in G$ that for all $k,n \geq1$
\begin{align}
\lefteqn{ \sup_{x, x^{\prime} \in\mathcal{X}} \sup_{f : \left\Vert f
\right\Vert _{\infty} \leq1} \left\vert \int f(x_{0}) \nabla_{\theta} p_{\theta}
\left(  x_{0} \vert Y_{-n:-1} ; Y^{\epsilon}_{1:k} ; X_{-n} = x \right)  \mu(
d x_{0} ) \right.  }\nonumber\\
&  \qquad\qquad\left.  - \int f(x_{0}) \nabla_{\theta} p_{\theta} \left(  x_{0} \vert
Y_{-n^{\prime}:-1} ; Y^{\epsilon}_{1:k^{\prime}} ; X_{-n^{\prime}} =
x^{\prime} \right)  \mu( d x_{0} ) \right\vert \nonumber\\
&  \qquad\qquad\qquad\qquad\qquad\qquad\qquad\qquad\qquad\leq64 C^{2}
\frac{\overline{c}_{1}^{2} \overline{c}_{2}}{\underline{c}_{1}^{2} \rho}
\sum_{r = \frac{n}{2} \wedge\frac{k}{2}}^{\infty} \rho^{r} .
\label{lemGradFiltStabeq1}%
\end{align}
Further the first part of \eqref{lemGradFiltStabeq4} follows from
\eqref{lemGradFiltStabeq5}-\eqref{lemGradFiltStabeq8}, (A3) and (A5), the
uniform boundedness of the densities of conditional probability densities
$p_{\theta} \left(  x_{0} \vert Y_{-n:-1} ; Y^{\epsilon}_{1:k} ; X_{-n} = x
\right)  $ (Corollary \ref{lemFiltCondProbUpLowBounds}) and Remark
\ref{corlemProdFiltStabresextend}. Let $K$ be the constant bounding the first
part of \eqref{lemGradFiltStabeq4} and for any $x\in\mathcal{X}$, $k,n\geq0$
and observations $Y_{-n},\ldots,Y_{-1};Y_{1}^{\epsilon},\ldots,Y_{k}%
^{\epsilon}$ let
\begin{align}
\label{lemGradFiltStabeq302}\nabla_{\theta} p_{\theta}^{K}\left(  x_{0}|Y_{-n:-1}%
;Y_{1:k}^{\epsilon};X_{-n}=x\right)  =\nabla_{\theta} p_{\theta}\left(  x_{0}%
|Y_{-n:-1};Y_{1:k}^{\epsilon};X_{-n}=x\right)  +K.\nonumber\\
\end{align}
The functions $\nabla_{\theta} p_{\theta}^{K}(\cdot|\cdots)$ are densities with respect
to $\mu$ of a collection of (random) finite positive measures, each with total
mass equal to $K$ and for which \eqref{lemGradFiltStabeq1} clearly still
holds. Since the space of positive finite measures equipped with the total
variation norm is a Banach space (see e.g.~\cite{parste1985}) it follows from
\eqref{lemGradFiltStabeq1} that given a doubly infinite sequence of
observations $\ldots,Y_{-1};Y_{1}^{\epsilon},\ldots$ there exists some
positive finite measure $\bar{\mu}_{Y_{-\infty:-1};Y_{1:\infty}^{\epsilon}%
}^{K}$ such that for any $n \geq1$
\begin{align}
\lefteqn{\sup_{x\in\mathcal{X}}\sup_{f:\left\Vert f\right\Vert _{\infty}\leq
1}\left\vert \int f(x_{0})\nabla_{\theta} p_{\theta}^{K}\left(  x_{0}|Y_{-n:-1}%
;Y_{1:n}^{\epsilon};X_{-n}=x\right)  \mu(dx_{0}) \right.  }\nonumber\\
&  \qquad\qquad\left.  - \int f(x_{0})\bar{\mu}_{Y_{-\infty:-1};Y_{1:\infty
}^{\epsilon}}^{K}(dx_{0})\right\vert \leq64 C^{2} \frac{\overline{c}_{1}^{2}
\overline{c}_{2}}{\underline{c}_{1}^{2} \rho(1 - \rho) } \rho^{ \frac{n}{2} }.
\label{lemGradFiltStabeq301}%
\end{align}
It follows by definition that $\nabla_{\theta} p_{\theta}^{K}(\cdot|\cdots)\leq2K$ and
thus from \eqref{lemGradFiltStabeq301} that $\bar{\mu}_{Y_{-\infty
:-1};Y_{1:\infty}^{\epsilon}}^{K}\ll\mu$ and that its density $\bar{\nabla_{\theta}
}p_{\theta; Y_{-\infty:-1};Y_{1:\infty}^{\epsilon}}^{K}(\cdot)$ is bounded
above by $2K$. Equation (\ref{lemGradFiltStabeq201}) and the second part of
equation (\ref{lemGradFiltStabeq4}) now follow by letting $\bar{\nabla_{\theta}
}p_{\theta; Y_{-\infty:-1};Y_{1:\infty}^{\epsilon}}(\cdot)=\bar{\nabla_{\theta}
}p_{\theta; Y_{-\infty:-1};Y_{1:\infty}^{\epsilon}}^{K}(\cdot)-K$. We shall
prove \eqref{lemGradFiltStabeq305} by, for any $f \in L_{\infty}$, applying
Lemma \ref{lemrealanalresstand} to the sequence of functions
\begin{equation}
\label{lemGradFiltStabeq16}\bar{\mathbb{E}}_{\theta}\left[  f(X_{0}%
)|Y_{-n:-1};Y_{1:n}^{\epsilon};X_{-n}=x\right]
\end{equation}
for $n \geq1$ and $x \in\mathcal{X}$ arbitrary. Clearly the sequence of
functions in \eqref{lemGradFiltStabeq16} are continuously differentiable by
$(A2)$ and $(A4)$. In order to be able to apply Lemma
\ref{lemrealanalresstand} we further need to establish that the functions in
\eqref{lemGradFiltStabeq16} and their derivatives are uniformly bounded. This
follows from \eqref{lemFiltCondProbUpLowBoundseq1} and
\eqref{lemGradFiltStabeq4}, that
\[
\bar{\mathbb{E}}_{\theta}\left[  f(X_{0})|Y_{-n:-1};Y_{1:n}^{\epsilon}%
;X_{-n}=x\right]  \to\bar{\mathbb{E}}_{\theta}\left[  f(X_{0}) \vert
Y_{-\infty:-1};Y_{1:\infty}^{\epsilon} \right]
\]
uniformly which follows from \eqref{lemFiltStabeq5} and finally that the
sequence of derivatives of the functions in \eqref{lemGradFiltStabeq16} is
uniformly Cauchy which follows from \eqref{lemGradFiltStabeq201}.
\end{proof}

\begin{corollary}
\label{remcorGradDiffRateextension}
Assume the same conditions as in Lemma \ref{corGradDiffRate}.  Then results analogous to
those in \eqref{lemGradFiltStabeq201}, \eqref{lemGradFiltStabeq4} and
\eqref{lemGradFiltStabeq305} hold for the gradients of the conditional densities $p_{\theta}\left(
x_{0}|Y_{-n:-1};X_{-n}=x\right)$, $p_{\theta}^{\epsilon}\left(
x_{0}|Y_{-n:-1}^{\epsilon};X_{-n}=x\right)$ and $p_{\theta}\left(  x_{0}|Y_{-\infty:-1};Y_{m:\infty}^{\epsilon}\right)$. Furthermore for every sequence of observations
$Y_{-\infty:-1}$, $Y_{1:\infty}^{\epsilon}$ and integer $m \geq 1$
\begin{equation}
\left\vert \nabla_{\theta} p_{\theta}\left(  x_{0}|Y_{-\infty:-1};Y_{m:\infty
}^{\epsilon}\right)  -\nabla_{\theta} p_{\theta}\left(  x_{0}|Y_{-\infty:-1}\right)
\right\vert \leq C\rho^{\frac{m}{2}}\label{remcorGradDiffRateextensionEq}%
\end{equation}
where $\rho$ is as in Lemma \ref{lemFiltStabStandard} and $C<\infty$ is a
global constant independent of $\theta$, $x_{0}$, $Y_{-\infty:-1}$,
$Y_{1:\infty}^{\epsilon}$ and $m$.
\end{corollary}

\begin{proof}
The first part of the corollary can be proved in exactly the same way as Lemma \ref{corGradDiffRate}.

To prove the second part of the corollary it is sufficient to show that
\begin{equation}
\left\vert \nabla_{\theta} p_{\theta}\left(  x_{0}|Y_{-n:-1};Y_{m:n
}^{\epsilon} ; X_{-n}=x\right)  -\nabla_{\theta} p_{\theta}\left(  x_{0}|Y_{-n:-1} ; X_{-n}=x\right)
\right\vert \leq C\rho^{\frac{m}{2} \wedge \frac{n}{2}} \label{extralabelforv510revisiontom10}
\end{equation}
for some $C$ and all $\theta$, $Y_{-\infty:-1}$,
$Y_{1:\infty}^{\epsilon}$, $x$ and $n$.  Inequality \eqref{extralabelforv510revisiontom10} can be established by decomposing the two gradients that appear on its left hand side in an analogous manner to \eqref{lemGradFiltStabeq5}-\eqref{lemGradFiltStabeq8}.  The bound on the right hand side then follows by bounding the terms in this decomposition individually using \eqref{lemGradFiltStabeq101}-\eqref{lemGradFiltStabeq102} and the fact that
\begin{equation}  \label{extralabelforv510revisiontom11}
\mathbb{P}_{\theta}\left(  x_{0}|Y_{-n:-1};Y_{m:n
}^{\epsilon} ; X_{-n}=x\right)  - \mathbb{P}_{\theta}\left(  x_{0}|Y_{-n:-1} ; X_{-n}=x\right)
 \leq 2 \rho^{m}
\end{equation}
for all $\theta$, $Y_{-\infty:-1}$,
$Y_{1:\infty}^{\epsilon}$, $x$ and $n$ which follows immediately from \eqref{extralabelforv510revisiontom3}.
\end{proof}

\section*{Appendix B: Proofs of Lemmas \ref{lemexpllhcont},
\ref{lemABCGradApprox}, \ref{asymmissinf} and \ref{lemABCMLEcomp}}

\begin{proof}
[Proof of Lemma \ref{lemexpllhcont}]We begin by observing that a
straightforward consequence of assumption (A3) is that for all $(\theta,
\epsilon) \in\Theta\times\lbrack0,\infty)$, $r > 0$ and sequences $y_{-r},
\ldots,y_{1}$
\begin{equation}
\label{thmprobconvpfeq1}\underline{c}_{1} \leq p_{\theta}^{\epsilon}(y_{1}
\vert y_{-r} \ldots, y_{0}) \leq\overline{c}_{1} .
\end{equation}
Further by Lemma \ref{lemFiltStabStandard} it follows that the finite history
conditional likelihoods $p^{\epsilon}_{\theta} ( y_{1} \vert y_{-r} , \ldots,
y_{0})$ converge to the infinite history conditional likelihoods $p^{\epsilon
}_{\theta} ( y_{1} \vert\ldots, y_{0})$ as $r \to\infty$ uniformly in $\theta
$, $\epsilon$, the sequence of observations $\ldots, y_{0}, y_{1}$ and initial
distribution $\pi(x_{-r-1})$. Thus by \eqref{thmprobconvpfeq1} it follows that
in order to show continuity w.r.t. the first term and right continuity w.r.t.
the second term of the mapping $(\theta,\epsilon)\in\Theta\times
\lbrack0,\infty)\rightarrow l^{\epsilon}(\theta)$ it is sufficient to show
that these properties hold for the mapping
\begin{equation}
\label{thmprobconvpfeq2}(\theta,\epsilon)\in\Theta\times\lbrack0,\infty)
\rightarrow\bar{\mathbb{E}}_{\theta^{\ast}}\left[  \log p_{\theta}^{\epsilon
}(Y_{1} \vert Y_{-r:0}) \right]
\end{equation}
for all $r > 0$. For the rest of the proof we shall assume an arbitrary fixed
$r >0$ and initial distribution $\pi(x_{-r-1})$ are given. Observe that by
(A2), (A3) Lemma \ref{lemnullprobprob} and the dominated convergence theorem
the mapping $(\theta,\epsilon)\in\Theta\times( 0,\infty)\rightarrow
g^{\epsilon}_{\theta} ( y \vert x)$ is continuous w.r.t. its first term and
right continuous w.r.t. its second term for all $y \in\mathcal{Y}$ and $x
\in\mathcal{X}$. Thus by a second application of (A2), (A3) and the dominated
convergence theorem one immediately obtains these properties of the mapping
$(\theta,\epsilon)\in\Theta\times(0,\infty)\rightarrow p_{\theta}^{\epsilon
}(y_{1} \vert y_{-r} \ldots, y_{0})$ for any $r > 0$ and sequence $y_{-r},
\ldots,y_{1}$. A final application of (A2), (A3) and the dominated convergence
theorem along with the inequality \eqref{thmprobconvpfeq1} yield that the
mapping $\Theta\times\lbrack0,\infty) \rightarrow\mathbb{R}$ given in
\eqref{thmprobconvpfeq2} is also respectively continuous and right continuous.
In order to prove continuity w.r.t. the first term and right continuity w.r.t.
the second term of \eqref{thmprobconvpfeq2} on $\Theta\times[ 0,\infty)$ we
shall show for any sequences $\epsilon_{n} \searrow0$ and $\theta_{n}
\in\Theta\to\theta\in\Theta$, that
\begin{equation}
\label{thmprobconvpfeq3}\bar{\mathbb{E}}_{\theta^{\ast}}\left[  \log
p_{\theta_{n}}^{\epsilon_{n}}(Y_{1} \vert Y_{-r:0}) \right]  \to
\bar{\mathbb{E}}_{\theta^{\ast}}\left[  \log p_{\theta} (Y_{1} \vert Y_{-r:0})
\right]
\end{equation}
as $n \to\infty$. First note that
\[
\bar{\mathbb{E}}_{\theta^{\ast}}\left[  \log p_{\theta_{n}}^{\epsilon_{n}%
}(Y_{1} \vert Y_{-r:0}) \right]  = \bar{\mathbb{E}}_{\theta^{\ast}}\left[
\log p_{\theta_{n}}^{\epsilon_{n}} (Y_{-r}, \ldots,Y_{1}) - \log p_{\theta_{n}
}^{\epsilon_{n}} (Y_{-r}, \ldots,Y_{0}) \right]  .
\]
Thus in order to prove \eqref{thmprobconvpfeq3} it is sufficient to show that
\begin{equation}
\label{thmprobconvpfeq4}\bar{\mathbb{E}}_{\theta^{\ast}}\left[  \log
p_{\theta_{n} }^{\epsilon_{n}} (Y_{-r}, \ldots,Y_{1}) \right]  \to
\bar{\mathbb{E}}_{\theta^{\ast}}\left[  \log p_{\theta} (Y_{-r}, \ldots,Y_{1})
\right]
\end{equation}
and
\begin{equation}
\label{thmprobconvpfeq5}\bar{\mathbb{E}}_{\theta^{\ast}}\left[  \log
p_{\theta_{n} }^{\epsilon_{n}} (Y_{-r}, \ldots,Y_{0}) \right]  \to
\bar{\mathbb{E}}_{\theta^{\ast}}\left[  \log p_{\theta} (Y_{-r}, \ldots,Y_{0})
\right]
\end{equation}
as $n \to\infty$. We will now conclude the proof of the theorem by proving
\eqref{thmprobconvpfeq4} and observing that the proof of
\eqref{thmprobconvpfeq5} follows in a completely identical manner. The
differences in the values of the likelihoods in \eqref{thmprobconvpfeq4}
evaluated at different parameter values $\theta_{n}$ and $\theta$ can be
bounded by
\begin{align}
\lefteqn{ \left\vert p_{\theta_{n} }^{\epsilon_{n}} (Y_{-r}, \ldots,Y_{1}) - p_{\theta}^{\epsilon_{n}} (Y_{-r}, \ldots,Y_{1}) \right\vert } \notag \\
& \leq \bigg\vert \int_{\mathcal{X}^{r+3}} \bigg[\prod_{k=-r}^{1}
q_{\theta_{n}} (x_{k-1}, x_{k}) g^{\epsilon_{n}}_{\theta_{n}} (Y_{k} \vert x_{k}) \bigg] \pi (d x_{-r-1}) \mu( d x_{-r:1}) \notag \\
& \qquad \qquad - \int_{\mathcal{X}^{r+3}} \bigg[\prod_{k=-r}^{1}
q_{\theta} (x_{k-1}, x_{k}) g_{\theta_{n}}^{\epsilon_{n}} (Y_{k} \vert x_{k}) \bigg] \pi (d x_{-r-1}) \mu( d x_{-r:1}) \bigg\vert \notag \\
& + \bigg\vert \int_{\mathcal{X}^{r+3}} \bigg[\prod_{k=-r}^{1}
q_{\theta} (x_{k-1}, x_{k}) g^{\epsilon_{n}}_{\theta_{n}} (Y_{k} \vert x_{k}) \bigg] \pi (d x_{-r-1}) \mu( d x_{-r:1})  \notag \\
& \qquad \qquad - \int_{\mathcal{X}^{r+3}} \bigg[\prod_{k=-r}^{1}
q_{\theta} (x_{k-1}, x_{k}) g_{\theta}^{\epsilon_{n}} (Y_{k} \vert x_{k}) \bigg] \pi_{\theta} (d x_{-r-1}) \mu( d x_{-r:1}) \bigg\vert \notag \\
& \leq \overline{c}_{1}^{r+2}  \int_{\mathcal{X}^{r+3}}  \Bigg\vert \prod_{k=-r}^{1}
q_{\theta_{n}} (x_{k-1}, x_{k}) - \prod_{k=-r}^{1}
q_{\theta} (x_{k-1}, x_{k})  \Bigg\vert \pi (d x_{-r-1}) \mu( d x_{-r:1}) \notag \\
 \label{thmprobconvpfeq6} \\
& + \overline{c}_{1}^{r+1} \Bigg( \sum_{l = -r}^{1}  \int_{\mathcal{X}^{r+3}}    \left\vert g_{\theta_{n}}^{\epsilon_{n}} (Y_{l} \vert x_{l}) - g_{\theta}^{\epsilon_{n}} (Y_{l} \vert x_{l}) \right\vert  \pi (d x_{-r-1}) \mu( d x_{-r:1}) \Bigg) \notag \\
 \label{thmprobconvpfeq7}
\end{align}
where $\overline{c}_{1}$ is as in (A3) and \eqref{thmprobconvpfeq7} follows
from (A3), the definition of $g^{\epsilon}_{\theta} (\cdot\vert\cdot)$ and the
telescopic identity
\[
\prod_{k=1}^{n} a_{k} b_{k} - \prod_{k=1}^{n} a_{k} \hat{b}_{k} = \sum
_{l=1}^{n} \left(  \prod_{k=1}^{n} a_{k} \times( b_{l} - \hat{b}_{l})
\prod_{k=1}^{l-1} b_{k} \prod_{k=l+1}^{n} \hat{b}_{k} \right)
\]
which holds for any collection of reals $a_{1}, \ldots, a_{n};b_{1} ,
\ldots, b_{n}; \hat{b}_{1} , \ldots, \hat{b}_{n}$. Clearly assumptions (A2)
and (A3) and the dominated convergence theorem imply that the quantity in
\eqref{thmprobconvpfeq6} converges to zero as $\theta_{n}$ converges to
$\theta$. Furthermore the definition of $g^{\epsilon}_{\theta} (\cdot
\vert\cdot)$ and convergence of $\theta_{n}$ to $\theta$ imply that for any
$\delta> 0$ the limit supremum of \eqref{thmprobconvpfeq7} as $n \to\infty$ is
bounded above by
\begin{align}
\label{thmprobconvpfeq8}\limsup_{n \to\infty} \overline{c}_{1}^{r+1} \left(
\sum_{l = -r}^{1} \int_{\mathcal{X}^{r+3}} \frac{ \int_{B_{Y_{l}}%
^{\epsilon_{n}}} \Delta g^{\delta}_{\theta} (y \vert x_{l}) \nu( dy )}%
{\nu(B_{Y_{l}}^{\epsilon_{n}})} \pi(d x_{-r-1}) \mu( d x_{-r:1}) \right)
\nonumber\\
\end{align}
where for all $x \in\mathcal{X}$ and $y \in\mathcal{Y}$
\[
\Delta g^{\delta}_{\theta} (y \vert x) = \sup_{\theta^{\prime} : \left\vert
\theta^{\prime} - \theta\right\vert \leq\delta} \left\vert g_{\theta^{\prime}}
(y \vert x) - g_{\theta} (y \vert x) \right\vert .
\]
It then follows from (A3), \eqref{thmprobconvpfeq8}, the dominated convergence
theorem and the Lebesgue differentiation theorem (see \cite{whezyg1977},
Chapter 10) that
\begin{align}
\lefteqn{ \limsup_{n \to\infty} \overline{c}_{1}^{r+1} \left(  \sum_{l =
-r}^{1} \int_{\mathcal{X}^{r+3}} \left\vert g_{\theta_{n}}^{\epsilon_{n}}
(Y_{l} \vert x_{l}) - g_{\theta}^{\epsilon_{n}} (Y_{l} \vert x_{l})
\right\vert \pi(d x_{-r-1}) \mu( d x_{-r:1}) \right)  }\nonumber\\
&  \qquad\qquad\qquad\qquad\leq\overline{c}_{1}^{r+1} \left(  \sum_{l =
-r}^{1} \int_{\mathcal{X}^{r+3}} \Delta g^{\delta}_{\theta} (Y_{l} \vert
x_{l}) \pi(d x_{-r-1}) \mu( d x_{-r:1}) \right)  \label{thmprobconvpfeq9}%
\end{align}
for any $\delta> 0$. Next we observe that by (A2) we have that $\lim
_{\delta\to0} \Delta g^{\delta}_{\theta} (y \vert x) = 0$ for all $y
\in\mathcal{Y}$ and $x \in\mathcal{X}$ and hence that by applying (A3) and the
dominated convergence theorem to \eqref{thmprobconvpfeq9} we have that
\begin{align}
\label{thmprobconvpfeq10}\limsup_{n \to\infty} \overline{c}_{1}^{r+1} \left(
\sum_{l = -r}^{1} \int_{\mathcal{X}^{r+3}} \left\vert g_{\theta_{n}}%
^{\epsilon_{n}} (Y_{l} \vert x_{l}) - g_{\theta}^{\epsilon_{n}} (Y_{l} \vert
x_{l}) \right\vert \pi(d x_{-r-1}) \mu( d x_{-r:1}) \right)  = 0 .\nonumber\\
\end{align}
Thus it follows from \eqref{thmprobconvpfeq1}, \eqref{thmprobconvpfeq6},
\eqref{thmprobconvpfeq7} and \eqref{thmprobconvpfeq10} that for almost all
sequences of observations $Y_{-r}, \ldots, Y_{1}$ that
\begin{align}
\lim_{n \to\infty} \log p_{\theta_{n} }^{\epsilon_{n}} (Y_{-r}, \ldots,Y_{1})
&  = \lim_{n \to\infty} \log p_{\theta}^{\epsilon_{n}} (Y_{-r}, \ldots,Y_{1})
. \label{thmprobconvpfeq11}%
\end{align}
Since
\begin{align*}
\lefteqn{ p_{\theta}^{\epsilon} (Y_{-r}, \ldots,Y_{1}) }\\
&  = \int_{\mathcal{X}^{r+3}} \bigg[\prod_{k=-r}^{1} q_{\theta} (x_{k-1},
x_{k}) \frac{ \int_{B^{\epsilon}_{Y_{k}}} g_{\theta} (y \vert x_{k}) \nu( dy )
}{ \nu(B_{Y_{k}}^{\epsilon}) } \bigg] \pi(d x_{-r-1}) \mu( d x_{-r:1})
\end{align*}
we have that \eqref{thmprobconvpfeq4} now follows from
\eqref{thmprobconvpfeq1}, \eqref{thmprobconvpfeq11}, the Lebesgue
differentiation theorem, (A3) and the dominated convergence theorem.
\end{proof}

\begin{proof}
[Proof of Lemma \ref{lemABCGradApprox}]We start by showing that $l(\theta)$ is
continuously differentiable. First observe that by (A3), (A4), (A5),
\eqref{lemFiltStabeq5}, \eqref{lemFiltCondProbUpLowBoundseq1},
\eqref{lemGradFiltStabeq201}, \eqref{lemGradFiltStabeq4},
\eqref{lemGradFiltStabeq305} and the dominated convergence theorem we have
that for arbitrary $x \in\mathcal{X}$
\begin{equation}
\label{lemABCGradApproxeq3}\begin{gathered} \lim_{n \to \infty} \log p_{\theta} (Y_{1} \vert Y_{-n:0} ; X_{-n} = x) = \log p_{\theta} (Y_{1} \vert Y_{-\infty:0} ) \\ \lim_{n \to \infty} \nabla_{\theta} \log p_{\theta} (Y_{1} \vert Y_{-n:0} ; X_{-n} = x) = \nabla_{\theta} \log p_{\theta} (Y_{1} \vert Y_{-\infty:0} ) \end{gathered}
\end{equation}
uniformly in $\theta\in G$ and $\ldots, Y_{0} , Y_{1}$ and that the quantities
in \eqref{lemABCGradApproxeq3} are uniformly bounded. It follows from
\eqref{lemABCGradApproxeq3} that
\[
l(\theta) = \lim_{n \to\infty} \bar{\mathbb{E}}_{\theta^{\ast}}\left[  \log
p_{\theta}(Y_{1} \vert Y_{-n:0}; X_{-n} = x)\right]
\]
and hence from \eqref{lemABCGradApproxeq3} and Lemma \ref{lemrealanalresstand}
that $\nabla_{\theta} l(\theta)$ exists, is continuous and is equal to $\lim_{n
\to\infty} \bar{\mathbb{E}}_{\theta^{\ast}}\left[  \nabla_{\theta}\log p_{\theta}(Y_{1}
\vert Y_{-n:0}; X_{-n} = x)\right]  $. Since $g_{\theta}^{\epsilon}(y|x)$
defined in (\ref{EqnPertCondLaw}) satisfies all the conditions laid out in
(A3)-(A5), the same conclusion applies to $l^{\epsilon}(\theta)$. To prove the
corresponding results for $\nabla_{\theta}^{2} l(\theta)$ observe that by the Fisher
information identity
\begin{align*}
\lefteqn{ \nabla_{\theta}^{2} \bar{\mathbb{E}}_{\theta^{\ast}}\left[  \log p_{\theta
}(Y_{1} \vert Y_{-n:0}; X_{-n} = x)\right]  }\\
&  = \bar{\mathbb{E}}_{\theta^{\ast}}\left[  \nabla_{\theta}\log p_{\theta}(Y_{1} \vert
Y_{-n:0}; X_{-n} = x) \nabla_{\theta}\log p_{\theta}(Y_{1} \vert Y_{-n:0}; X_{-n} =
x)^{T} \right]  .
\end{align*}
Existence and continuity of $\nabla_{\theta}^{2} l(\theta)$ then follows from
\eqref{lemABCGradApproxeq3} and Lemma \ref{lemrealanalresstand} applied to the
functions $\nabla_{\theta}\log p_{\theta}(Y_{1} \vert Y_{-n:0}; X_{-n} = x) $.
Furthermore the fact that $\nabla_{\theta}^{2} l(\theta^{\ast}) = I(\theta^{\ast})$ now
follows from \cite{doumouryd2004}. We begin the proof of
\eqref{eq:uniformGradlnApprox1} by observing that from
\eqref{lemABCGradApproxeq3} and the identity $\log p_{\theta}(Y_{1}%
,\ldots,Y_{n} \vert X_{1} = x) = \sum_{k=1}^{n} \log p_{\theta}(Y_{k} \vert
Y_{1:k-1}; X_{1} = x)$ we have that
\[
\nabla_{\theta} l(\theta)=\lim_{n\rightarrow\infty}\frac{1}{n}\bar{\mathbb{E}}%
_{\theta^{\ast}}\left[  \nabla_{\theta}\log p_{\theta}(Y_{1},\ldots,Y_{n} \vert X_{1} =
x)\right]
\]
and similarly, for any $\epsilon>0$ that
\[
\nabla_{\theta} l^{\epsilon}(\theta)=\lim_{n\rightarrow\infty}\frac{1}{n}%
\bar{\mathbb{E}}_{\theta^{\ast}}\left[  \nabla_{\theta}\log p_{\theta}^{\epsilon}%
(Y_{1},\ldots,Y_{n} \vert X_{1} = x)\right]  .
\]
Thus it is sufficient to show that there exists some positive constant $R$
such that for any sequence $Y_{1},\ldots,Y_{n}$, initial distribution $\pi
_{0}$ and $\theta\in G$,
\begin{equation}
\left\vert \nabla_{\theta}\log p_{\theta}(Y_{1},\ldots,Y_{n} \vert X_{1} =
x)-\nabla_{\theta}\log p_{\theta}^{\epsilon}(Y_{1},\ldots,Y_{n} \vert X_{1} =
x)\right\vert \leq nR\epsilon. \label{lemABCGradApproxeq2}%
\end{equation}
For all $\theta\in G$, sequences $Y_{1},\ldots,Y_{n}$ and $Y_{1}^{\epsilon
},\ldots,Y_{n}^{\epsilon}$ drawn from the original and perturbed processes
respectively and $x \in\mathcal{X}$
\begin{align}
\lefteqn{\nabla_{\theta}\log p_{\theta}\left(  Y_{1},\ldots,Y_{n} \vert X_{1} = x
\right)  -\nabla_{\theta}\log p_{\theta}^{\epsilon}\left(  Y_{1}^{\epsilon}%
,\ldots,Y_{n}^{\epsilon} \vert X_{1} = x \right)  }\nonumber\\
&  =\sum_{i=1}^{n}\left(  \nabla_{\theta}\log p_{\theta}\left(  Y_{1:i};Y_{i+1:n}%
^{\epsilon} \vert X_{1} = x \right)  -\nabla_{\theta}\log p_{\theta}\left(
Y_{1:i-1};Y_{i:n}^{\epsilon}\vert X_{1} = x \right)  \right) \nonumber\\
&  =\sum_{i=1}^{n}\left(  \nabla_{\theta}\log p_{\theta}\left(  Y_{i}|Y_{1:i-1}%
;Y_{i+1:n}^{\epsilon} ; X_{1} = x \right)  \right. \nonumber\\
&  \qquad\qquad\qquad\qquad\left.  -\nabla_{\theta}\log p_{\theta}\left(
Y_{i}^{\epsilon}|Y_{1:i-1};Y_{i+1:n}^{\epsilon} ; X_{1} = x \right)  \right)
\nonumber\\
&  =\sum_{i=1}^{n}\left(  \int\nabla_{\theta} g_{\theta}\left(  Y_{i}|x_{i}\right)
p_{\theta}\left(  x_{i}|Y_{1:i-1};Y_{i+1:n}^{\epsilon} ; X_{1} = x \right)
\mu(dx_{i})\right. \nonumber\\
&  \qquad\qquad\qquad\qquad\left.  +\int g_{\theta}\left(  Y_{i}|x_{i}\right)
\nabla_{\theta} p_{\theta}\left(  x_{i}|Y_{1:i-1};Y_{i+1:n}^{\epsilon} ; X_{1} = x
\right)  \mu(dx_{i})\right) \nonumber\\
&  \qquad\qquad\times\left(  \int g_{\theta}\left(  Y_{i}|x_{i}\right)
p_{\theta}\left(  x_{i}|Y_{1:i-1};Y_{i+1:n}^{\epsilon} ; X_{1} = x \right)
\mu(dx_{i})\right)  ^{-1}\nonumber\\
&  -\sum_{i=1}^{n}\left(  \int\nabla_{\theta} g_{\theta}^{\epsilon}\left(
Y_{i}^{\epsilon}|x_{i}\right)  p_{\theta}\left(  x_{i}|Y_{1:i-1}%
;Y_{i+1:n}^{\epsilon} ; X_{1} = x \right)  \mu(dx_{i})\right. \nonumber\\
&  \qquad\qquad\qquad\qquad\left.  +\int g_{\theta}^{\epsilon}\left(
Y_{i}^{\epsilon}|x_{i}\right)  \nabla_{\theta} p_{\theta}\left(  x_{i}|Y_{1:i-1}%
;Y_{i+1:n}^{\epsilon} ; X_{1} = x \right)  \mu(dx_{i})\right) \nonumber\\
&  \qquad\qquad\times\left(  \int g_{\theta}^{\epsilon}\left(  Y_{i}%
^{\epsilon}|x_{i}\right)  p_{\theta}\left(  x_{i}|Y_{1:i-1};Y_{i+1:\epsilon} ;
X_{1} = x \right)  \mu(dx_{i})\right)  ^{-1}. \label{lemABCGradApproxeq1}%
\end{align}
In particular (\ref{lemABCGradApproxeq1}) holds true when $Y_{1}^{\epsilon
},\ldots,Y_{n}^{\epsilon}=Y_{1},\ldots,Y_{n}$ and so
(\ref{lemABCGradApproxeq2}) follows from (A3), (A5), (A6), (A7) and
(\ref{lemGradFiltStabeq4}).
\end{proof}

\begin{proof}
[Proof of Lemma \ref{asymmissinf}]Throughout this proof we shall assume that
the density of any finite collection of random variables from $\ldots
,Y_{0},Y_{1},\ldots$ and $\ldots,Y_{0}^{\epsilon},Y_{1}^{\epsilon},\ldots$\ is
computed assuming that the initial condition of the hidden state process has
the stationarity distribution $\bar{\mathbb{P}}_{\theta^{\ast}}$. We begin by
observing that from \cite{doumouryd2004} we have that
\begin{equation}
\begin{gathered} I ( \theta^{\ast} )= \lim_{n \to \infty} \frac{1}{n} \bar{\mathbb{E}}_{\theta^{\ast}} \left[ \nabla_{\theta} \log p_{\theta^{\ast}} \left( Y_{1} , \ldots , Y_{n} \right). \nabla_{\theta} \log p_{\theta^{\ast}} \left( Y_{1} , \ldots , Y_{n} \right)^{T} \right] , \\ I^{\epsilon} ( \theta^{\ast} )= \lim_{n \to \infty} \frac{1}{n} \bar{\mathbb{E}}_{\theta^{\ast}} \left[ \nabla_{\theta} \log p^{\epsilon}_{\theta^{\ast}} \left( Y_{1}^{\epsilon} , \ldots , Y_{n}^{\epsilon} \right) . \nabla_{\theta} \log p^{\epsilon}_{\theta^{\ast}} \left( Y_{1}^{\epsilon} , \ldots , Y_{n}^{\epsilon} \right)^{T} \right] . \end{gathered} \label{lemasymissinfpfeq203}%
\end{equation}
By the Fisher identity we have that for any $1\leq k<k^{\prime}$ and
subsequences $\left\{  j_{1},\ldots,j_{l}\right\}  \subset\left\{
1,\ldots,k\right\}  $ and $\left\{  j_{1}^{\prime},\ldots,j_{l^{\prime}%
}^{\prime}\right\}  \subset\left\{  k+1,\ldots,k^{\prime}\right\}  $ that
\begin{equation}
\nabla_{\theta}\log p_{\theta}\left(  Y_{j_{1:l}};Y_{j_{1:l^{\prime}}^{\prime}%
}^{\epsilon}\right)  =\mathbb{E}\left[  \nabla_{\theta}\log p_{\theta}\left(
Y_{1:k};Y_{k+1:k^{\prime}}^{\epsilon}\right)  |Y_{j_{1:l}};Y_{j_{1:l^{\prime}%
}^{\prime}}^{\epsilon}\right]  . \label{lemasymissinfpfeq201}%
\end{equation}
Further by construction of the perturbed process one can easily show that
given any $n$ and any subset $\left\{  j_{1},\ldots,j_{l}\right\}
\subset\left\{  1,\ldots,n\right\}  $ that
\begin{equation}
\nabla_{\theta}\log p_{\theta}\left(  Y_{1},\ldots,Y_{n}\right)  =\nabla_{\theta}\log p_{\theta
}\left(  Y_{1},\ldots,Y_{n};Y_{j_{1}}^{\epsilon},\ldots,Y_{j_{l}}^{\epsilon
}\right)  . \label{lemasymissinfpfeq1001}%
\end{equation}
Using \eqref{lemasymissinfpfeq201} and \eqref{lemasymissinfpfeq1001} it
follows that for any $1\leq i<n$
\begin{align*}
\lefteqn{\bar{\mathbb{E}}_{\theta^{\ast}}\left[  \nabla_{\theta}\log p_{\theta^{\ast}%
}\left(  Y_{1:i},Y_{i+1:n}^{\epsilon}\right)  .\nabla_{\theta}\log p_{\theta^{\ast}%
}\left(  Y_{1:i-1},Y_{i:n}^{\epsilon}\right)  ^{T}\right]  }\\
&  =\bar{\mathbb{E}}_{\theta^{\ast}}\left[  \nabla_{\theta}\log p_{\theta^{\ast}%
}\left(  Y_{1:i-1},Y_{i:n}^{\epsilon}\right)  .\nabla_{\theta}\log p_{\theta^{\ast}%
}\left(  Y_{1:i-1},Y_{i:n}^{\epsilon}\right)  ^{T}\right]  ,
\end{align*}
and hence that
\begin{align}
\lefteqn{
\bar{\mathbb{E}}_{\theta^{\ast}}\left[  \left(  \nabla_{\theta}\log
p_{\theta^{\ast}}\left(  Y_{1:i},Y_{i+1:n}^{\epsilon}\right)  -\nabla_{\theta}\log
p_{\theta^{\ast}}\left(  Y_{1:i-1},Y_{i:n}^{\epsilon}\right)  \right)
\right.
}
\nonumber\\
&  \qquad\qquad\left.  \cdot\left(  \nabla_{\theta}\log p_{\theta^{\ast}}\left(
Y_{1:i},Y_{i+1:n}^{\epsilon}\right)  -\nabla_{\theta}\log p_{\theta^{\ast}}\left(
Y_{1:i-1},Y_{i:n}^{\epsilon}\right)  \right)  ^{T}\right] \nonumber\\
&  =\bar{\mathbb{E}}_{\theta^{\ast}}\left[  \nabla_{\theta}\log p_{\theta^{\ast}%
}\left(  Y_{1:i},Y_{i+1:n}^{\epsilon}\right)  .\nabla_{\theta}\log p_{\theta^{\ast}%
}\left(  Y_{1:i},Y_{i+1:n}^{\epsilon}\right)  ^{T}\right] \nonumber\\
&  \qquad\qquad-\bar{\mathbb{E}}_{\theta^{\ast}}\left[  \nabla_{\theta}\log
p_{\theta^{\ast}}^{\epsilon}\left(  Y_{1:i-1},Y_{i:n}^{\epsilon}\right)
.\nabla_{\theta}\log p_{\theta^{\ast}}^{\epsilon}\left(  Y_{1:i-1},Y_{i:n}^{\epsilon
}\right)  ^{T}\right]  . \label{lemasymissinfpfeq3}%
\end{align}
Using (\ref{lemasymissinfpfeq3}) and (\ref{lemasymissinfpfeq203}) and
stationarity we have that
\begin{align}
& I(\theta^{\ast})-I^{\epsilon}(\theta^{\ast})  \nonumber \\
&=\lim_{n\rightarrow
\infty}\frac{1}{n}\sum_{i=1}^{n}\bar{\mathbb{E}}_{\theta^{\ast}}\left[
\bigg(\nabla_{\theta}\log p_{\theta^{\ast}}\left(  Y_{0}|Y_{1-i:-1};Y_{1:n-i}%
^{\epsilon}\right)   \right. \nonumber\\
&  \qquad\qquad\qquad\qquad\qquad\qquad\qquad-\nabla_{\theta}\log p_{\theta^{\ast}%
}\left(  Y_{0}^{\epsilon}|Y_{1-i:-1};Y_{1:n-i}^{\epsilon}\right)
\bigg)\nonumber\\
&  \qquad\qquad\qquad  \cdot \bigg(\nabla_{\theta}\log p_{\theta^{\ast}}\left(  Y_{0}\vert
Y_{1-i:-1};Y_{1:n-i}^{\epsilon}  \right)  \nonumber \\
& \left.  \qquad\qquad\qquad\qquad\qquad\qquad\qquad -\nabla_{\theta}\log p_{\theta
^{\ast}}\left(  Y_{0}^{\epsilon}|Y_{1-i:-1};Y_{1:n-i}^{\epsilon}
\right)
\bigg)^{T}\right] \nonumber\\
&  =\lim_{n\rightarrow\infty}\bar{\mathbb{E}}_{\theta^{\ast}}\left[
\bigg(\nabla_{\theta}\log p_{\theta^{\ast}}\left(  Y_{0}|Y_{-n:-1};Y_{1:n}^{\epsilon
}\right)
-\nabla_{\theta}\log p_{\theta^{\ast}}\left(  Y_{0}^{\epsilon}|Y_{-n:-1}%
;Y_{1:n}^{\epsilon}\right) \bigg) \right. \nonumber\\
&\qquad\qquad \quad \, \, \left.  \cdot\bigg(\nabla_{\theta}\log p_{\theta^{\ast}%
}\left(  Y_{0} \vert Y_{-n:-1};Y_{1:n}^{\epsilon}  \right)
-\nabla_{\theta}\log p_{\theta^{\ast}}\left(
Y_{0}^{\epsilon}\vert Y_{-n:-1};Y_{1:n}^{\epsilon}\right)  \bigg)^{T}\right].\nonumber \\
&\label{lemasymmissinfeq2}%
\end{align}
where the last equality follows from assumptions (A3) and (A5) and equations
\eqref{lemFiltStabStandardeq2}, \eqref{lemGradFiltStabeq4} and
\eqref{lemGradFiltStabeq1}. In addition, using \eqref{lemGradFiltStabeq201} and the
dominated convergence theorem, we conclude from \eqref{lemasymmissinfeq2} that

\begin{align}
\lefteqn{I(\theta^{\ast})-I^{\epsilon}(\theta^{\ast})=\bar{\mathbb{E}}%
_{\theta^{\ast}}\left[  \left(  G_{\theta^{\ast};Y_{-\infty:-1};Y_{1:\infty
}^{\epsilon}}(Y_{0})-G_{\theta^{\ast};Y_{-\infty:-1};Y_{1:\infty}^{\epsilon}%
}^{\epsilon}(Y_{0}^{\epsilon})\right)  \cdot\right.  }\nonumber\\
&  \qquad\qquad\qquad\qquad\qquad\qquad\left.  \left(  G_{\theta^{\ast
};Y_{-\infty:-1};Y_{1:\infty}^{\epsilon}}(Y_{0})-G_{\theta^{\ast}%
;Y_{-\infty:-1};Y_{1:\infty}^{\epsilon}}^{\epsilon}(Y_{0}^{\epsilon})\right)
^{T}\right]  \label{lemasymmissinfeq504}%
\end{align}
where for any sequence $\ldots,Y_{-1};Y_{1}^{\epsilon},\ldots$
\begin{align}
\lefteqn{G_{\theta^{\ast};Y_{-\infty:-1};Y_{1:\infty}^{\epsilon}}(Y_{0}%
)=\lim_{n\rightarrow\infty}\nabla_{\theta}\log p_{\theta^{\ast}}\left(  Y_{0}%
|Y_{-n:-1};Y_{1:n}^{\epsilon}\right)  }\nonumber\\
&  \qquad=\left(  \int\bigg(\nabla_{\theta} g_{\theta^{\ast}}\left(  Y_{0}%
|x_{0}\right)  p_{\theta^{\ast}}\left(  x_{0} \vert Y_{-\infty : -1};Y_{1: \infty}^{\epsilon
} \right)  \right. \nonumber\\
&  \qquad\qquad\qquad\qquad\qquad\qquad\left.  +g_{\theta^{\ast}}\left(
Y_{0}|x_{0}\right)  \bar{\nabla_{\theta}}p_{\theta^{\ast}}\left(  x_{0} \vert Y_{-\infty : -1};Y_{1: \infty}^{\epsilon
} \right)  \bigg)\mu(dx_{0})\right) \nonumber\\
&  \qquad\qquad\qquad\times\left(  \int g_{\theta^{\ast}}\left(  Y_{0}%
|x_{0}\right)  p_{\theta^{\ast}}\left(  x_{0} \vert Y_{-\infty : -1};Y_{1: \infty}^{\epsilon
} \right)  \mu(dx_{0})\right)  ^{-1} \label{lemasymmissinfeq5001}%
\end{align}
and
\begin{align}
\lefteqn{G_{\theta^{\ast};Y_{-\infty:-1};Y_{1:\infty}^{\epsilon}}^{\epsilon
}(Y_{0}^{\epsilon})=\lim_{n\rightarrow\infty}\nabla_{\theta}\log p_{\theta^{\ast}%
}\left(  Y_{0}^{\epsilon}|Y_{-n:-1};Y_{1:n}^{\epsilon}\right)  }\nonumber\\
&  \qquad=\left(  \int\bigg(\nabla_{\theta} g_{\theta^{\ast}}^{\epsilon}\left(
Y_{0}^{\epsilon}|x_{0}\right)  p_{\theta^{\ast}}\left(  x_{0}|Y_{-\infty
:-1};Y_{1:\infty}^{\epsilon}\right)  \right. \nonumber\\
&  \qquad\qquad\qquad\qquad\qquad\qquad\left.  +g_{\theta^{\ast}}^{\epsilon
}\left(  Y_{0}^{\epsilon}|x_{0}\right)  \bar{\nabla_{\theta}}p_{\theta^{\ast}}\left(
x_{0}|Y_{-\infty:-1};Y_{1:\infty}^{\epsilon}\right)  \bigg)\mu(dx_{0})\right)
\nonumber\\
&  \qquad\qquad\qquad\times\left(  \int g_{\theta^{\ast}}^{\epsilon}\left(
Y_{0}^{\epsilon}|x_{0}\right)  p_{\theta^{\ast}}\left(  x_{0}|Y_{-\infty
:-1};Y_{1:\infty}^{\epsilon}\right)  \mu(dx_{0})\right)  ^{-1}.
\label{lemasymmissinfeq5002}%
\end{align}
Further by using assumptions (A3), (A5) and \eqref{lemFiltStabStandardeq2},
\eqref{lemGradFiltStabeq201}, \eqref{lemGradFiltStabeq4} and
\eqref{lemGradFiltStabeq305} we have that the conditional likelihood functions
$\log p_{\theta^{\ast}}\left(  Y_{0}|Y_{-n:-1};Y_{1:n}^{\epsilon}\right)  $
and $\log p_{\theta^{\ast}}\left(  Y_{0}^{\epsilon}|Y_{-n:-1};Y_{1:n}%
^{\epsilon}\right)  $ as well as their derivatives are bounded uniformly in
$\theta$ and $\ldots,Y_{-1};Y_{1}^{\epsilon},\ldots$ and that the derivatives
are uniformly Cauchy in $n$ whilst the conditional likelihoods themselves
converge uniformly to the quantities $\log p_{\theta^{\ast}}\left(
Y_{0}|Y_{-\infty:-1};Y_{1:\infty}^{\epsilon}\right)  $ and $\log
p_{\theta^{\ast}}\left(  Y_{0}^{\epsilon}|Y_{-\infty:-1};Y_{1:\infty
}^{\epsilon}\right)  $. Hence we can apply Lemma \ref{lemrealanalresstand} to obtain
\begin{equation}
\begin{gathered} G_{\theta^{\ast}; Y_{-\infty:-1};Y_{1:\infty}^{\epsilon} } ( Y_{0} ) = \nabla_{\theta} \log p_{\theta^{\ast}} (Y_{0} \vert Y_{-\infty:-1};Y_{1:\infty}^{\epsilon} ), \\ G^{\epsilon}_{\theta^{\ast}; Y_{-\infty:-1};Y_{1:\infty}^{\epsilon} } ( Y^{\epsilon}_{0} ) = \nabla_{\theta} \log p_{\theta^{\ast}} (Y^{\epsilon}_{0} \vert Y_{-\infty:-1};Y_{1:\infty}^{\epsilon} ). \end{gathered} \label{lemasymmissinfeq501}%
\end{equation}
It now follows from \eqref{lemasymmissinfeq504} and
\eqref{lemasymmissinfeq501} that
\begin{align}
&I(\theta^{\ast})-I^{\epsilon}(\theta^{\ast})
\nonumber\label{lemasymissinfpfeq2001}\\
&=\bar{\mathbb{E}}_{\theta^{\ast}} \bigg[  \Big(\nabla_{\theta}\log p_{\theta^{\ast}%
}(Y_{0}|Y_{-\infty:-1};Y_{1:\infty}^{\epsilon})-\nabla_{\theta}\log p_{\theta^{\ast}%
}(Y_{0}^{\epsilon}|Y_{-\infty:-1};Y_{1:\infty}^{\epsilon})\Big) \cdot
\nonumber\\
&\qquad\qquad\quad   \Big(\nabla_{\theta}\log p_{\theta^{\ast}}(Y_{0}%
|Y_{-\infty:-1};Y_{1:\infty}^{\epsilon})-\nabla_{\theta}\log p_{\theta^{\ast}}%
(Y_{0}^{\epsilon}|Y_{-\infty:-1};Y_{1:\infty}^{\epsilon})\Big)^{T} \bigg]
\nonumber\\
&=\bar{\mathbb{E}}_{\theta^{\ast}}\bigg[  \bar{\mathbb{E}}_{\theta^{\ast}%
}\bigg[  \Big(\nabla_{\theta}\log p_{\theta^{\ast}}(Y_{0}|Y_{-\infty:-1};Y_{1:\infty
}^{\epsilon})-\nabla_{\theta}\log p_{\theta^{\ast}}(Y_{0}^{\epsilon}|Y_{-\infty
:-1};Y_{1:\infty}^{\epsilon})\Big) \cdot  \nonumber\\
& \Big(\nabla_{\theta}\log p_{\theta^{\ast}}(Y_{0}|Y_{-\infty
:-1};Y_{1:\infty}^{\epsilon})-\nabla_{\theta}\log p_{\theta^{\ast}}(Y_{0}^{\epsilon
}|Y_{-\infty:-1};Y_{1:\infty}^{\epsilon})\Big)^{T}|Y_{-\infty:-1}%
;Y_{1:\infty}^{\epsilon} \bigg]  \bigg]  .\nonumber\\
&
\end{align}
Finally by applying the Fisher inequality \eqref{lemasymissinfpfeq201} and
\eqref{lemasymissinfpfeq1001} to the conditional laws $\mathbb{P}%
_{\theta^{\ast}}(Y_{0}|Y_{-\infty:-1};Y_{1:\infty}^{\epsilon})$ and
$\mathbb{P}_{\theta^{\ast}}(Y_{0}^{\epsilon}|Y_{-\infty:-1};Y_{1:\infty
}^{\epsilon})$ we get that
\begin{align}
\lefteqn{\bar{\mathbb{E}}_{\theta^{\ast}}\left[  \Big(\nabla_{\theta}\log
p_{\theta^{\ast}}(Y_{0}|Y_{-\infty:-1};Y_{1:\infty}^{\epsilon})-\nabla_{\theta}\log
p_{\theta^{\ast}}(Y_{0}^{\epsilon}|Y_{-\infty:-1};Y_{1:\infty}^{\epsilon
})\Big)\cdot\right.  }\nonumber\label{lemasymissinfpfeq2002}\\
&  \left.  \Big(\nabla_{\theta}\log p_{\theta^{\ast}}(Y_{0}|Y_{-\infty:-1}%
;Y_{1:\infty}^{\epsilon})-\nabla_{\theta}\log p_{\theta^{\ast}}(Y_{0}^{\epsilon
}|Y_{-\infty:-1};Y_{1:\infty}^{\epsilon})\Big)^{T}|Y_{-\infty:-1}%
;Y_{1:\infty}^{\epsilon}\right] \nonumber\\
&  \qquad=\bar{\mathbb{E}}_{\theta^{\ast}}\bigg[\nabla_{\theta}\log p_{\theta^{\ast}%
}(Y_{0}|Y_{-\infty:-1};Y_{1:\infty}^{\epsilon})\cdot\nonumber\\
&  \qquad\qquad\qquad\qquad\qquad\qquad\qquad\nabla_{\theta}\log p_{\theta^{\ast}%
}(Y_{0}|Y_{-\infty:-1};Y_{1:\infty}^{\epsilon})^{T}|Y_{-\infty:-1}%
;Y_{1:\infty}^{\epsilon}\bigg]\nonumber\\
&  \qquad-\bar{\mathbb{E}}_{\theta^{\ast}}\bigg[\nabla_{\theta}\log p_{\theta^{\ast}%
}(Y_{0}^{\epsilon}|Y_{-\infty:-1};Y_{1:\infty}^{\epsilon})\cdot\nonumber\\
&  \qquad\qquad\qquad\qquad\qquad\qquad\qquad\nabla_{\theta}\log p_{\theta^{\ast}%
}(Y_{0}^{\epsilon}|Y_{-\infty:-1};Y_{1:\infty}^{\epsilon})^{T}|Y_{-\infty
:-1};Y_{1:\infty}^{\epsilon}\bigg]\nonumber\\
&  \qquad:=I_{Y_{-\infty:-1};Y_{1:\infty}^{\epsilon}}^{Y_{0}:Y_{0}^{\epsilon}%
}(\theta^{\ast}).\nonumber\\
&
\end{align}
The result now follows from \eqref{lemasymissinfpfeq2001} and \eqref{lemasymissinfpfeq2002}.
\end{proof}

\begin{rem}
\label{corasymmissinf}Assume (A1)-(A5).  Then in exactly the same way as one proves Lemma \ref{asymmissinf} one may prove that
\[
I(\theta^{\ast})=I^{\epsilon}(\theta^{\ast})+\bar{\mathbb{E}}_{\theta^{\ast}%
}\left[  I_{Y_{-\infty:-1};Y_{m:\infty}^{\epsilon}}^{Y_{0:m-1}:Y_{0:m-1}%
^{\epsilon}}(\theta^{\ast})\right]
\]
where for any pair of sequences
$Y_{-\infty:-1}$ and $Y_{1:\infty}^{\epsilon}$ and any integer $m \geq 1$
\begin{align}
\lefteqn{ I_{Y_{-\infty:-1};Y_{m:\infty}^{\epsilon}}^{Y_{0:m-1}:Y_{0:m-1}^{\epsilon}%
}(\theta^{\ast}) } \nonumber\\
&  \qquad = \frac{1}{m} \bar{\mathbb{E}}_{\theta^{\ast}}\bigg[ \nabla_{\theta}\log p_{\theta
^{\ast}}(Y_{0:m-1}|Y_{-\infty:-1};Y_{m:\infty}^{\epsilon})\cdot \nonumber \\
&  \qquad\qquad\qquad\qquad\qquad  \nabla_{\theta}\log p_{\theta^{\ast}%
}(Y_{0:m-1}|Y_{-\infty:-1};Y_{m:\infty}^{\epsilon})^{T}|Y_{-\infty
:-1};Y_{m:\infty}^{\epsilon}\bigg]\nonumber\\
&  \qquad - \frac{1}{m} \bar{\mathbb{E}}_{\theta^{\ast}}\bigg[ \nabla_{\theta}\log
p_{\theta^{\ast}}(Y_{0:m-1}^{\epsilon}|Y_{-\infty:-1};Y_{m:\infty}^{\epsilon
})\cdot\nonumber\\
&  \qquad\qquad\qquad\qquad\qquad \nabla_{\theta}\log p_{\theta^{\ast}%
}(Y_{0:m-1}^{\epsilon}|Y_{-\infty:-1};Y_{m:\infty}^{\epsilon})^{T}%
|Y_{-\infty:-1};Y_{m:\infty}^{\epsilon}\bigg] . \nonumber \\
& \label{corasymmissinfeq1}
\end{align}
\end{rem}

\begin{proof}[Proof of Lemma \ref{lemABCMLEcomp}]
We begin by establishing part 1. From \eqref{lemnoiseinflosspfeq1} and
\eqref{lemasymissinfpfeq2002} we have for any $Y_{-\infty:-1}$ and
$Y_{1:\infty}^{\epsilon}$ that $I_{Y_{-\infty:-1};Y_{1:\infty}^{\epsilon}}%
^{Y_{0}:Y_{0}^{\epsilon}}(\theta^{\ast})\geq0$ from which the first assertion of part 1 immediately follows. In order to prove the second assertion of part 1 we note that it is sufficient to prove that under the assumption of connectivity we must have that $\bar{\mathbb{E}}%
_{\theta^{\ast}}\left[  I_{Y_{-\infty:-1};Y_{1:\infty}^{\epsilon}}%
^{Y_{0}:Y_{0}^{\epsilon}}(\theta^{\ast})\right]  =0$ implies $I(\theta^{\ast})=0$.  Since we have by Remark \ref{corasymmissinf} that $\bar{\mathbb{E}}%
_{\theta^{\ast}}\left[  I_{Y_{-\infty:-1};Y_{1:\infty}^{\epsilon}}%
^{Y_{0}:Y_{0}^{\epsilon}}(\theta^{\ast})\right]  = \bar{\mathbb{E}}_{\theta^{\ast
}}\left[  I_{Y_{-\infty:-1};Y_{m:\infty}^{\epsilon}}^{Y_{0:m-1}:Y_{0:m-1}%
^{\epsilon}}(\theta^{\ast})\right] $ for all $m \geq 1$ this will follow once we show that $\bar{\mathbb{E}}_{\theta^{\ast
}}\left[  I_{Y_{-\infty:-1};Y_{m:\infty}^{\epsilon}}^{Y_{0:m-1}:Y_{0:m-1}%
^{\epsilon}}(\theta^{\ast})\right] = 0$ for all $m \geq 1$ implies $I(\theta^{\ast})=0$.

Observe that by Lemmas
\ref{remmeasureconnectedness} and \ref{lemmeasureconnectedness} and the assumption that $\nu$ is connected it follows that $\left(  \nu\ast\mathcal{U}%
_{B_{0}^{\epsilon}}\right)  ^{\otimes m}$ is connected for all $\epsilon>0$
and $m \geq 1$ and thus from (A3) and \eqref{EqnPertCondLaw} that the conditional laws
$\mathbb{P}_{\theta^{\ast}}(Y_{0:m-1}|Y_{-\infty:-1};Y_{m:\infty}^{\epsilon})$
and $\mathbb{P}_{\theta^{\ast}}(Y_{0:m-1}^{\epsilon}|\ Y_{-\infty
:-1};Y_{m:\infty}^{\epsilon})$ are also connected for all $\epsilon>0$ and sequences
$Y_{-\infty:-1}$ and $Y_{m:\infty}^{\epsilon}$.  It then follows from \eqref{lemnoiseinflosspfeq1} and \eqref{corasymmissinfeq1} that for all $m \geq 1$ that $\bar{\mathbb{E}}_{\theta^{\ast
}}\left[  I_{Y_{-\infty:-1};Y_{m:\infty}^{\epsilon}}^{Y_{0:m-1}:Y_{0:m-1}%
^{\epsilon}}(\theta^{\ast})\right]  =0$ implies that
\begin{align*}
&  \bar{\mathbb{E}}_{\theta^{\ast}}\bigg[\nabla_{\theta}\log p_{\theta^{\ast}%
}(Y_{0:m-1}|Y_{-\infty:-1};Y_{m:\infty}^{\epsilon})\cdot\\
&  \qquad\qquad\qquad\qquad\qquad \nabla_{\theta}\log p_{\theta^{\ast}%
}(Y_{0:m-1}|Y_{-\infty:-1};Y_{m:\infty}^{\epsilon})^{T}|Y_{-\infty
:-1};Y_{m:\infty}^{\epsilon}\bigg]\\
&  \qquad=\bar{\mathbb{E}}_{\theta^{\ast}}\bigg[\nabla_{\theta}\log p_{\theta^{\ast}%
}(Y_{0:m-1}^{\epsilon}|Y_{-\infty:-1};Y_{m:\infty}^{\epsilon})\cdot\\
&  \qquad\qquad\qquad\qquad\qquad \nabla_{\theta}\log p_{\theta^{\ast}%
}(Y_{0:m-1}^{\epsilon}|Y_{-\infty:-1};Y_{m:\infty}^{\epsilon})^{T}%
|Y_{-\infty:-1};Y_{m:\infty}^{\epsilon}\bigg]
\end{align*}
for all $Y_{-\infty:-1}$,$Y_{m:\infty}^{\epsilon}$ a.s. and hence by Lemma
\ref{lemnoiseinfloss} that
\[
\nabla_{\theta}\log p_{\theta^{\ast}}(Y_{0:m-1}|Y_{-\infty:-1};Y_{m:\infty}^{\epsilon
})=0
\]
$p_{\theta^{\ast}}(Y_{0:m-1}|Y_{-\infty:-1};Y_{m:\infty}^{\epsilon})$ a.s. and thus by
Fisher's identity applied to the conditional probabilities that
\begin{equation} \label{extralabelforv510revisiontom1}
\nabla_{\theta}\log p_{\theta^{\ast}}(Y_{0}|Y_{-\infty:-1};Y_{m:\infty}^{\epsilon})=0
\end{equation}
$p_{\theta^{\ast}}(Y_{0}|Y_{-\infty:-1};Y_{m:\infty}^{\epsilon})$ a.s. also.  Finally observe that one can derive expressions  for the gradients of the conditional densities $p_{\theta^{\ast}}(Y_{0}|Y_{-\infty:-1};Y_{m:\infty}^{\epsilon})$ and $p_{\theta^{\ast}}(Y_{0}|Y_{-\infty:-1})$ analogous to \eqref{lemasymmissinfeq5001} and \eqref{lemasymmissinfeq501}.  It then follows from these and (A3), (A5), \eqref{extralabelforv510revisiontom3}, \eqref{lemFiltStabeq5}, \eqref{extralabelforv510revisiontom4} and \eqref{remcorGradDiffRateextensionEq} that
\begin{equation} \label{extralabelforv510revisiontom2}
\nabla_{\theta}\log p_{\theta^{\ast}}(Y_{0}|Y_{-\infty:-1}) = \lim_{m \to \infty} \nabla_{\theta}\log p_{\theta^{\ast}}(Y_{0}|Y_{-\infty:-1};Y_{m:\infty}^{\epsilon})
\end{equation}
a.s..  It now follows from \eqref{extralabelforv510revisiontom1} and \eqref{extralabelforv510revisiontom2} that if $\bar{\mathbb{E}}_{\theta^{\ast
}}\left[  I_{Y_{-\infty:-1};Y_{m:\infty}^{\epsilon}}^{Y_{0:m-1}:Y_{0:m-1}%
^{\epsilon}}(\theta^{\ast})\right]  =0$ for all $m \geq 1$ then $\nabla_{\theta}\log p_{\theta^{\ast}}(Y_{0}|Y_{-\infty:-1})=0$ a.s. and hence that $I(\theta^{\ast})=0$.

Next we establish part 2. Recall that given a positive semi-definite matrix
$M\in\mathbb{R}^{d\times d}$ and a sequence of $\mathbb{R}^{d\times d}$ valued
positive semi-definite matrices $\left\{  M_{n}\right\}  _{n\geq1}$ that
$M_{n}\rightarrow M$ if and only if $v^{T}M_{n}v\rightarrow v^{T}Mv$ for all
$v\in\mathbb{R}^{d}$. Thus in order to prove part 2 it is sufficient to show
that for every sequence $\epsilon_{n}\nearrow\infty$ and every $v\in
\mathbb{R}^{m}$ that $v^{T}\bar{\mathbb{E}}_{\theta^{\ast}}\left[
I_{Y_{-\infty:-1};Y_{1:\infty}^{\epsilon}}^{Y_{0}:Y_{0}^{\epsilon}}%
(\theta^{\ast})\right]  v\rightarrow v^{T}I\left(  \theta^{\ast}\right)  v$.
By definition and stationarity
\begin{align}
\lefteqn{v^{T}\bar{\mathbb{E}}_{\theta^{\ast}}\left[  I_{Y_{-\infty
:-1};Y_{1:\infty}^{\epsilon}}^{Y_{0}:Y_{0}^{\epsilon}}(\theta^{\ast})\right]
v}\nonumber\label{lemABCMLEcomppfeq501}\\
&  =\mathbb{E}\left[  \left(  v^{T}.\left(  \nabla_{\theta}\log p_{\theta^{\ast}%
}\left(  Y_{0}|Y_{-\infty:-1};Y_{1:\infty}^{\epsilon}\right)  -\nabla_{\theta}\log
p_{\theta^{\ast}}\left(  Y_{0}^{\epsilon}|Y_{-\infty:-1};Y_{1:\infty
}^{\epsilon}\right)  \right)  \right)  ^{2}\right] \nonumber\\
&
\end{align}
and
\begin{equation}
v^{T}Iv=\mathbb{E}\left[  \left(  v^{T}.\nabla_{\theta}\log p_{\theta^{\ast}}\left(
Y_{0}|Y_{-\infty:-1}\right)  \right)  ^{2}\right]  .
\label{lemABCMLEcomppfeq502}%
\end{equation}
Further by assumptions (A3) and (A5) and (\ref{lemGradFiltStabeq4}) we have
that there exists some $K<\infty$ such that
\begin{equation}
\begin{gathered} v^{T} . \nabla_{\theta}\log p_{\theta^{\ast}} \left( Y_{0} \vert Y_{-\infty:-1} \right) , \, v^{T} . \nabla_{\theta}\log p_{\theta^{\ast}} \left( Y_{0} \vert Y_{-\infty:-1}; Y^{\epsilon}_{1:\infty} \right) , \\ v^{T} . \nabla_{\theta}\log p_{\theta^{\ast}} \left( Y^{\epsilon}_{0} \vert Y_{-\infty:-1}; Y^{\epsilon}_{1:\infty} \right) \leq K \label{lemABCMLEcomppfeq503} \end{gathered}
\end{equation}
a.s. for all $\epsilon>0$ and sequences $\ldots,Y_{-1};Y_{1}^{\epsilon}%
,\ldots$. The proof of the second part of the result will follow from
\eqref{lemABCMLEcomppfeq501}, \eqref{lemABCMLEcomppfeq502} and
\eqref{lemABCMLEcomppfeq503} once we show that for any $\ldots,Y_{-1}%
;Y_{1}^{\epsilon},\ldots$
\begin{equation}
\begin{gathered} \nabla_{\theta} \log p_{\theta^{\ast}} \left( Y^{\epsilon}_{0} \vert Y_{-\infty:-1}; Y^{\epsilon}_{1:\infty} \right) \to 0 , \\ \nabla_{\theta} \log p_{\theta^{\ast}} \left( Y_{0} \vert Y_{-\infty:-1}; Y^{\epsilon}_{1:\infty} \right) \to \nabla_{\theta} \log p_{\theta^{\ast}} \left( Y_{0} \vert Y_{-\infty:-1} \right) \end{gathered} \label{lemABCMLEcomppfeq504}%
\end{equation}
as $\epsilon\rightarrow\infty$ in $p_{\theta^{\ast}}\left(  Y_{0}%
|Y_{-\infty:-1};Y_{1:\infty}^{\epsilon}\right)  $ probability. The first part
of \eqref{lemABCMLEcomppfeq504} is a straightforward consequence of applying
Lemma \ref{lemlargenoiseinfasym} to the conditional laws $\mathbb{P}%
_{\theta^{\ast}}\left(  Y_{0}|Y_{-\infty:-1};Y_{1:\infty}^{\epsilon}\right)
$. To establish the second part of \eqref{lemABCMLEcomppfeq504} observe that
from \eqref{lemFiltStabStandardeq2}, \eqref{lemGradFiltStabeq201},
\eqref{lemGradFiltStabeq305}, \eqref{lemasymmissinfeq501} and
\eqref{lemasymmissinfeq5001} we have that there exists some $C<\infty$ such
that for all $\ldots,Y_{-1};Y_{1}^{\epsilon},\ldots$
\begin{equation}
\begin{gathered} \left\vert \nabla_{\theta}\log p_{\theta^{\ast}} \left( y \vert Y_{-\infty:-1}; Y^{\epsilon}_{1:\infty} \right) - \nabla_{\theta}\log p_{\theta^{\ast}} \left( y \vert Y_{-n:-1}; Y^{\epsilon}_{1:n} \right) \right\vert \leq C \rho^{n / 2} \\ \left\vert \nabla_{\theta}\log p_{\theta^{\ast}} \left( y \vert Y_{-\infty:-1}\right) - \nabla_{\theta}\log p_{\theta^{\ast}} \left( y \vert Y_{-n:-1} \right) \right\vert \leq C \rho^{n / 2} \end{gathered} \label{lemABCMLEcomppfeq601}%
\end{equation}
for all $n\geq1$, $\epsilon>0$ and $y\in\mathcal{Y}$. It then follows from
(A3), (A5) and \eqref{lemGradFiltStabeq4}, the representation of the score
functions $\nabla_{\theta}\log p_{\theta^{\ast}}\left(  \cdot|\cdots\right)  $ given in
\eqref{lemABCGradApproxeq1}, the representation of integrals w.r.t.~the filter
gradients given in \eqref{lemGradFiltStabeq5}-\eqref{lemGradFiltStabeq8} and
Lemma \ref{lemcondprobsatbasepsilongotozero} that
\begin{equation}
\nabla_{\theta}\log p_{\theta^{\ast}}\left(  Y_{0}|Y_{-n:-1};Y_{1:k}^{\epsilon}\right)
\rightarrow\nabla_{\theta}\log p_{\theta^{\ast}}\left(  Y_{0}|Y_{-n:-1}\right)
\label{lemABCMLEcomppfeq602}%
\end{equation}
in probability as $\epsilon\nearrow\infty$. One can then conclude that the
second part of \eqref{lemABCMLEcomppfeq504} holds via
\eqref{lemABCMLEcomppfeq601} and \eqref{lemABCMLEcomppfeq602}. In order to
complete the proof of the lemma recall the two random variables $G_{\theta
^{\ast};Y_{-\infty:-1};Y_{1:\infty}^{\epsilon}}(Y_{0})$ and $G_{\theta^{\ast
};Y_{-\infty:-1};Y_{1:\infty}^{\epsilon}}^{\epsilon}(Y_{0}^{\epsilon})$
defined in \eqref{lemasymmissinfeq5001} and \eqref{lemasymmissinfeq5002}. It
follows from (A3), (A5) and the Lebesgue differentiation theorem that
$G_{\theta^{\ast};Y_{-\infty:-1};Y_{1:\infty}^{\epsilon}}(Y_{0})\rightarrow
G_{\theta^{\ast};Y_{-\infty:-1};Y_{1:\infty}^{\epsilon}}^{\epsilon}%
(Y_{0}^{\epsilon})$ a.s. as $\epsilon\searrow0$. Since it follows from the
proof of \eqref{lemasymmissinfeq501} that there exists some $K<\infty$ such
that for all $\epsilon>0$ the functions $G_{\theta^{\ast};Y_{-\infty
:-1};Y_{1:\infty}^{\epsilon}}(Y_{0})$ and $G_{\theta^{\ast};Y_{-\infty
:-1};Y_{1:\infty}^{\epsilon}}^{\epsilon}(Y_{0}^{\epsilon})$ are bounded above
by $K$ for all $\ldots,Y_{-1},Y_{0};Y_{1}^{\epsilon},\ldots$ we have that part
3 follows from \eqref{lemasymissinfpfeq2002} and a simple application of the
dominated convergence theorem. Finally, part 4 is a trivial consequence of
\eqref{lemasymmissinfeq504}, \eqref{lemasymmissinfeq5001} and
\eqref{lemasymmissinfeq5002} and assumptions (A3), (A6) and (A7).
\end{proof}

\bigskip

\bibliographystyle{imsart-nameyear}
\bibliography{references}

\end{document}